\documentclass{article}
\usepackage[toc,page]{appendix}
\usepackage[utf8]{inputenc}
\usepackage[english]{babel}
\DeclareMathAlphabet{\mathscr}{OT1}{pzc}{m}{it} 
\usepackage{amsmath}
\usepackage{bbold}
\usepackage{dsfont} 
\usepackage{amssymb}
\usepackage[T1]{fontenc}
\usepackage{mathtools}   % loads »amsmath«
\usepackage{amsfonts} 
\usepackage{bbm}
\usepackage{stmaryrd}
\usepackage{mathrsfs}  
\usepackage{graphicx}
\usepackage{bigints}
\usepackage{relsize}
\usepackage{enumitem}
\usepackage[colorinlistoftodos]{todonotes}
\usepackage{amssymb,amsmath,amsthm}
\numberwithin{equation}{section}
\usepackage{authblk}
\usepackage{pdfpages}
\newtheorem{theorem}{Theorem}[section]
\newtheorem{notation}[theorem]{Notation}
\newtheorem{lemma}[theorem]{Lemma}
\newtheorem{proposition}[theorem]{Proposition}
\newtheorem{corollary}[theorem]{Corollary}
\newtheorem{definition}[theorem]{Definition}
\newtheorem{hypothesis}[theorem]{Hypothesis}
\newtheorem{remark}[theorem]{Remark}
\newtheorem{example}[theorem]{Example}
\newenvironment{prooff}[1]{\begin{trivlist}
\item {\it \bf Proof}\quad} {\qed\end{trivlist}}
\usepackage{amsmath,amssymb}
\makeatletter
\newsavebox\myboxA
\newsavebox\myboxB
\newlength\mylenA
\newcommand*\xoverline[2][0.75]{%
    \sbox{\myboxA}{$\m@th#2$}%
    \setbox\myboxB\null% Phantom box
    \ht\myboxB=\ht\myboxA%
    \dp\myboxB=\dp\myboxA%
    \wd\myboxB=#1\wd\myboxA% Scale phantom
    \sbox\myboxB{$\m@th\overline{\copy\myboxB}$}%  Overlined phantom
    \setlength\mylenA{\the\wd\myboxA}%   calc width diff
    \addtolength\mylenA{-\the\wd\myboxB}%
    \ifdim\wd\myboxB<\wd\myboxA%
       \rlap{\hskip 0.5\mylenA\usebox\myboxB}{\usebox\myboxA}%
    \else
        \hskip -0.5\mylenA\rlap{\usebox\myboxA}{\hskip 0.5\mylenA\usebox\myboxB}%
    \fi}
\makeatother

\title{Path-dependent Martingale Problems and Additive Functionals}

\author{
Adrien BARRASSO \thanks{ENSTA ParisTech, Unit\'e de Math\'ematiques
 appliqu\'ees, 828, boulevard des Mar\'echaux, F-91120 Palaiseau, France 
 and Ecole Polytechnique,  F-91128 Palaiseau, France.
E-mail: \sf adrien.barrasso@ensta-paristech.fr \\}
\qquad\quad
Francesco RUSSO\thanks{ENSTA ParisTech, Unit\'e de Math\'ematiques appliqu\'ees, 828, boulevard des Mar\'echaux, F-91120 Palaiseau, France. E-mail: \sf francesco.russo@ensta-paristech.fr}}

\date{April 19th 2018}

\begin{document}
\maketitle
%\tableofcontents

{\bf Abstract.} The paper 
introduces and investigates the natural extension to the path-dependent setup
of the  usual concept of canonical Markov class introduced by Dynkin
and which is at the basis of the theory of Markov processes. 
That extension, indexed by starting paths rather than starting points 
will be called path-dependent canonical class. 
Associated with this is the generalization  of the notions of semi-group and
of additive functionals  to
the path-dependent framework.
A typical example of such family is constituted
by the laws $(\mathbbm{P}^{s,\eta})_{(s,\eta)\in\mathbbm{R}_+\times \Omega} $,
where for fixed time $s$ and fixed path $\eta$ defined on $[0,s]$,
$\mathbbm{P}^{s,\eta}$
being the (unique) solution of a 
path-dependent martingale problem or more specifically a
 weak solution of a path-dependent 
SDE with jumps, 
with initial path $\eta$.
In a companion paper we apply those results to study path-dependent
analysis problems associated with BSDEs.

\bigskip
{\bf MSC 2010} Classification. 
60H30; 60H10; 35S05; 60J35;  
60J75.


%\newpage
\section{Introduction}
%Philosophie générale

In this paper we extend some aspects of the theory of Markov processes to the (non-Markovian) path-dependent case.
The crucial object of Markov canonical class introduced by Dynkin is replaced 
with the one of
\textit{path-dependent canonical class}. The associated notion of Markov  semigroup is 
extended to the notion of \textit{path-dependent system of projectors}. 
The classical Markovian concept of (Martingale) Additive Functional is generalized to the one of
\textit{path-dependent (Martingale) Additive Functional}.
%We focus on the notion of \textbf{path-dependent canonical class} which permits to adapt some of the theory of Markov processes to the path-dependent setup. 
%We will in particular introduce the notion of \textbf{path-dependent system of projectors} which plays the role of the semi-group, as well as the notion of \textbf{path-dependent (Martingale) Additive Functional} (in short path-dependent (M)AF) and extend
% some results of the Markovian theory to the path-dependent setup.
%Indeed path-dependent canonical class extend the notion of
%Markov class and a path-dependent system of projectors 
%generalize the notion of (time-dependent semigroup), 
%see Remark \ref{RMarkov}. 
We then study some general path-dependent martingale problems with applications to weak solutions of path-dependent SDEs
 (possibly) with jumps and show that, under well-posedness, the solution of the martingale problem provides a path-dependent
 canonical class.
 The companion paper \cite{paperPathDep} will exploit these results to extend the links between BSDEs and (possibly Integro)
 PDEs obtained in \cite{paper3}, to a path-dependent framework.
%Histoire

The theory of Additive Functionals associated to a Markov process was initiated during the early '60s, see the historical papers
\cite{dynkin1959foundations}, \cite{MeyerPhdAF}, \cite{BlumGetAF} and see \cite{dellmeyerD} for a complete theory in the homogeneous setup. The strong links between martingale 
problems and Markov processes were first observed for the study of weak solutions of SDEs in \cite{stroock}, and more generally in \cite{EthierKurz} or \cite{jacod79} for example.
Weak solutions of path-dependent SDEs possibly with jumps were studied in \cite{jacod79}, where the author shows their equivalence to some path-dependent martingale problems and
 proves existence and uniqueness of a solution under Lipschitz conditions.  More recent results concerning path-dependent martingale problems may be found in 
\cite{BionNadal}. However, at our 
knowledge, the structure of the set of solutions for different starting paths was not yet
 studied.
%Philosophie 

The setup of this paper is the canonical space $(\Omega,\mathcal{F})$ where $\Omega:=\mathbbm{D}(\mathbbm{R}_+,E)$ is the Skorokhod space of cadlag functions from $\mathbbm{R}_+$ into a Polish space $E$ and $\mathcal{F}$ is its Borel $\sigma$-field. 
$X = (X_t)_{t\in\mathbbm{R}_+}$ denotes the canonical process and 
the initial filtration $\mathbbm{F}^o$ is defined by $\mathcal{F}^o_t:=\sigma(X_r|r\in[0,t])$ for all $t\geq 0$. 

A path-dependent canonical class will be a set of probability measures \\ $(\mathbbm{P}^{s,\eta})_{(s,\eta)\in\mathbbm{R}_+\times \Omega}$ defined 
on the canonical space and such that, for some fixed $(s,\eta)$, $\mathbbm{P}^{s,\eta}$ models
a forward (path-dependent) dynamics in law,
 with imposed initial path $\eta$ on the time interval $[0,s]$. As already mentioned, it constitutes the natural  adaptation to the path-dependent world of the notion of 
canonical Markov class $(\mathbbm{P}^{s,x})_{(s,x)\in\mathbbm{R}_+ \times E}$, where in general,  $\mathbbm{P}^{s,x}$ models the law of some Markov
stochastic process,
 with imposed value $x$ at time $s$.
 $\mathbbm{F}^{s,\eta}$ is the augmented initial filtration fulfilling the usual 
conditions. 

In substitution of a Markov semigroup associated with a Markov canonical class, we introduce a path-dependent system of projectors denoted $(P_s)_{s\in\mathbbm{R}_+}$
and a one-to-one connection between them and  path-dependent canonical classes.
% and what we call   path-dependent system of projectors denoted $(P_s)_{s\in\mathbbm{R}_+}$. 
Each projector $P_s$ acts on the space of bounded random variables.
%, and they adapt to the path-dependent setup the notion of semi-group of a Markov process.
 This brings us to introduce the notion of \textbf{weak generator} $(\mathcal{D}(A), A)$ of $(P_s)_{s\in\mathbbm{R}_+}$ which will permit us in the companion paper \cite{paperPathDep} to define \textit{mild} type solutions of path-dependent PDEs of the form 
\begin{equation}\label{PDE}
	\left\{
	\begin{array}{l}
	D\Phi + \frac{1}{2}Tr(\sigma\sigma^{\intercal} \nabla^2\Phi)  + \beta\nabla  \Phi + f(\cdot,\cdot,\Phi,\sigma\sigma^{\intercal} \nabla \Phi )=0\text{ on }[0,T]\times\Omega\\
	\Phi_T=\xi\text{ on }\Omega,
	\end{array}\right.
	\end{equation}
where $D$ is the horizontal derivative and $\nabla$ the vertical gradient in the sense of \cite{dupire,contfournie} and $\beta,\sigma$ are progressively measurable
 path-dependent coefficients. 
 
As mentioned earlier, given a path-dependent canonical class we also introduce the notion of  path-dependent Additive Functional (resp. path-dependent square integrable Martingale Additive Functional), which is   a real-valued random-field  
$M:=(M_{t,u})_{0\leq t\leq u<+\infty}$ such that for any $(s,\eta)\in\mathbbm{R}_+\times \Omega$, there exists a real cadlag $\mathbbm{F}^{s,\eta}$-adapted process (resp. $\mathbbm{F}^{s,\eta}$-square integrable martingale) $M^{s,\eta}$ called the \textbf{cadlag version of $M$ under }$\mathbbm{P}^{s,\eta}$,  and verifying for all $s\leq t\leq u$ that $M_{t,u} = M^{s,\eta}_u-M^{s,\eta}_t$ $\mathbbm{P}^{s,\eta}$ a.s.
Under some reasonable measurability assumptions on the path-dependent canonical class, we extend to our path-dependent setup some classical results of Markov processes theory concerning the quadratic covariation and the angular bracket of square integrable MAFs.
As in the Markovian set-up, examples of  path-dependent canonical classes
arise from solutions of a (this time path-dependent) martingale problem
as we explain below.
Let $\chi$ be a set of cadlag processes adapted to the initial filtration $\mathbbm{F}^o$. 
%defined by $\mathcal{F}^o_t:=\sigma(X_r|r\in[0,t])$ for all $t\geq 0$. 
For some given  $(s,\eta)\in\mathbbm{R}_+\times\Omega$, we say that a probability measure
$\mathbbm{P}^{s,\eta}$ on $(\Omega,\mathcal{F})$ \textbf{solves the martingale problem with respect to $\chi$ starting in $(s,\eta)$} if 
\begin{itemize}
		\item $\mathbbm{P}^{s,\eta}(\omega^s=\eta^s)=1$;
		\item all elements 
 of $\chi$ are on $[s,+\infty[$ $(\mathbbm{P}^{s,\eta},\mathbbm{F}^o)$-martingales. 
\end{itemize}
We show that merely under some well-posedness assumptions,
 the set of solutions
 for varying starting times and paths $(\mathbbm{P}^{s,\eta})_{(s,\eta)\in\mathbbm{R}_+\times \Omega}$ defines a path-dependent canonical class. This in particularly holds for weak solutions of path-dependent SDEs possibly with jumps.
% We give some condition on the coefficients under which $(s,\eta)\longmapsto \mathbbm{P}^{s,\eta}$ is continuous for the topology of weak convergence of measures.

The paper is organized as follows.
In Section \ref{A1}, we introduce the notion of path-dependent canonical class in Definition \ref{DefCondSyst} and of path-dependent system of projectors in Definition \ref{DefCondOp} and prove a one-to-one correspondence between those two concepts in Corollary \ref{EqProbaOp}.
In Section \ref{SectionMAF}, we introduce the notion of path-dependent Additive Functional, in short AF (resp. Martingale Additive Functional, in short MAF). We state in Proposition \ref{VarQuadAF} and Corollary \ref{AFbracket} that for a given  square integrable path-dependent MAF $(M_{t,u})_{(t,u)\in\Delta}$,  we can exhibit
two non-decreasing path-dependent AFs with ${\mathcal L}^1$-terminal value, denoted respectively by  $([M]_{t,u})_{(t,u)\in\Delta}$ 
and    $(\langle M \rangle_{t,u})_{(t,u)\in\Delta}$,
which will play respectively 
the role of a quadratic variation and an angular bracket of it.
Then in Corollary \ref{BracketMAFnew}, we state 
that the Radon-Nikodym derivative of the mentioned angular bracket of 
a square integrable path-dependent MAF with respect to a reference function
 $V$, is 
a progressively measurable process which does not depend on the probability. 
In Section \ref{SMP}, we introduce what we mean by path-dependent martingale
 problem with respect to a set of processes $\chi$,
to  a time $s$ and a starting path $\eta$,
 see Definition \ref{D44}.
Suppose that  $\chi$ is a countable set of  cadlag $\mathbbm{F}^o$-adapted processes which are uniformly bounded on each interval $[0,T]$;
in Proposition \ref{MPimpliesCond}, we state that, whenever
 the  martingale problem with respect to $\chi$
is well-posed, 
then the solution  $(\mathbbm{P}^{s,\eta})_{(s,\eta)\in\mathbbm{R}_+\times \Omega}$ defines a path-dependent canonical class. 
In Subsection \ref{S3_3}, Definition \ref{WeakGen} introduces the notion of weak generator of a path-dependent system of projectors, and Definition \ref{MPop} that of martingale problem associated to a path-dependent operator $(D(A),A)$.
Suppose now that for any $(s,\eta)$
the martingale problem associated with $(D(A),A)$ is well-posed, 
and 
let $(P_s)_{s\in\mathbbm{R}_+}$ be
the  system of projectors associated to the canonical class
constituted by the solutions $(\mathbbm{P}^{s,\eta})_{(s,\eta)\in\mathbbm{R}_+\times \Omega}$. 
%$(P_s)_{s\in\mathbbm{R}_+}$ is the system of projectors associated to that canonical class, then
Then $(D(A),A)$ is a weak generator of $(P_s)_{s\in\mathbbm{R}_+}$, and $(P_s)_{s\in\mathbbm{R}_+}$ is the unique system of projectors such that this holds. 
In other words, $(P_s)_{s\in\mathbbm{R}_+}$ can be analytically associated to $(D(A),A)$ without ambiguity.
Finally, in Section \ref{SDE}, we consider path-dependent SDEs with jumps, whose coefficients are denoted by $\beta,\sigma,w$. 
If for any couple $(s,\eta)$, the SDE has a unique
 weak solution, then
Theorem \ref{SDEcond} ensures that 
%if for any starting time $s$ and starting path $\eta$, the SDE has a unique weak solution, then 
the set of solutions $(\mathbbm{P}^{s,\eta})_{(s,\eta)\in\mathbbm{R}_+\times \Omega}$ defines a path-dependent canonical class. Under the additional assumptions that $\beta,\sigma,w$ are bounded and continuous in $\omega$ for fixed other variables, then Proposition \ref{UniqueMPcontImpliesProg} states that $(s,\eta)\longmapsto \mathbbm{P}^{s,\eta}$ is continuous for the topology of weak convergence. 

\section{Preliminaries}

In the whole paper we will use the following notions, notations and vocabulary.

A topological space $E$ will always be considered as a measurable space with its Borel $\sigma$-field which shall be denoted $\mathcal{B}(E)$ and if $S$ 
is another topological space equipped
with its Borel $\sigma$-field, $\mathcal{B}(E,S)$  will denote the set of Borel functions from $E$ to $S$.
For some fixed $d\in\mathbbm{N}^*$, $\mathcal{C}^{\infty}_c(\mathbbm{R}^d)$ will denote the set of smooth functions with compact support. For fixed $d,k\in\mathbbm{N}^*$, $\mathcal{C}^{k}(\mathbbm{R}^d)$, (resp. $\mathcal{C}^{k}_b(\mathbbm{R}^d)$) will denote the set of functions $k$ times differentiable with continuous (resp. bounded continuous) derivatives.

Let $(\Omega,\mathcal{F})$, $(E,\mathcal{E})$ be two measurable spaces. A measurable mapping from $(\Omega,\mathcal{F})$ to $(E,\mathcal{E})$ shall often be called a \textbf{random variable} (with values in $E$), or in short r.v. If $\mathbbm{T}$ is indices set, a family  $(X_t)_{t\in \mathbbm{T}}$
of r.v. with values in $E$,  will be called a \textbf{random field} (indexed by $\mathbbm{T}$ with values in $E$). In the particular case when 
 $\mathbbm{T}$ is a subinterval of $\mathbbm{R}_+$, $(X_t)_{t\in \mathbbm{T}}$ will be called a \textbf{stochastic process} (indexed by $\mathbbm{T}$ with values in $E$). 
 If the mapping 
$\begin{array}{rcl}
(t,\omega)&\longmapsto&X_t(\omega)\\
(\mathbbm{T}\times\Omega,\mathcal{B}(\mathbbm{T})\otimes\mathcal{F})&
\longrightarrow&(E,\mathcal{E})
\end{array}$
is measurable, then the process (or random field) $(X_t)_{t\in \mathbbm{T}}$
 will be said to be \textbf{measurable} (indexed by $\mathbbm{T}$ with values in $E$).

On a fixed probability space $\left(\Omega,\mathcal{F},\mathbbm{P}\right)$, for any $p \ge 1$, $\mathcal{L}^p$ will denote the set of real-valued random variables with finite $p$-th moment.
Two random fields (or stochastic processes) $(X_t)_{t\in \mathbbm{T}}$, $(Y_t)_{t\in \mathbbm{T}}$ indexed by the same set and with values in the same space will be said to be \textbf{modifications (or versions)
 of each other} if for every $t\in\mathbbm{T}$, $\mathbbm{P}(X_t=Y_t)=1$.
%A measurable space equipped with a right-continuous filtration $\left(\Omega,\mathcal{F},(\mathcal{F}_t)_{t\in\mathbbm{T}}\right)$ (where $\mathbbm{T}$ is equal to $\mathbbm{R}_+$ or to $[0,T]$ for some $T\in\mathbbm{R}_+^*$) will be called a \textbf{filtered space}. 
%\\
A filtered probability space $\left(\Omega,\mathcal{F},\mathbbm{F}:=(\mathcal{F}_t)_{t\in\mathbbm{R}_+},\mathbbm{P}\right)$  will be called called  \textbf{stochastic basis} and will be said to \textbf{fulfill the usual conditions} if the filtration is right-continuous, if the probability space is complete and if $\mathcal{F}_0$ contains all the $\mathbbm{P}$-negligible sets.
%Concerning spaces of stochastic processes, in a fixed
Let us fix a stochastic basis $\left(\Omega,\mathcal{F},\mathbbm{F},\mathbbm{P}\right)$.
If $Y = (Y_t)_{t\in\mathbbm{R}_+}$ is a stochastic process and 
 $\tau$ is  a stopping time, we denote  $Y^{\tau}$ the process $t\mapsto Y_{t\wedge\tau}$ which we call \textbf{stopped process} (by $\tau$).
  If $\mathcal{C}$ is a set of processes, 
%we define its \textbf{localized class} $\mathcal{C}_{loc}$ as the set of processes 
we will say that $Y$ is locally in $\mathcal{C}$ (resp. locally verifies some property) if there exist an a.s. increasing sequence of stopping times  $(\tau_n)_{n\geq 0}$ tending a.s. to infinity such that for every $n$, the stopped process $Y^{\tau_n}$ belongs to $\mathcal{C}$ (resp. verifies this property).

Given two martingales $M,N$, we denote by $[M]$ (resp. $[M,N]$) the \textbf{quadratic variation} of $M$ (resp. \textbf{covariation} of $M,N$). If $M,N$ are locally square integrable martingales, $\langle M,N\rangle$ (or simply $\langle M\rangle$ if $M=N$) will denote their (predictable) \textbf{angular bracket}.
%$\mathcal{H}_0^2$ will be equipped with scalar product defined by $(M,N)_{\mathcal{H}^2}:=\mathbbm{E}[M_{\infty}N_{\infty}] 
%=\mathbbm{E}[\langle M, N\rangle_{\infty}] $ which makes it a Hilbert space.
Two locally square integrable martingales vanishing at zero $M,N$  will be said to be \textbf{strongly orthogonal} if  $\langle M,N\rangle=0$.
\\
If $A$ is an adapted process with bounded variation then $Var(A)$ (resp. 
$Pos(A)$,  $Neg(A)$) will denote its total variation (resp. positive variation,
  negative variation), see Proposition 3.1, chap. 1 in \cite{jacod}.
In particular for almost all $\omega \in \Omega$, 
$t\mapsto Var_t(A(\omega))$ is the total variation function of
the function $t\mapsto A_t(\omega)$.

\section{Path-dependent canonical classes}\label{A1}

We will introduce here an abstract context  which is relevant for the study of path-dependent stochastic equations. The definitions and results which will be presented here are inspired from the theory of Markov processes and  of additive functionals which one can find for example in \cite{dellmeyerD}.

The first definition refers to the  canonical space that one can find in \cite{jacod79}, see paragraph 12.63.
\begin{notation}\label{canonicalspace}
In the whole section  $E$ will be a fixed  Polish  space, i.e. 
a separable complete metrizable topological space, that we call 
the \textbf{state space}.

$\Omega$ will denote  $\mathbbm{D}(\mathbbm{R}_+,E)$ the  space of functions from $\mathbbm{R}_+$ to $E$ being right-continuous  with left limits (e.g. cadlag).
% or $\mathcal{C}(\mathbbm{R}_+,E)$.
For every $t\in\mathbbm{R}_+$ we denote the coordinate mapping $X_t:\omega\mapsto\omega(t)$ and we define on $\Omega$ the $\sigma$-field  $\mathcal{F}:=\sigma(X_r|r\in\mathbbm{R}_+)$. 
On the measurable space $(\Omega,\mathcal{F})$, we introduce \textbf{initial filtration}  $\mathbbm{F}^o:=(\mathcal{F}^o_t)_{t\in\mathbbm{R}_+}$, where $\mathcal{F}^o_t:=\sigma(X_r|r\in[0,t])$, and the (right-continuous) \textbf{canonical filtration} $\mathbbm{F}:=(\mathcal{F}_t)_{t\in\mathbbm{R}_+}$, where $\mathcal{F}_t:=\underset{s>t}{\bigcap}\mathcal{F}^o_s$. 
$\left(\Omega,\mathcal{F},\mathbbm{F}\right)$ will be called the \textbf{canonical space} (associated to $E$).
On $\mathbbm{R}_+\times\Omega$, we will denote by $\mathcal{P}ro^o$ (resp.  $\mathcal{P}re^o$) the $\mathbbm{F}^o$-progressive (resp. $\mathbbm{F}^o$-predictable)  $\sigma$-field.
$\Omega$ will be equipped with the Skorokhod topology 
which is Polish since $E$ is Polish (see Theorem 5.6 in chapter 3 of
\cite{EthierKurz}), and for which the Borel $\sigma$-field is $\mathcal{F}$, 
see Proposition 7.1 in chapter 3 of \cite{EthierKurz}.  This in particular implies that $\mathcal{F}$ is separable, being the Borel $\sigma$-field of a separable metric space.

$\mathcal{P}(\Omega)$ will denote the set of probability measures on $\Omega$ and will be equipped with the topology of weak convergence of measures which also makes it a Polish space since $\Omega$ is Polish (see Theorems 1.7 and 3.1 in \cite{EthierKurz} chapter 3). It will also be equipped with the associated Borel $\sigma$-field.

\end{notation}

\begin{notation}\label{Stopped}
For any $\omega\in\Omega$ and $t\in\mathbbm{R}_+$, the path $\omega$ stopped at time $t$ $r\mapsto \omega(r\wedge t)$ will be denoted $\omega^t$.
%\\
%For any $\omega\in\Omega$ and $t\in\mathbbm{R}_+$, we denote $\Omega^t:=\left\{\omega\in\Omega:\omega=\omega^t\right\}$ the subset of paths constant after time $t$.
%\\
%We introduce the following subset of $\mathbbm{R}_+\times \Omega$. $\Lambda:=\{(s,\eta)\in\mathbbm{R}_+\times \Omega: \omega=\omega^s\}$. Meaning that $(s,\eta)$ belongs to $\Lambda$ if and only if $\eta$ is constant after $s$.
\end{notation}

\begin{remark}
In  Sections \ref{A1},\ref{SectionMAF} and Subsections \ref{S3_1}, \ref{S3_3}, all notions and results can easily be adapted to different canonical spaces 
$\Omega$: for instance,  $\mathcal{C}(\mathbbm{R}_+,E)$,
 the space of continuous functions from $\mathbbm{R}_+$ to $E$;  
 $\mathcal{C}([0,T],E)$ (resp. $\mathbbm{D}([0,T],E)$) the space of continuous (resp. cadlag) functions from $[0,T]$ to $E$, for some $T > 0$; 
fixing $x\in E$, $\mathcal{C}_x(\mathbbm{R}_+,E)$ (resp. $\mathcal{C}_x([0,T],E)$) the space of continuous functions from $\mathbbm{R}_+$ (resp. $[0,T]$) to
 $E$ starting at $x$ .
\end{remark}

\begin{definition}\label{DefCondSyst}
	A \textbf{path-dependent canonical class} will be a family 
$(\mathbbm{P}^{s,\eta})_{(s,\eta)\in\mathbbm{R}_+\times \Omega}$
 of probability measures defined on the canonical space $(\Omega,\mathcal{F})$,
which verifies the three following items.
	\begin{enumerate}
		\item For every  $(s,\eta)\in\mathbbm{R}_+\times \Omega$, $\mathbbm{P}^{s,\eta}( \omega^s=\eta^s)=1$;
		\item for every $s\in\mathbbm{R}_+$ and $F\in\mathcal{F}$, the mapping
		\\
		$\begin{array}{ccl}
		\eta&\longmapsto& \mathbbm{P}^{s,\eta}(F)\\
		\Omega&\longrightarrow&[0,1]
		\end{array}$ is $\mathcal{F}^o_s$-measurable;
		\item for every  $(s,\eta)\in\mathbbm{R}_+\times \Omega$, $t\geq s$  and $F\in\mathcal{F}$,
		\begin{equation} \label{DE13}
			\mathbbm{P}^{s,\eta}(F|\mathcal{F}^o_t)(\omega)=\mathbbm{P}^{t,\omega}(F)\text{ for }\mathbbm{P}^{s,\eta}\text{ almost all }\omega.
		\end{equation}
	\end{enumerate}
This implies in particular that for every  $(s,\eta)\in\mathbbm{R}_+\times \Omega$ and $t\geq s$, then $(\mathbbm{P}^{t,\omega})_{\omega\in\Omega}$ is a regular 
conditional expectation of $\mathbbm{P}^{s,\eta}$ by $\mathcal{F}^o_t$, see the Definition above Theorem 1.1.6 in \cite{stroock} for instance.
\\
\\
A path-dependent canonical class $(\mathbbm{P}^{s,\eta})_{(s,\eta)\in\mathbbm{R}_+\times \Omega}$ will be said to be \textbf{progressive} if for every $F\in\mathcal{F}$, the mapping 
$(t,\omega)\longmapsto \mathbbm{P}^{t,\omega}(F)$ is $\mathbbm{F}^o$-progressively measurable.
\end{definition}

In concrete examples, path-dependent canonical classes will always verify the following important hypothesis which is a reinforcement of \eqref{DE13}.
\begin{hypothesis}\label{HypClass}
	For every  $(s,\eta)\in\mathbbm{R}_+\times \Omega$, $t\geq s$  and $F\in\mathcal{F}$,
	\begin{equation} \label{DE14}
	\mathbbm{P}^{s,\eta}(F|\mathcal{F}_t)(\omega)=\mathbbm{P}^{t,\omega}(F)\text{ for }\mathbbm{P}^{s,\eta}\text{ almost all }\omega.
	\end{equation}
\end{hypothesis}

%\begin{remark}
%	Given some $0\leq s\leq t\leq u$ in $\mathbbm{R}_+$, $\eta\in\Omega$ and $A_u\in\mathcal{F}_u$, previous definition implies that 
%\begin{equation}
%	\begin{array}{rcl}
%	\mathbbm{P}^{s,\eta}(\omega\in A_u)&=&\mathbbm{E}^{s,\eta}[\mathbbm{P}^{s,\eta}(\omega'\in A_u|\mathcal{F}_t)(\omega)]\\
%	&=&\mathbbm{E}^{s,\eta}[\mathbbm{P}^{t,\omega}(\omega'\in A_u)]\\
%	&=&\int_{\omega\in\Omega}\mathbbm{P}^{t,\omega}(\omega'\in A_u)\mathbbm{P}^{s,\eta}(d\omega),
%	\end{array}
%\end{equation}
%which may be interpreted as a path-dependent Chapman-Kolmogorov equation.
%\end{remark}

\begin{remark}\label{Borel}
	By approximation through simple functions, one can easily show the
following.
Let  $Z$ be a random variable. 	
\begin{itemize}	
\item Let $s\geq 0$.
The functional $\eta \longmapsto \mathbbm{E}^{s,\eta}[Z]$ is $\mathcal{F}^o_s$-measurable and for every  $(s,\eta)\in\mathbbm{R}_+\times \Omega$, $t\geq s$,
		$\mathbbm{E}^{s,\eta}[Z|\mathcal{F}^o_t](\omega)=\mathbbm{E}^{t,\omega}[Z]\text{ for }\mathbbm{P}^{s,\eta}\text{ almost all }\omega$, provided
	previous expectations  are finite;
\item 	if the  path-dependent canonical class is progressive,
% then for any random variable $Z$, 
$(t,\omega)\longmapsto
	\mathbbm{E}^{t,\omega}[Z]$ is $\mathbbm{F}^o$-progressively measurable, provided
	previous expectations are finite.
\end{itemize}
\end{remark}

 \begin{notation}
\begin{itemize}\
\item 
 $\mathcal{B}_b(\Omega)$  (resp. $\mathcal{B}^+_b(\Omega)$)  
%(resp. $\mathcal{B}^s_b(\Omega)$)
 will denote the space of 
 measurable (resp. non-negative  measurable)
%(resp. $\mathcal{F}^o_s$-measurable)
 bounded r.v. 
\item    Let $s \ge 0$.  $\mathcal{B}^s_b(\Omega)$ will denote the space of
  $\mathcal{F}^o_s$-measurable bounded r.v.
\end{itemize}
\end{notation}

\begin{definition}\label{DefCondOp}
\begin{enumerate}\
\item A linear map $Q: \mathcal{B}_b(\Omega) \rightarrow \mathcal{B}_b(\Omega)$
is said {\bf positivity preserving monotonic} 
%$\mathcal{B}_b(\Omega)$ 
 if for every $\phi\in\mathcal{B}^+_b(\Omega)$ then 
$Q[\phi]\in\mathcal{B}^+_b(\Omega)$ and 
for every increasing converging (in the pointwise sense) sequence 
$f_n\underset{n}{\longrightarrow}f$ we have  that $Q[f_n]\underset{n}{\longrightarrow} Q[f]$ in the pointwise sense.
\item 	A family $(P_s)_{s\in\mathbbm{R}_+}$ of positivity preserving monotonic linear operators on $\mathcal{B}_b(\Omega)$ 
%(meaning that if $\phi\in\mathcal{B}^+_b(\Omega)$ then $P_s[\phi]\in\mathcal{B}^+_b(\Omega)$ for all $s$, and that for every increasing converging (in the pointwise sense) sequence $f_n\underset{n}{\longrightarrow}f$ we have  that $P_s[f_n]\underset{n}{\longrightarrow}P_s[f]$ in the pointwise sense for all $s$)
will be called a \textbf{path-dependent system of projectors}  if it verifies the three following properties.

	\begin{itemize}
		\item For all $s\in\mathbbm{R}_+$, the restriction of $P_s$ to $\mathcal{B}^s_b(\Omega)$
		% (resp. $\mathcal{C}^s_b(\Omega)$)
		 coincides with the identity;
		\item for all $s\in\mathbbm{R}_+$, $P_s$ maps $\mathcal{B}_b(\Omega)$ 
		%(resp. $\mathcal{C}_b(\Omega)$) 
		into $\mathcal{B}^s_b(\Omega)$;
		% (resp. $\mathcal{C}^s_b(\Omega)$);
		\item for all $s,t\in\mathbbm{R}_+$ with $t\geq s$, $P_s\circ P_t=P_s$.
	\end{itemize}
\end{enumerate}
\end{definition}

\begin{proposition} \label{P19}
Let $(\mathbbm{P}^{s,\eta})_{(s,\eta)\in\mathbbm{R}_+\times \Omega}$ be a  path-dependent canonical class.
For every $s\in\mathbbm{R}_+$, we define $P_s:\phi\longmapsto (\eta\mapsto\mathbbm{E}^{s,\eta}[\phi])$. Then
$(P_s)_{s\in\mathbbm{R}_+}$ defines a path-dependent system of projectors. 
%\\
%If moreover for every $(s,\eta)\in\mathbbm{R}_+\times\Omega$, $\mathbbm{P}^{s,\eta}(\omega\in\mathcal{C}^s(\mathbbm{R}_+,E))=1$ then $(P_s)_{s\in\mathbbm{R}_+}$ is supported by continuous paths.
%If moreover $(\mathbbm{P}^{s,\eta})_{(s,\eta)\in\mathbbm{R}_+\times \Omega}$ is Feller, then for any $s\in\mathbbm{R}_+$, $\mathcal{C}_b(\Omega)$ is stable by $P_s$, and if we still denote by $P_s$ the restriction of $P_s$ on $\mathcal{C}_b(\Omega)$ then $(P_s)_{s\in\mathbbm{R}_+}$ defines a Feller path-dependent system of projectors.
\end{proposition}
\begin{proof}
For every $s \ge 0$ each map $P_s$ is linear, positivity preserving and
monotonic using the usual properties of the expectation
under a given probability.
The rest follows taking into account Definitions \ref{DefCondSyst}, \ref{DefCondOp}
and Remark \ref{Borel}. 

\end{proof}

\begin{proposition}\label{OpImpliesProba}
Let $(P_s)_{s\in\mathbbm{R}_+}$ be a path-dependent system of projectors. For any $(s,\eta)\in\mathbbm{R}_+\times\Omega$,  we set
\begin{equation}
\mathbbm{P}^{s,\eta}:\left(\begin{array}{rcl}
F&\longmapsto& P_s[\mathds{1}_F](\eta)\\
\mathcal{F}&\longrightarrow&\mathbbm{R}
\end{array}\right).
\end{equation}
Then for all $(s,\eta)$, $\mathbbm{P}^{s,\eta}$ defines a probability measure and  $(\mathbbm{P}^{s,\eta})_{(s,\eta)\in\mathbbm{R}_+\times \Omega}$ is a path-dependent canonical class.
%verifies items 1 and 2 of Definition  \ref{DefCondSyst} and 
%for every  $(s,\eta)\in\mathbbm{R}_+\times \Omega$, $t\geq s$  and $F\in\mathcal{F}$,
%\begin{equation}\label{Item3}
%\mathbbm{P}^{s,\eta}(F|\mathcal{F}^o_t)(\omega)=\mathbbm{P}^{t,\omega}(F)\text{ for }\mathbbm{P}^{s,\eta}\text{ almost all }\omega,
%\end{equation}
%which is weaker that item 3 of Definition  \ref{DefCondSyst}.
\end{proposition}

\begin{proof}
We fix $s$ and $\eta$. Since $\emptyset,\Omega\in\mathcal{F}^o_s$, then by
 the first item of Definition \ref{DefCondOp}, $P_s[\mathds{1}_{\emptyset}]=\mathds{1}_{\emptyset}$ and $P_s[\mathds{1}_{\Omega}]=\mathds{1}_{\Omega}$, 
%so for every $\eta\in\Omega$, 
so $\mathbbm{P}^{s,\eta}(\emptyset)=0$ and $\mathbbm{P}^{s,\eta}(\Omega)=1$. For any $F\in\mathcal{F}$, since $P_s$ is positivity preserving and $\mathds{1}_{\emptyset}\leq \mathds{1}_F\leq \mathds{1}_{\Omega}$ then $\mathds{1}_{\emptyset}\leq P_s[\mathds{1}_F]\leq \mathds{1}_{\Omega}$ so, $\mathbbm{P}^{s,\eta}$ takes values in $[0,1]$. If $(F_n)_n$ is a sequence of pairwise disjoint elements of $\mathcal{F}$ then the increasing sequence  
${\sum}_{k=0}^N \mathds{1}_{F_k}$ converges pointwise to $\mathds{1}_{\underset{n}{\bigcup}F_n}$. Since the $P_s$ are linear and monotonic  then
 $\underset{n}{\sum}P_s[\mathds{1}_{F_n}]=P_s[\mathds{1}_{\underset{n}{\bigcup}F_n}]$, hence 
%for every $\eta\in\Omega$,
 $\underset{n}{\sum}\mathbbm{P}^{s,\eta}(F_n)=\mathbbm{P}^{s,\eta}\left(\underset{n}{\bigcup}F_n\right)$. So for every $(s,\eta)$, $\mathbbm{P}^{s,\eta}$, is $\sigma$-additive, positive, vanishing in $\emptyset$ and takes value $1$ in $\Omega$ hence is a probability measure.
\\
Then, for any $(s,\eta)$ we have $\mathbbm{P}^{s,\eta}(\omega^s=\eta^s)=P_s[\mathds{1}_{\{\omega^s=\eta^s\}}](\eta)=\mathds{1}_{\{\omega^s=\eta^s\}}(\eta)=1$ since $\{\omega^s=\eta^s\}\in\mathcal{F}^o_s$, so item 1. of Definition  \ref{DefCondSyst} is satisfied. Concerning item 2., at fixed $s\in\mathbbm{R}_+$ and $F\in\mathcal{F}$, we have $(\eta\mapsto \mathbbm{P}^{s,\eta}(F))= P_s[\mathds{1}_F]$ which is $\mathcal{F}^o_s$-measurable since $P_s$ has its range in $\mathcal{B}_b^s(\Omega)$, see Definition \ref{DefCondOp}.

It remains to show item 3. We now fix $(s,\eta)\in\mathbbm{R}_+\times \Omega$, $t\geq s$  and $F\in\mathcal{F}$ and show that \eqref{DE13} holds.
Let $G\in\mathcal{F}^o_t$. 
 We need to show that $\mathbbm{E}^{s,\eta}[\mathds{1}_G\mathds{1}_F]=\mathbbm{E}^{s,\eta}[\mathds{1}_G(\zeta)\mathbbm{E}^{t,\zeta}[\mathds{1}_F]]$. We have
\begin{eqnarray*}
\mathbbm{E}^{s,\eta}[\mathds{1}_G\mathds{1}_F] &=& \mathbbm{E}^{s,\eta}[\mathbbm{E}^{t,\zeta}[\mathds{1}_G(\omega)\mathds{1}_F(\omega)]]\\
&=& \mathbbm{E}^{s,\eta}[\mathbbm{E}^{t,\zeta}[\mathds{1}_G(\zeta)\mathds{1}_F(\omega)]]\\
&=& \mathbbm{E}^{s,\eta}[\mathds{1}_G(\zeta)\mathbbm{E}^{t,\zeta}[\mathds{1}_F(\omega)]],
\end{eqnarray*}
where the first equality comes from the fact that $P_s = P_s\circ P_t$ and the second from the fact that $G\in\mathcal{F}^o_t$ and $\mathbbm{P}^{t,\zeta}(\omega^t=\zeta^t)=1$ so $\mathds{1}_G=\mathds{1}_G(\zeta)$ $\mathbbm{P}^{t,\zeta}$ a.s.
\end{proof}

\begin{corollary}\label{EqProbaOp}
The mapping 
\begin{equation}
\Phi:(\mathbbm{P}^{s,\eta})_{(s,\eta)\in\mathbbm{R}_+\times \Omega}\longmapsto\left(Z\longmapsto (\eta\mapsto\mathbbm{E}^{s,\eta}[Z])\right)_{s\in\mathbbm{R}_+},
\end{equation}
is a bijection between the set of path-dependent canonical classes and the set of path-dependent system of projectors, whose  reciprocal map is given by
\begin{equation}
\Phi^{-1}:(P_s)_{s\in\mathbbm{R}_+}\longmapsto\left(F\mapsto P_s[\mathds{1}_F](\eta)\right)_{(s,\eta)\in\mathbbm{R}_+\times \Omega}.
\end{equation}
\end{corollary}
\begin{proof}
$\Phi$ is by Proposition \ref{P19} well-defined.
Moreover it is injective since if $\mathbbm{P}^1$ and $\mathbbm{P}^2$ are two probabilities such that respective expectations of all the bounded r.v. are 
the same then $\mathbbm{P}^1= \mathbbm{P}^2$. 
%and is clearly injective.
 Then given a path-dependent system of projectors $(P_s)_{s\in\mathbbm{R}_+}$, 
 by Proposition \ref{OpImpliesProba} $\left(\mathbbm{P}^{s,\eta}:F\mapsto P_s[\mathds{1}_F](\eta)\right)_{(s,\eta)\in\mathbbm{R}_+\times \Omega}$ is a path-dependent canonical class. 
It is then enough to show that the image through $\Phi$ of that  path-dependent canonical class
is indeed $(P_s)_{s\in\mathbbm{R}_+}$.
Let $(Q_s)_{s\in\mathbbm{R}_+}$ denote its image by $\Phi$, in order to conclude we are left to show that $Q_s=P_s$ for all $s$. 
\\
We fix $s$. For every $F\in\mathcal{F}, \eta \in \Omega $ we have $Q_s[\mathds{1}_F](\eta)= \mathbbm{P}^{s,\eta}(F)=P_s[\mathds{1}_F](\eta)$ so $Q_s$ and $P_s$ coincide on the indicator functions, hence on the simple functions by linearity, and everywhere by monotonicity and the fact that every bounded Borel function is the limit of an increasing sequence of simple functions.
\end{proof}

\begin{definition}\label{ProbaOp}
From now on, two elements mapped by the previous bijection will be said to be \textbf{associated}. 
\end{definition}

\begin{remark} \label{RMarkov}
Path-dependent canonical classes naturally extend  canonical Markov classes (see Definition C.5 in \cite{paper3} for instance) as follows. \\
	Let $(\mathbbm{P}^{s,x})_{(s,x)\in\mathbbm{R}_+\times E}$ be a canonical Markov class with state space $E$ and let $(P_{s,t})_{0\leq s\leq t}$ denote its transition kernel, see Definition C.3 in \cite{paper3}. \\ 
For all $(s,\eta)\in\mathbbm{R}_+\times \Omega$, let $\mathbbm{P}^{s,\eta}$
 be the unique probability measure on $(\Omega,\mathcal{F})$ such that  $\mathbbm{P}^{s,\eta}(\omega^s=\eta^s)$ and $\mathbbm{P}^{s,\eta}$ coincides on $\sigma(X_r|r\geq s)$ with  $\mathbbm{P}^{s,\eta(s)}$. Then $(\mathbbm{P}^{s,\eta})_{(s,\eta)\in\mathbbm{R}_+\times \Omega}$ is a path-dependent canonical class. Let $(P_s)_{s\in\mathbbm{R}_+}$ denote the associated path-dependent system of projectors. Then for all bounded Borel $\phi:E\mapsto\mathbbm{R}$, $\eta\in\Omega$ and $0\leq s\leq t$ we have
	\begin{equation}
	P_s[\phi\circ X_t](\eta)=\mathbbm{E}^{s,\eta}[\phi(X_t)]=\mathbbm{E}^{s,\eta(s)}[\phi(X_t)]=P_{s,t}[\phi](\eta(s)).
	\end{equation}
\end{remark}
%
%\begin{notation}
%Let $(P_s)_{s\in\mathbbm{R}_+}$ be a path-dependent system of projectors, and $(\mathbbm{P}^{s,\eta})_{(s,\eta)\in\mathbbm{R}_+\times \Omega}$ the associated path-dependent canonical class.
%Then for any r.v. $\phi$ $\mathcal{L}^1$ under $\mathbbm{P}^{s,\eta}$, $P_s[\phi](\eta)$ will still denote the expectation of $\phi$ under $\mathbbm{P}^{s,\eta}$. In other words we extend the linear form $\phi\longmapsto P_s[\phi](\eta)$ from $\mathcal{B}_b(\Omega)$ to $\mathcal{L}^1(\mathbbm{P}^{s,\eta})$.
%\end{notation}

\begin{notation}
For the rest of this section, we are given a   path-dependent canonical class $(\mathbbm{P}^{s,\eta})_{(s,\eta)\in\mathbbm{R}_+\times \Omega}$ and $(P_s)_{s\in\mathbbm{R}_+}$ denotes the associated path-dependent system of projectors.
\end{notation}

\begin{definition}\label{defcompletion}
Let $\mathbbm{P}$ be a probability on $(\Omega,\mathcal{F})$. 
%and $\mathcal{N}_{\mathbbm{P}}$ be the set of $\mathbbm{P}$-negligible sets. 
If $\mathcal{G}$ be a sub-$\sigma$-field of $\mathcal{F}$, we call \textbf{$\mathbbm{P}$-closure} of $\mathcal{G}$ the $\sigma$-field generated by $\mathcal{G}$ and the set of $\mathbbm{P}$-negligible sets. We denote it $\mathcal{G}^\mathbbm{P}$.
In the particular case $\mathcal{G}=\mathcal{F}$, we call $\mathcal{F}^\mathbbm{P}$ \textbf{$\mathbbm{P}$-completion} of  $\mathcal{F}$.

\end{definition}

\begin{remark}\label{remcompletion}
Thanks to Remark 32.b) in Chapter II of \cite{dellmeyer75}, we have an equivalent definition of the $\mathbbm{P}$-closure of some sub-$\sigma$-field $\mathcal{G}$ of $\mathcal{F}$ which can be 
characterized by the following property: $B\in\mathcal{G}^{\mathbbm{P}}$ if and only if there exist $F\in\mathcal{G}$
% and a $\mathbbm{P}$-negligible set $N$ 
such that 
%$F\backslash N\subset B\subset F\cup N$. or 
 $\mathds{1}_B=\mathds{1}_F$ $\mathbbm{P}$ a.s. 
\\
 Moreover, $\mathbbm{P}$ can be extended to a probability on $\mathcal{G}^{\mathbbm{P}}$ by setting $\mathbbm{P}(B):=\mathbbm{P}(F)$ for such events.  
\end{remark}

\begin{notation}\label{CompletedBasis}
For any $(s,\eta)\in\mathbbm{R}_+\times \Omega$ we will consider the  stochastic basis $\left(\Omega,\mathcal{F}^{s,\eta},\mathbbm{F}^{s,\eta}:=(\mathcal{F}^{s,\eta}_t)_{t\in\mathbbm{R}_+},\mathbbm{P}^{s,\eta}\right)$ where $\mathcal{F}^{s,\eta}$ is  the $\mathbbm{P}^{s,\eta}$-completion of $\mathcal{F}$,  $\mathbbm{P}^{s,\eta}$ is extended to $\mathcal{F}^{s,\eta}$ and   $\mathcal{F}^{s,\eta}_t$ is 
the $\mathbbm{P}^{s,\eta}$-closure of $\mathcal{F}_t$ for every $t\in\mathbbm{R}_+$. 
\end{notation}
We remark that, for any $(s,\eta)\in\mathbbm{R}_+\times \Omega$, $\left(\Omega,\mathcal{F}^{s,\eta},\mathbbm{F}^{s,\eta},\mathbbm{P}^{s,\eta}\right)$ is a stochastic basis fulfilling the usual conditions, see 1.4 in \cite{jacod} Chapter I. 
\\
\\
A direct consequence of Remark 32.b) in Chapter II of \cite{dellmeyer75} is the following.
\begin{proposition}\label{Fversion}
Let $\mathcal{G}$ be a sub-$\sigma$-field of $\mathcal{F}$, $\mathbbm{P}$ a probability on $(\Omega,\mathcal{F})$ and $\mathcal{G}^{\mathbbm{P}}$ the $\mathbbm{P}$-closure of $\mathcal{G}$. Let $Z^{\mathbbm{P}}$ be a real $\mathcal{G}^{\mathbbm{P}}$-measurable random variable. There exists a $\mathcal{G}$-measurable random variable $Z$ such that $Z=Z^{\mathbbm{P}}$ $\mathbbm{P}$-a.s.
\end{proposition}

Proposition \ref{Fversion} yields the following.

\begin{proposition}\label{ConditionalExp}
	Let $\mathbbm{P}$ be a probability measure on $(\Omega,\mathcal{F})$, let $\mathbbm{G}:=(\mathcal{G}_t)_{t\in\mathbbm{R}_+}$ be a filtration and $\mathbbm{G}^{\mathbbm{P}}$ denote $(\mathcal{G}^{\mathbbm{P}}_t)_{t\in\mathbbm{R}_+}$. Let $Z$
 be a positive or $\mathcal{L}^1$-random variable and $t\in\mathbbm{R}_+$. 
Then $\mathbbm{E}[Z|\mathcal{G}_t]=\mathbbm{E}[Z|\mathcal{G}^{\mathbbm{P}}_t]$ $\mathbbm{P}$ a.s. In particular, $(\mathbbm{P},\mathbbm{G})$-martingales are also $(\mathbbm{P},\mathbbm{G}^{\mathbbm{P}})$-martingales.
\end{proposition}

According to Proposition \ref{ConditionalExp} for $ \mathbbm{P} =
 \mathbbm{P}^{s,\eta}$, the
related conditional expectations
with respect to $\mathcal{F}_t^{s,\eta}$
coincide with conditional expectations with respect to $\mathcal{F}_t$.
For that reason we will only use the notation
 $\mathbbm{E}^{s,\eta}[\,\cdot\, \vert \mathcal{F}_t]$
 omitting  the $(s,\eta)$-superscript
over $\mathcal{F}_t$.

In the next proposition, $\mathcal{F}_t^{o,s,\eta}$
 will denote for any $(s,\eta)\in\mathbbm{R}_+\times\Omega$ and $t\geq s$ the $\mathbbm{P}^{s,\eta}$-closure of $\mathcal{F}^o_t$.

\begin{proposition}\label{CompFiltIni}
Assume that Hypothesis \ref{HypClass} holds.
For any $(s,\eta)\in\mathbbm{R}_+\times\Omega$ and $t\geq s$,  $\mathcal{F}_t^{o,s,\eta}=\mathcal{F}_t^{s,\eta}$.
\end{proposition}
\begin{proof}
We fix $s,\eta,t$. Since inclusion $\mathcal{F}_t^{o,s,\eta}\subset\mathcal{F}_t^{s,\eta}$ is obvious, we  show the converse inclusion. 
\\
Let $F^{s,\eta}\in\mathcal{F}_t^{s,\eta}$. By  Remark \ref{remcompletion}, there exists $F\in\mathcal{F}_t$, such that $\mathds{1}_{F^{s,\eta}}=\mathds{1}_{F}$ $\mathbbm{P}^{s,\eta}$ a.s. It is therefore sufficient to prove the existence of some $F^o\in\mathcal{F}^o_t$ such that $\mathds{1}_{F^o}=\mathds{1}_{F}$ $\mathbbm{P}^{s,\eta}$ a.s. (and therefore $\mathds{1}_{F^o}=\mathds{1}_{F^{s,\eta}}$ $\mathbbm{P}^{s,\eta}$ a.s.) to conclude that $F^{s,\eta}\in\mathcal{F}_t^{o,s,\eta}$.
\\
\\
We set $Z:\begin{array}{rcl} \omega &\longmapsto&\mathbbm{P}^{t,\omega}(F)\\ \Omega &\longrightarrow &[0,1]\end{array}$. 
By \eqref{DE14} and the fact that $F\in\mathcal{F}_t$, we have  
\begin{equation}\label{eqCompFilt}
Z(\omega)=\mathbbm{P}^{t,\omega}(F)=\mathbbm{E}^{s,\eta}[\mathds{1}_F|\mathcal{F}_t](\omega)=\mathds{1}_F(\omega)\quad \mathbbm{P}^{s,\eta}\text{a.s.}
\end{equation}
By Definition \ref{DefCondSyst}, $Z$ is $\mathcal{F}^o_t$-measurable, so $F^o:=Z^{-1}(\{1\})$ belongs to $\mathcal{F}^o_t$, and we will proceed showing that $\mathds{1}_{F^o}=\mathds{1}_{F}$ $\mathbbm{P}^{s,\eta}$ a.s.
\\
By construction, $\mathds{1}_{F^o}(\omega)=1$ iff $\mathbbm{P}^{t,\omega}(F)=1$ and  $\mathds{1}_{F^o}(\omega)=0$ iff $\mathbbm{P}^{t,\omega}(F)\in[0,1[$. So 
\begin{equation}
\begin{array}{rcl}
&&\{\omega:\mathds{1}_{F^o}(\omega)\neq\mathds{1}_{F}(\omega)\}\\
&=&\{\omega:\mathds{1}_{F^o}(\omega)=1\text{ and }\mathds{1}_{F}(\omega)=0\}\bigcup\{\omega:\mathds{1}_{F^o}(\omega)=0\text{ and }\mathds{1}_{F}(\omega)=1\}\\
&=&\{\omega:\mathbbm{P}^{t,\omega}(F)=1\text{ and }\mathds{1}_{F}(\omega)=0\}\bigcup\{\omega:\mathbbm{P}^{t,\omega}(F)\in[0,1[\text{ and }\mathds{1}_{F}(\omega)=1\}\\
& \subset &\{\omega:\mathbbm{P}^{t,\omega}(F)\neq\mathds{1}_{F}(\omega)\},
\end{array}
\end{equation}
where the latter set is $\mathbbm{P}^{s,\eta}$-negligible by \eqref{eqCompFilt}.
\end{proof}

Combining Propositions \ref{Fversion} and \ref{CompFiltIni}, we have the following.
\begin{corollary}\label{F0version}
Assume that Hypothesis \ref{HypClass} holds and let us fix $(s,\eta)\in\mathbbm{R}_+\times\Omega$ and $t\geq s$. Given an  $\mathcal{F}^{s,\eta}_t$-measurable r.v.  $Z^{s,\eta}$, there exists  an $\mathcal{F}^o_t$-measurable r.v. $Z^o$
  such that $Z^{s,\eta}=Z^o$ $\mathbbm{P}^{s,\eta}$ a.s.
\end{corollary}

\begin{definition}\label{trivial}
If $(\tilde{\Omega},\tilde{\mathcal{F}},\tilde{\mathbbm{P}})$ 
is a probability space and
 $\mathcal{G}$ is a sub-$\sigma$-field of $\tilde{\mathcal{F}}$, we say that $\mathcal{G}$ is \textbf{$\mathbbm{P}$-trivial} if for any element $G$ of $\mathcal{G}$,
 then $\mathbbm{P}(G)\in\{0,1\}$.
\end{definition} 
\begin{corollary}\label{CoroTrivial}
Assume that Hypothesis \ref{HypClass} holds.
For every $(s,\eta)\in\mathbbm{R}_+\times\Omega$, $\mathcal{F}^o_s$ and $\mathcal{F}_s$ are $\mathbbm{P}^{s,\eta}$-trivial.
\end{corollary}
\begin{proof}
We fix $(s,\eta)\in\mathbbm{R}_+\times\Omega$.
% Since $\mathbbm{P}^{s,\eta}(\omega^s=\eta^s)=1$, 
 We start by showing that $\mathcal{F}^o_s$ is $\mathbbm{P}^{s,\eta}$-trivial. For every
$B\in \mathcal{F}^o_s$ and $\omega$ we have $\mathds{1}_B(\omega)=\mathds{1}_B(\omega^s)$, and since $\mathbbm{P}^{s,\eta}(\omega^s=\eta^s)=1$,
 we have $\mathds{1}_B(\omega^s)=\mathds{1}_B(\eta^s)$ $\mathbbm{P}^{s,\eta}$ a.s.  So $\mathbbm{P}^{s,\eta}(B)=\mathbbm{E}^{s,\eta}[\mathds{1}_B(\omega)]=\mathds{1}_B(\eta^s)\in\{0,1\}$.
\\
Then, it is clear that adding $\mathbbm{P}^{s,\eta}$-negligible sets does
not change the fact of being $\mathbbm{P}^{s,\eta}$-trivial, so $\mathcal{F}_s^{o,s,\eta}$ (which by Proposition \ref{CompFiltIni} is equal to $\mathcal{F}_s^{s,\eta}$) is $\mathbbm{P}^{s,\eta}$-trivial and therefore so is $\mathcal{F}_s\subset \mathcal{F}_s^{s,\eta}$. 

\end{proof}

\section{Path-dependent Additive Functionals}\label{SectionMAF}

In this section, we introduce the notion of Path-dependent Additive Functionals
 that we use in the paper. As already anticipated, 
this can be interpreted as a path-dependent extension of the notion of non-homogeneous Additive Functionals of a canonical Markov class developed in  \cite{paperAF}.
For that reason,  several proofs of this section are 
 very similar to those of \cite{paperAF} and 
are inspired from  \cite{dellmeyerD} Chapter XV,
which treats the time-homogeneous case.
% which deals with similar results, and very similar to those of \cite{paperAF}. Some proofs will be therefore  postponed  to the Appendix 
%and sometimes we will refer to \cite{paperAF}. 
%The advised reader may consult a slightly longer version of the present paper in \cite{paperMPv1} where all proofs are carefully written.

We keep on using Notation \ref{canonicalspace} and we fix a  
path-dependent canonical class $(\mathbbm{P}^{s,\eta})_{(s,\eta)\in\mathbbm{R}_+\times \Omega}$ and assume the following for the whole section.
\begin{hypothesis}\label{HypAF}
$(\mathbbm{P}^{s,\eta})_{(s,\eta)\in\mathbbm{R}_+\times \Omega}$ is progressive and verifies Hypothesis \ref{HypClass}. 
\end{hypothesis}
We will use the notation $\Delta:=\{(t,u)\in\mathbbm{R}_+^2|t\leq u\}$.

\begin{definition}\label{DefAF}
 On $(\Omega,\mathcal{F})$, a \textbf{path-dependent Additive Functional} (in short path-dependent AF)
will be  a random-field  
 $A:=(A_{t,u})_{(t,u)\in\Delta}$ with values in $\mathbbm{R}$ 
verifying the two following conditions.
\begin{enumerate}
\item For any $(t,u)\in\Delta$, $A_{t,u}$ is $\mathcal{F}^o_{u}$-measurable;
\item for any $(s,\eta)\in\mathbbm{R}_+\times \Omega$, there exists a real cadlag $\mathbbm{F}^{s,\eta}$-adapted process $A^{s,\eta}$ (taken equal to zero on $[0,s]$ by convention) such that for any $\eta\in \Omega$ and $s\leq t\leq u$, 
\begin{equation*}
A_{t,u} = A^{s,\eta}_u-A^{s,\eta}_t \,\text{  }\, \mathbbm{P}^{s,\eta}\text{ a.s.}
\end{equation*}
\end{enumerate}
We denote by $A^t$ the ($\mathbbm{F}^o$-adapted) process
 $u\mapsto A_{t,u}$ indexed by $[t,+\infty[$. For any $(s,\eta)\in [0,t]\times \Omega$, $A^{s,\eta}_{\cdot}-A^{s,\eta}_t$ 
 is a $\mathbbm{P}^{s,\eta}$-version  of $A^t$ on $[t,+\infty[$.
$A^{s,\eta}$ will be called the \textbf{cadlag version of $A$ under} $\mathbbm{P}^{s,\eta}$.
% since $A$ is not a process.

A path-dependent Additive Functional will be called a \textbf{path-dependent Martingale Additive Functional} (in short path-dependent MAF) if under any $\mathbbm{P}^{s,\eta}$ its cadlag version is a martingale.

More generally, a path-dependent AF will be said to verify a certain property 
(being non-decreasing, of bounded variation, square integrable,
having  ${\mathcal L}^1$-terminal value) if under any $\mathbbm{P}^{s,\eta}$ its cadlag version verifies it.

Finally, given two increasing path-dependent AFs $A$ and $B$, $A$ will be said to be
 \textbf{absolutely continuous with respect to} $B$ if for any 
$(s,\eta)\in \mathbbm{R}_+\times \Omega$, $dA^{s,\eta} \ll dB^{s,\eta}$ in the sense of stochastic measures. This means that  $dA^{s,\eta}(\omega)$ is absolutely continuous with respect to $dB^{s,\eta}(\omega)$ for $\mathbbm{P}^{s,\eta}$ almost all $\omega$. 
\end{definition}

\begin{remark}
The set of path-dependent AFs (resp.  path-dependent AFs with bounded variation,  path-dependent AFs with ${\mathcal L}^1$-terminal value,  path-dependent MAFs, square integrable path-dependent MAFs) is a linear space.
\end{remark}

\begin{lemma}\label{LemMAF}
Let $M$ be  an $\mathbbm{F}^o$-adapted process such that for all 
$(s,\eta)$, on $[s,+\infty[$, 
%\begin{itemize}
%\item 
$M$ is a
 $(\mathbbm{P}^{s,\eta},\mathbbm{F}^o)$-martingale. \\
% (resp. square integrable martingale). \\
Then, for all $(s,\eta)$,
 $M_{\cdot \vee s} -M_s$ admits 
 a $\mathbbm{P}^{s,\eta}$-version which is a  $(\mathbbm{P}^{s,\eta},\mathbbm{F}^{s,\eta})$
 cadlag martingale
%(resp. square integrable martingale) 
$M^{s,\eta}$ vanishing in $[0,s]$.
%\end{itemize}
In particular $M_{t,u}(\omega):=M_u(\omega)-M_t(\omega)$ defines a 
path-dependent MAF 
%(resp. square integrable path-dependent MAF)
 with 
cadlag version $M^{s,\eta}$ 
%(extended by convention on $[0,s[$ by the value $0$)
 under $\mathbbm{P}^{s,\eta}$.
\end{lemma}

\begin{proof}
By Propositions \ref{ConditionalExp} and \ref{CompFiltIni},  $M$ is also on $[s,+\infty[$ a $(\mathbbm{P}^{s,\eta},\mathbbm{F}^{s,\eta})$-martingale hence $M_{\cdot \vee s} -M_s$ is on $\mathbbm{R}_+$ a $(\mathbbm{P}^{s,\eta},\mathbbm{F}^{s,\eta})$-martingale and vanishes on $[0,s]$. Since $\mathbbm{F}^{s,\eta}$ satisfies the usual conditions, then $M_{\cdot \vee s} -M_s$ admits  a cadlag $\mathbbm{P}^{s,\eta}$-modification $M^{s,\eta}$ which also is a $(\mathbbm{P}^{s,\eta},\mathbbm{F}^{s,\eta})$-martingale vanishing in $[0,s]$.  It clearly verifies that $M_{t,u}=M_u-M_t=M^{s,\eta}_u-M^{s,\eta}_t$ $\mathbbm{P}^{s,\eta}$-a.s. for all $s\leq t\leq u$. 
\end{proof}

\begin{example}
Let $Z$ be an  $\mathcal{F}$-measurable bounded r.v. 
%REMARQUE borel
 A typical example of process 
% is that typical examples of processes 
verifying the conditions of previous Lemma \ref{LemMAF} is
 given by $M^Z:(t,\omega)\longmapsto \mathbbm{E}^{t,\omega}[Z]$,
see Remark \ref{Borel}.
% for any $\mathcal{F}$-measurable bounded r.v. $Z$.
\end{example}

The following results state that, for a given  square integrable path-dependent MAF $(M_{t,u})_{(t,u)\in\Delta}$  we can exhibit
 two non-decreasing path-dependent AFs with ${\mathcal L}^1$-terminal value, denoted respectively by  $([M]_{t,u})_{(t,u)\in\Delta}$ 
and    $(\langle M \rangle_{t,u})_{(t,u)\in\Delta}$,
which will play respectively 
the role of a quadratic variation and an angular bracket of it.
Moreover we will show 
 that the Radon-Nikodym derivative of the mentioned angular bracket of 
a square integrable path-dependent MAF with respect to a reference function $V$ is 
a progressively measurable process which does not depend on the probability.
\\
The proof of the proposition below is postponed to the appendix.
\begin{proposition}\label{VarQuadAF}
Let $(M_{t,u})_{(t,u)\in\Delta}$ be a square integrable path-dependent MAF, and for any $(s,\eta)\in\mathbbm{R}_+\times \Omega$,
  $[M^{s,\eta}]$ denote the quadratic variation of its cadlag version $M^{s,\eta}$ under $\mathbbm{P}^{s,\eta}$.
 Then there exists a non-decreasing path-dependent AF with ${\mathcal L}^1$-terminal value  which we will call $([M]_{t,u})_{(t,u)\in\Delta}$ and which, for any $(s,\eta)\in\mathbbm{R}_+\times \Omega$, has  $[M^{s,\eta}]$
as cadlag version
 under $\mathbbm{P}^{s,\eta}$. 
\end{proposition}

The next result can be seen as an extension of Theorem 15 Chapter XV in \cite{dellmeyerD} to a path-dependent
context and will be needed to show that the result above also holds for the angular bracket. Its proof is also postponed to the appendix.
\begin{proposition}\label{AngleBracketAF}
Let $(B_{t,u})_{(t,u)\in\Delta}$ be a non-decreasing  path-dependent AF with $\mathcal{L}^1$-
terminal value. For any $(s,\eta)\in\mathbbm{R}_+\times \Omega$, let 
$B^{s,\eta}$ be its cadlag version under $\mathbbm{P}^{s,\eta}$ and let $A^{s,\eta}$ be the predictable dual projection of $B^{s,\eta}$ in 
$(\Omega,\mathcal{F}^{s,\eta},\mathbbm{F}^{s,\eta},\mathbbm{P}^{s,\eta})$. Then there exists a non-decreasing  path-dependent AF with ${\mathcal L}^1$-terminal value $
(A_{t,u})_{(t,u)\in\Delta}$ such that under any  $\mathbbm{P}^{s,\eta}$, the cadlag version of $A$ is $A^{s,\eta}$.
\end{proposition}
\begin{remark}\
	\begin{enumerate}
		\item About the notion of dual predictable projection (also called compensator) related to some	stochastic basis we refer to Theorem 3.17 in Chapter I of \cite{jacod}.
		\item We recall that, whenever $M,N$ are two local martingales,
		the angle bracket $\langle M, N \rangle$ is the dual predictable
		projection of $[M,N]$, see Proposition 4.50 b) in Chapter I of \cite{jacod}.
	\end{enumerate}
\end{remark}

\begin{corollary}\label{AFbracket}
Let $(M_{t,u})_{(t,u)\in\Delta}$, $(N_{t,u})_{(t,u)\in\Delta}$ be two square integrable path-dependent MAFs, let $M^{s,\eta}$ (respectively $N^{s,\eta}$) be the cadlag version of $M$ (respectively $N$) under $\mathbbm{P}^{s,\eta}$.
Then there exists a bounded variation path-dependent AF with ${\mathcal L}^1$-terminal
value, denoted $(\langle M,N\rangle_{t,u})_{(t,u)\in\Delta}$, such that under any $\mathbbm{P}^{s,\eta}$, the cadlag version of $\langle M,N\rangle$ is $\langle M^{s,\eta},N^{s,\eta}\rangle$. If $M=N$ the path-dependent AF $\langle M,N\rangle$ will be denoted $\langle M\rangle$ and is non-decreasing.
\end{corollary}

\begin{proof}
This can be proved as for Corollary 4.11 in \cite{paperAF}, replacing parameter $(s,x)$ with $(s,\eta)$.
%If $M=M'$, the corollary comes from the combination of Propositions \ref{VarQuadAF} and \ref{AngleBracketAF}, and the fact that the angular bracket of a square integrable martingale is the dual predictable projection of its quadratic variation. 
%\\
%Otherwise, it is clear that $M+M'$ and $M-M'$ are square integrable path-dependent MAFs, so we can consider the non-decreasing path-dependent AFs $\langle M-M'\rangle$ and $\langle M+M'\rangle$. We introduce the path-dependent AF
%\begin{equation*}
%\langle M,M'\rangle := \frac{1}{4}(\langle M+M'\rangle - \langle M-M'\rangle),
%\end{equation*}
%which by polarization has cadlag version $\langle M^{s,\eta},M'^{s,\eta}\rangle$ under $\mathbbm{P}^{s,\eta}$. $\langle M,M'\rangle$ is therefore a  bounded variation path-dependent AF with $L^1$-terminal value.
\end{proof}

The result below concerns 
%We are now going to state our result concerning 
the Radon-Nikodym derivative of a non-decreasing continuous path-dependent AF with respect to some reference measure $dV$. Its proof is postponed to the Appendix.

\begin{proposition}\label{RadonDerivAF}
Let $V:\mathbbm{R}_+\longrightarrow \mathbbm{R}$ be a non-decreasing continuous function.
Let $A$ be a non-negative, non-decreasing path-dependent AF absolutely continuous with respect to $V$, and for any $(s,\eta)\in\mathbbm{R}_+\times \Omega$ let $A^{s,\eta}$ be the cadlag version of $A$ under $\mathbbm{P}^{s,\eta}$. There exists an $\mathbbm{F}^o$-progressively measurable process $h$  such that for any $(s,\eta)\in\mathbbm{R}_+\times \Omega$,  $A^{s,\eta}=\int_s^{\cdot \vee s}h_rdV_r$, in the sense of indistinguishability.
\end{proposition}

\begin{proposition}\label{VarTotAF}
	Let $(A_{t,u})_{(t,u)\in\Delta}$ be a path-dependent AF with bounded variation, taking ${\mathcal L}^1$-terminal value. 
	Then there exists an increasing path-dependent AF that we denote $(Pos(A)_{t,u})_{(t,u)\in\Delta}$ (resp. $(Neg(A)_{t,u})_{(t,u)\in\Delta}$), which, for any $(s,\eta)\in\mathbbm{R}_+\times \Omega$, has $Pos(A^{s,\eta})$ (resp. $Neg(A^{s,\eta})$))
	as cadlag version
	under $\mathbbm{P}^{s,\eta}$. 
\end{proposition}
\begin{proof}
This can be proved similarly as forProposition 4.14 in \cite{paperAF}, replacing parameter $(s,x)$ with $(s,\eta)$.
\end{proof}

\begin{corollary}\label{BracketMAFnew}
	Let $V$ be a continuous non-decreasing function.
	Let $M$ and $N$ be two square integrable path-dependent MAFs and let $M^{s,\eta}$ (respectively $N^{s,\eta}$) be the cadlag version of $M$ (respectively $N$) under a fixed $\mathbbm{P}^{s,\eta}$. Assume that  $\langle N\rangle$ is absolutely continuous with respect to $dV$.
	There exists an $\mathbbm{F}^o$-progressively measurable process  $k$  such that for any $(s,\eta)\in\mathbbm{R}_+\times \Omega$, $\langle M^{s,\eta},N^{s,\eta}\rangle =  \int_s^{\cdot\vee s}k_rdV_r$.
\end{corollary}
\begin{proof}
The proof follows the same lines as the one of
 Proposition 4.17 in \cite{paperAF} replacing parameter $(s,x)$ by
 $(s,\eta)$ and Borel functions of $(t,X_t)$ with $\mathbbm{F}^o$-progressively measurable processes. 
We make use of Corollary \ref{AFbracket}, Propositions
\ref{VarTotAF} and  \ref{RadonDerivAF},
  respectively in substitution of
Corollary 4.11 an Propositions 4.14 and 4.13.
 % NECESSAIRE DE DIRE CECI. DANS LA PREUVE? PLUTOT UN COMMETAIRE? It is proved exactly as Proposition 4.17 in \cite{paperAF} replacing parameter $(s,x)$ by $(s,\eta)$ and Borel functions of $(t,X_t)$ with $\mathbbm{F}^o$-progressively measurable processes.
\end{proof}

\begin{corollary}
Let $V$ be a continuous non-decreasing function.
Let $M$ (resp. $N$) be an $\mathbbm{F}^o$-adapted process such that for all $(s,\eta)$, $M$ (resp. $N$) is on $[s,+\infty[$ a $(\mathbbm{P}^{s,\eta},\mathbbm{F}^o)$ square integrable martingale. 
For any $(s,\eta)$, let $M^{s,\eta}$ (resp. $N^{s,\eta}$) denote its $\mathbbm{P}^{s,\eta}$-cadlag version.
% in the sense of Lemma \ref{LemMAF}).
 Assume that for all $(s,\eta)$, $d\langle N^{s,\eta}\rangle\ll dV$.
\\
Then there exists an $\mathbbm{F}^o$-progressively measurable process  $k$  such that for any $(s,\eta)\in\mathbbm{R}_+\times \Omega$, $\langle M^{s,\eta},N^{s,\eta}\rangle =  \int_s^{\cdot\vee s}k_rdV_r$.
\end{corollary}
\begin{proof}
The mentioned cadlag versions exist because of 
 Lemma \ref{LemMAF}. The statement follows by 
the same Lemma  \ref{LemMAF} and
Corollary \ref{BracketMAFnew}.

\end{proof}

\section{Path-dependent Martingale problems}
\label{SMP}

\subsection{Abstract Martingale Problems}\label{S3_1}

In this section we show that, whenever a (path-dependent) martingale problem 
is well-posed, then its solution is a 
path-dependent canonical class verifying Hypothesis \ref{HypClass}.
 This relies on the same mathematical tools than those used by D.S Stroock and S.R.S Varadhan in the context of Markovian diffusions in \cite{stroock}. 
Indeed it was already 
known that the ideas of \cite{stroock}
could be used in any type of Markovian setup and 
not just for martingale problems associated to diffusions, see \cite{EthierKurz} for example. One of the interests of the following lines is to show that 
their scope goes beyond the Markovian framework.
First we prove that $\eta\mapsto \mathbbm{P}^{s,\eta}$ is measurable, 
using  well-posedness arguments and the celebrated Kuratowsky Theorem.
Then we show in Proposition \ref{MPimpliesCond} 
that the solution of the martingale problem  verifies \eqref{DE14},
which is the analogous formulation of Markov property,
through the theory of regular conditional expectations and 
again the fact that the martingale problem is well-posed.

\begin{notation}\label{Omegat}
	For every $t\in\mathbbm{R}_+$, $\Omega^t:=\{\omega\in\Omega:\omega=\omega^t\}$ will denote the set of constant paths  after time $t$.
	We also denote $\Lambda:=\{(s,\eta)\in\mathbbm{R}_+\times\Omega: \eta\in\Omega^s\}$. 
\end{notation}

\begin{proposition}
\begin{enumerate}\
\item	$\Lambda$ is a closed subspace of $\mathbbm{R}_+\times\Omega$, hence a Polish space when equipped with the induced topology. 
\item	For any  $t\in\mathbbm{R}_+$, $\Omega^t$ is also a closed subspace of
 $\Omega$.
% hence a Polish space when equipped with the induced topology.
\end{enumerate}
\end{proposition}
\begin{proof}
	We will only show the first statement since the proof of the second one is similar but simpler. 
	Let $(s_n,\eta_n)_n$ be a sequence in $\Lambda$.
	Let 
	$(s,\eta)\in \mathbbm{R}_+\times\Omega$ and assume that $s_n\rightarrow s$ and that $\eta_n$ tends to $\eta$ for the Skorokhod topology. Then $\eta_n$ tends to $\eta$ Lebesgue a.e. Let $\epsilon>0$. There is a 
	subsequence  $(s_{n_k})$ such that
	$|s_{n_k}-s|\le \epsilon$, implying that for all $k$, $\eta_{n_k}$ is constantly equal to $\eta_{n_k}(s_{n_k})$ on
	$[s+\epsilon,+\infty[$. Since $\eta_n$ tends to $\eta$ Lebesgue a.e., then necessarily, $\eta_{n_k}(s_{n_k})$
	tends to some $c\in E$ and $\eta$ takes value $c$ a.e. on $[s+\epsilon,+\infty[$. This holds for every 
	$\epsilon$, and $\eta$ is cadlag, so $\eta$ is constantly equal to $c$ on $[s,+\infty[$, implying that $(s,\eta)\in\Lambda$. 
\end{proof}
From now on, $\Lambda$, introduced in Notation \ref{Omegat}, is equipped with the trace topology. 

\begin{proposition}\label{BorelLambda}
	The Borel $\sigma$-field $\mathcal{B}(\Lambda)$ is equal to the trace $\sigma$-field $\Lambda\cap\mathcal{P}ro^o$. For any  $t\in\mathbbm{R}_+$, the Borel $\sigma$-field $\mathcal{B}(\Omega^t)$ is equal to the 
	trace $\sigma$-field $\Omega^t\cap\mathcal{F}^o_t$.
\end{proposition}
\begin{proof}
	Again we only show the first statement since the proof of the second one is similar.
	By definition of the topology on $\Lambda$, it is clear that $\mathcal{B}(\Lambda)=\Lambda\cap\mathcal{B}(\mathbbm{R}_+\times\Omega)=\Lambda\cap(\mathcal{B}(\mathbbm{R}_+)\otimes\mathcal{F})$ contains $\Lambda\cap\mathcal{P}ro^o$. We show the converse inclusion. The sets $\Lambda\cap([s,u]\times\{\omega(r)\in A\})$  for $s,u,r\in\mathbbm{R}_+$ with $s\leq u$, $A\in\mathcal{B}(E)$ generate $\Lambda\cap(\mathcal{B}(\mathbbm{R}_+)\otimes\mathcal{F})$ so it is enough to show that these sets belong to $\Lambda\cap\mathcal{P}ro^o$.
	
	We fix $s\leq u$ and $r$ in $\mathbbm{R}_+$, and $A\in\mathcal{B}(E)$. We have
	\begin{equation}
	\begin{array}{rcl}
	\Lambda\bigcap\left([s,u]\times\{\omega(r)\in A\}\right) &=& \left\{(t,\omega):\left\{\begin{array}{l}t\in[s,u]\\ \omega=\omega^t\\\omega(r)\in A
	\end{array}\right.\right\}   \\
	&=& \left\{(t,\omega):\left\{\begin{array}{l}t\in[s,u]\\ \omega=\omega^t\\\omega(r\wedge t)\in A
	\end{array}\right.\right\}   \\
	&=& \Lambda\bigcap\left\{(t,\omega):\left\{\begin{array}{l}t\in[s,u]\\ \omega(r\wedge t)\in A.
	\end{array}\right.\right\}
	\end{array}.
	\end{equation}
	We are left to show that $\left\{(t,\omega):\left\{\begin{array}{l}t\in[s,u]\\ \omega(r\wedge t)\in A
	\end{array}\right.\right\}\in\mathcal{P}ro^o$, or equivalently that
	\begin{equation} \label{EProgressive}
	t\mapsto \mathds{1}_{[s,u]}(t)\mathds{1}_A(X_{r\wedge t}) \text{ is }\mathbbm{F}^o-\text{progressively measurable}.
	\end{equation}
	Now $t \mapsto X_{r\wedge t}$ is right-continuous and $\mathbbm{F}^o$-adapted so it is an
	$E$-valued $\mathbbm{F}^o$-progressively measurable process, 
	see Theorem 15 in \cite{dellmeyer75} Chapter IV. By composition with a Borel function, $t\mapsto\mathds{1}_A(X_{r\wedge t})$ is a real-valued $\mathbbm{F}^o$-progressively measurable process; 
	\eqref{EProgressive} follows since 
	$t\mapsto\mathds{1}_{[s,u]}(t)$ is $\mathbbm{F}^o$-progressively measurable and
	the product 
	of the two   $\mathbbm{F}^o$-progressively measurable processes remains 
	$\mathbbm{F}^o$-progressively measurable.
\end{proof} 

\begin{definition} \label{D44}
	Let $(s,\eta)\in\Lambda$ and $\chi$ be a set of   $\mathbbm{F}^o$-adapted processes. We say that a probability measure
	$\mathbbm{P}$ on $(\Omega,\mathcal{F})$ \textbf{solves the martingale problem with respect to $\chi$ starting in $(s,\eta)$} if 
	\begin{itemize}
		\item $\mathbbm{P}(\omega^s=\eta^s)=1$,
		\item all elements 
		of $\chi$ are on $[s,+\infty[$ $(\mathbbm{P},\mathbbm{F}^o)$-martingales. 
	\end{itemize}
\end{definition}

% , finalement j'ai mis la tribu F^o dans le problème martingale et les processus ne sont plus forcément cadlag dans la def
\begin{remark}\label{cadlag} 
We insist on the following important fact. If $M\in\chi$ is cadlag and $\mathbbm{P}$ solves the martingale problem associated to $\chi$, then by Theorem 3 in \cite{dellmeyerB} Chapter VI,  $M$ is also on $[s,+\infty[$ a $(\mathbbm{P},\mathbbm{F})$-martingale.
\end{remark}

\begin{notation}
	For fixed $(s,\eta)\in\Lambda$ and $\chi$, the set of probability measures solving the martingale problem with respect to $\chi$ starting in $(s,\eta)$ will be denoted $MP^{s,\eta}(\chi)$.
\end{notation}

\begin{definition}\label{MP}
	Let us consider a set $\chi$  of processes. If for every $(s,\eta)\in\Lambda$, $MP^{s,\eta}(\chi)$ is reduced to a single element $\mathbbm{P}^{s,\eta}$, we will say that the martingale problem associated to $\chi$ is \textbf{well-posed}. In this case we will always extend the mapping 
	\begin{equation}
	\begin{array}{ccl}
	(s,\eta)&\longmapsto&\mathbbm{P}^{s,\eta}\\
	\Lambda&\longrightarrow&\mathcal{P}(\Omega)
	\end{array}
	\end{equation}
	to $\mathbbm{R}_+\times \Omega$ by setting for all $(s,\eta)\in\mathbbm{R}_+\times \Omega$,  $\mathbbm{P}^{s,\eta}:=\mathbbm{P}^{s,\eta^s}$.
\end{definition}

\begin{notation}\label{PiSystem}
	We fix a dense sequence $(x_n)_{n\geq 0}$ of elements of $E$.
	\\
	For any $s\in\mathbbm{R}_+$, we will denote by $\Pi_s$ the set of elements of $\mathcal{F}^o_s$ of type $\{\omega(t_1)\in B(x_{i_1},r_1),\cdots,\omega(t_N)\in B(x_{i_N},r_N)\}$ where $N\in\mathbbm{N}$, $t_1,\cdots,t_N\in[0,s]\cap\mathbbm{Q}$, $i_1,\cdots,i_N\in\mathbbm{N}$, $r_1,\cdots,r_N\in\mathbbm{Q}_+$ and where $B(x,r)$ denotes the open ball centered in $x$ and of radius $r$.
\end{notation}
It is easy to show that for any $s\in\mathbbm{R}_+$, $\Pi_s$ is a countable 
$\pi$-system generating $\mathcal{F}^o_s$, see \cite{aliprantis} Definition 4.9
for the notions of $\pi$-system and $\lambda$-system.

Below we consider the set $\mathcal{A}_s$  of probability measures $\mathbbm{P}$ on $(\Omega,\mathcal{F})$ for which there exists $\eta\in \Omega$ such that $\mathbbm{P}$ solves the martingale problem with respect to $\chi$ starting at
 $(s,\eta)$. 
\begin{proposition} \label{Llocbound}
	We fix a countable set $\chi$ of  cadlag $\mathbbm{F}^o$-adapted processes which are uniformly bounded on each interval $[0,T]$, and some $s\in\mathbbm{R}_+$.  Let $\mathcal{A}_s:=\underset{\eta\in\Omega}{\bigcup}MP^{s,\eta}(\chi)$.
	Then $\mathcal{A}_s$ is a Borel set of $\mathcal{P}(\Omega)$.
\end{proposition}
For the proof of this proposition we need a technical lemma.

\begin{lemma}\label{LemmaLlocbound}
	We fix $s\in\mathbbm{R}_+$. An element $\mathbbm{P}$ of $\mathcal{P}(\Omega)$ belongs to $\mathcal{A}_s$ if and only if it verifies the following conditions:
	\begin{enumerate}
		\item $\mathbbm{P}(F)\in\{0,1\}$ for all $F\in\Pi_s$;
		\item $\mathbbm{E}^{\mathbbm{P}}[(M_u-M_t)\mathds{1}_F]=0$ for all $M\in\chi$, $t,u\in[s,+\infty[\cap\mathbbm{Q}$ such that $t\leq u$, $F\in\Pi_t$.
	\end{enumerate}	
\end{lemma}

\begin{proof}
	%We fix $s\in\mathbbm{R}_+$.
	By definition of $\mathcal{A}_s$, an element $\mathbbm{P}$ of $\mathcal{P}(\Omega)$ belongs to $\mathcal{A}_s$ iff
	\begin{description}
		\item{a)} there exists $\eta\in\Omega$ such that $\mathbbm{P}(\omega^s=\eta^s)=1$;
		\item{b)} for all $M\in\chi$, $(M_t)_{t\in[s,+\infty[}$ is a $(\mathbbm{P},\mathbbm{F}^o)$-martingale.
	\end{description}

	Item a) above is equivalent to saying that $\mathcal{F}^o_s$ is $\mathbbm{P}$-trivial which is equivalent to item 1. of the Lemma's statement
	%saying that $\mathbbm{E}^{\mathbbm{P}}[\mathds{1}_{F}]\in\{0,1\}$ for all $F\in\Pi_s$, 
	by Dynkin's Lemma (see 4.11 in \cite{aliprantis}), since $\Pi_s$ is a $\pi$-system generating $\mathcal{F}^o_s$ and since the sets  $F\in\mathcal{F}^o_s$ such that $\mathbbm{P}(F)\in\{0,1\}$ form a $\lambda$-system. 
% see Definition 4.9 in \cite{aliprantis}.
	\\
	On the other hand, it is clear that item b) above implies item 2. in the statement of the Lemma. Conversely,  assume that $M\in \chi$ satisfies item 2. of the statement. We fix $s\leq t\leq u$.
	Let $(t_n)_n,(u_n)_n$ be two sequences of rational numbers which
	converge  to respectively to $t,u$ strictly from the right  
	and such that $t_n\leq u_n$ for all $n$. For every fixed $n$, we have $\mathbbm{E}^{\mathbbm{P}}[(M_{u_n}-M_{t_n})\mathds{1}_G]=0$ 
	for all $G\in \Pi_t$. 
	We then pass to the limit in $n$ using the fact that $M$ is right-continuous at fixed $\omega$, and the dominated convergence theorem and taking into account
	the fact that $M$ is 
	bounded on compact intervals; this yields $\mathbbm{E}^{\mathbbm{P}}[(M_u-M_t)\mathds{1}_G]=0$ for all $G\in \Pi_t$. Since sets $G\in\mathcal{F}^o_t$ verifying this property form a $\lambda$-system and since $\Pi_t$ is a $\pi$-system generating $\mathcal{F}^o_t$, then by Dynkin's lemma (see 4.11 in \cite{aliprantis}), $\mathbbm{E}^{\mathbbm{P}}[(M_u-M_t)\mathds{1}_G]=0$ for all $G\in \mathcal{F}^o_t$.
	% meaning that $\mathbbm{E}^{\mathbbm{P}}[(M_u-M_t)| \mathcal{F}^o_t]=0$ for all $t\leq u$, hence that
	This implies that   $(M_t)_{t\in[s,+\infty[}$ is a $(\mathbbm{P},\mathbbm{F}^o)$-martingale
%	. Applying Theorem 3 in \cite{dellmeyerB} Chapter VI, since $M$ is cadlag, it is also an $(\mathbbm{P},(\mathcal{F}_t)_{t\in[s,+\infty[})$-martingale. Therefore Item b) above is equivalent to Item 2 of the statement,
	which concludes the proof of Lemma \ref{LemmaLlocbound}.
\end{proof}

\begin{prooff}. of Proposition \ref{Llocbound}.
	\\
	We fix $s\in\mathbbm{R}_+$. We recall that for any bounded random variable 
	$\phi$, $\mathbbm{P}\mapsto\mathbbm{E}^{\mathbbm{P}}[\phi]$ is 
	Borel.
	In particular for all $F\in\Pi_s$, $\mathbbm{P}\longmapsto \mathbbm{P}(F)$ and for all $M\in\chi$, $t,u\in[s,+\infty[\cap\mathbbm{Q}$, $F\in\Pi_t$, $\mathbbm{P}\longmapsto\mathbbm{E}^{\mathbbm{P}}[(M_u-M_t)\mathds{1}_F]$ are Borel maps. The result follows
	by Lemma \ref{LemmaLlocbound}, taking into account the fact 
	$\Pi_t$ is countable for any $t$, and $\chi$ and the rational 
	number set $\mathbbm{Q}$ are also countable. Indeed since $\{0\}$ and $\{0,1\}$ are Borel sets,  $\mathcal{A}_s$ is Borel being a countable intersection of preimages of Borel sets by Borel functions.
\end{prooff}
\begin{proposition} \label{P210}
	Let $\chi$ be a countable set of cadlag
	$\mathbbm{F}^o$-adapted processes
	which are uniformly bounded on each interval $[0,T]$.
	We assume that the martingale problem associated to $\chi$ is well-posed, see Definition \ref{MP}.
	Let $s\in\mathbbm{R}_+$. Then
	$
	\Phi_s:\left(\begin{array}{ccc}
		\eta&\longmapsto&\mathbbm{P}^{s,\eta}\\
	\Omega^s &\longrightarrow & \mathcal{P}(\Omega)
	\end{array}\right)$ is Borel.
	Moreover,  $\left(\begin{array}{ccc}
	(s,\eta)&\longmapsto&\mathbbm{P}^{s,\eta}\\
	\mathbbm{R}_+\times \Omega &\longrightarrow & \mathcal{P}(\Omega)
	\end{array}\right)$ is $\mathbbm{F}^o$-adapted.
\end{proposition}
\begin{proof}
	We fix $s\in\mathbbm{R}_+$ and set 
	\begin{equation}
	\Phi_s:\begin{array}{ccl}
	\eta&\longmapsto& \mathbbm{P}^{s,\eta}\\
	\Omega^s&\longrightarrow&\mathcal{A}_s, 
	\end{array}
	\end{equation}
where $\mathcal{A}_s$ is defined as in Proposition \ref{Llocbound}. $\Phi_s$ is surjective by construction. It is also injective. Indeed, if $\eta_1,\eta_2\in\Omega^s$ are different, there exists $t\in[0,s]$ such that $\eta_1(t)\neq \eta_2(t)$ and we have  $\mathbbm{P}^{s,\eta_1}(\omega(t)=\eta_1(t))=1$ and $\mathbbm{P}^{s,\eta_2}(\omega(t)=\eta_2(t))=1$ so clearly $\mathbbm{P}^{s,\eta_1}\neq\mathbbm{P}^{s,\eta_2}$.
	
	We can therefore introduce the reciprocal mapping 
	\begin{equation}
	\Phi^{-1}_s:\begin{array}{ccl}
	\mathbbm{P}^{s,\eta}&\longmapsto& \eta\\
	\mathcal{A}_s&\longrightarrow&\Omega^s ,
	\end{array}
	\end{equation}
	which is a bijection. We wish to show  that it is Borel. Since the Borel $\sigma$-algebra of $\Omega^s$ is generated by the sets of type $\{\omega(r\wedge s)\in A\}$ where $r\in\mathbbm{R}_+$ and $A\in\mathcal{B}(E)$, it is enough to show that 
	$\Phi_s(\{\omega(r\wedge s)\in A\})$ is for any $r,A$ a Borel subset of $\mathcal{P}(\Omega)$. 
	We then have $\Phi_s(\{\omega(r\wedge s)\in A\})=\mathcal{A}_s\cap\{\mathbbm{P}:\mathbbm{P}(\omega(r\wedge s)\in A)=1\}$ which is Borel being
	the intersection of $\mathcal{A}_s$ which is Borel by Lemma 
	\ref{LemmaLlocbound}, and of the preimage of $\{1\}$ by the Borel function $\mathbbm{P}\mapsto \mathbbm{P}(F)$ with $F=\{\omega(r\wedge s)\in A\}$.
	So $\Phi^{-1}_s$ is a Borel bijection which maps the Borel 
	%(by the previous lemma) 
	set $\mathcal{A}_s$ of the Polish space $\mathcal{P}(\Omega)$ into the Polish space $\Omega^s$. By Kuratowsky theorem (see Corollary 3.3 in \cite{parthasarathy}), $\Phi_s:\begin{array}{ccl}
	\eta&\longmapsto&\mathbbm{P}^{s,\eta}\\
	\Omega^s&\longrightarrow&\mathcal{P}(\Omega)
	\end{array}$ is Borel. 

	Let us justify the second part of the statement.  Since by Proposition \ref{BorelLambda}, $\mathcal{B}(\Omega^s)=\Omega^s\cap\mathcal{F}^o_s$ for all $s$, it is clear that $\left(\begin{array}{ccl}
	\eta&\longmapsto&\eta^s\\
	\Omega&\longrightarrow&\Omega^s 
	\end{array}\right)$ is $(\mathcal{F}^o_s,\mathcal{B}(\Omega^s))$-measurable and therefore  that $\left(\begin{array}{ccc}
	\eta&\longmapsto&\mathbbm{P}^{s,\eta}\\
	\Omega &\longrightarrow & \mathcal{P}(\Omega)
	\end{array}\right)$ is $\mathcal{F}^o_s$-measurable.
	% Since this holds for any $s\in\mathbbm{R}_+$, the proof is complete.
	
\end{proof}

\begin{proposition}\label{MPimpliesCond}
	Let $\chi$ be a countable set of 
	cadlag $\mathbbm{F}^o$-adapted processes which are uniformly bounded on each interval $[0,T]$, and assume that the martingale problem associated to $\chi$ is well-posed, see Definition \ref{MP}.
	Then $(\mathbbm{P}^{s,\eta})_{(s,\eta)\in\mathbbm{R}_+\times \Omega}$ is a  path-dependent canonical class verifying Hypothesis \ref{HypClass} .
	
\end{proposition}

\begin{proof} 
	The first two items of Definition \ref{DefCondSyst} are directly implied by Proposition \ref{P210} and the fact that $\mathbbm{P}^{s,\eta}\in MP^{s,\eta}(\chi)$ hence $\mathbbm{P}^{s,\eta}(\omega^s=\eta^s)$ for all $(s,\eta)$.
	%	We are left to show the third item of Definition \ref{DefCondSyst}.
	It remains to show the validity of Hypothesis \ref{HypClass}. 
	
	We fix $(s,\eta)\in\mathbbm{R}_+\times \Omega$ and $t\geq s$. Since $\Omega$ is Polish and $\mathcal{F}_t$ is a sub $\sigma$-field of its Borel $\sigma$-field, there exists a regular conditional expectation of $\mathbbm{P}^{s,\eta}$ by $\mathcal{F}_t$ (see Theorem 1.1.6 in \cite{stroock}), meaning a set of probability measures $(\mathbbm{Q}^{t,\zeta})_{\zeta\in\Omega}$ 
	on $(\Omega,\mathcal{F})$  such that 
	\begin{enumerate}
		\item for any $F\in\mathcal{F}$, $\zeta\mapsto\mathbbm{Q}^{t,\zeta}(F)$ is $\mathcal{F}_t$-measurable;
		\item for any $F\in\mathcal{F}$, $\mathbbm{P}^{s,\eta}(F|\mathcal{F}_t)(\zeta)=\mathbbm{Q}^{t,\zeta}(F)$ $\mathbbm{P}^{s,\eta}$ a.s.
	\end{enumerate} 
	We will now show that for $\mathbbm{P}^{s,\eta}$ almost all $\zeta$, we have  
	\begin{equation} \label{EPQ}
	\mathbbm{Q}^{t,\zeta}=\mathbbm{P}^{t,\zeta},
	\end{equation}
	so that item 2. above will imply Hypothesis \ref{HypClass}. 
	%%% OLD Item 3 of Definition \ref{DefCondSyst}. 
	In order to show that equality, we will show that for $\mathbbm{P}^{s,\eta}$ almost all $\zeta$, 
	$\mathbbm{Q}^{t,\zeta}$ solves the Martingale problem associated to $\chi$ 
	starting in $(t,\zeta)$ and conclude 
	\eqref{EPQ} since 
 $MP^{t,\zeta}$
 is a singleton,
	taking into account the fact the corresponding martingale problem 
	is well-posed.
	
	For any $F\in\mathcal{F}^o_t$, by item 2. above we have  $\mathbbm{Q}^{t,\zeta}(F)=\mathds{1}_F(\zeta)$ $\mathbbm{P}^{s,\eta}$ a.s. Since $\Pi_t$ is countable, there exists a $\mathbbm{P}^{s,\eta}$-null set $N_1$ such that for all
 $\zeta\in 	N_1^c$ we have $\mathbbm{Q}^{t,\zeta}(F)=\mathds{1}_F(\zeta)$ for all $F\in\Pi_t$. Then since $\Pi_t$ is a $\pi$-system generating $\mathcal{F}^o_t$ and since sets verifying the previous relation define a $\lambda$-system, we have by Dynkin's lemma (see 4.11 in \cite{aliprantis}) that for all $\zeta\in
	N_1^c$, $\mathbbm{Q}^{t,\zeta}(F)=\mathds{1}_F(\zeta)$ for all $F\in\mathcal{F}^o_t$. Now for every fixed $\zeta\in N_1^c$, since $\{\omega:\omega^t=\zeta^t\}\in \mathcal{F}^o_t$, we have $\mathbbm{Q}^{t,\zeta}(\omega^t=\zeta^t)=\mathds{1}_{\{\omega:\omega^t=\zeta^t\}}(\zeta)=1$, which is the first item of 
Definition \ref{D44} related to $MP^{t,\zeta}(\chi)$.
	
	We then show that for  $\mathbbm{P}^{s,\eta}$-almost all $\zeta$, the elements of $\chi$ are $(\mathbbm{Q}^{t,\zeta},\mathbbm{F}^o)$-martingales, which constitutes the second item of Definition \ref{D44}. 
	\\
	For any $t_1\leq t_2$ in $[t,+\infty[$, $M\in\chi$ and $F\in\mathcal{F}^o_{t_1}$, we have 
\begin{equation}
\begin{array}{rcl}
\mathbbm{E}^{\mathbbm{Q}^{t,\zeta}}[(M_{t_2}-M_{t_1})\mathds{1}_F]&=&\mathbbm{E}^{s,\eta}[(M_{t_2}-M_{t_1})\mathds{1}_F|\mathcal{F}_{t}](\zeta)\\
&=&\mathbbm{E}^{s,\eta}[\mathbbm{E}^{s,\eta}[(M_{t_2}-M_{t_1})\mathds{1}_F|\mathcal{F}_{t_1}]|\mathcal{F}_{t}](\zeta)\\
&=&\mathbbm{E}^{s,\eta}[\mathbbm{E}^{s,\eta}[(M_{t_2}-M_{t_1})|\mathcal{F}_{t_1}]\mathds{1}_F|\mathcal{F}_{t}](\zeta)\\
&=&0,
\end{array}
\end{equation}	
 for  $\mathbbm{P}^{s,\eta}$ almost all $\zeta$  by Remark \ref{cadlag}
since $M$ is a $(\mathbbm{P}^{s,\eta},\mathbbm{F})$-martingale on$[s,+\infty[$ and $F\in\mathcal{F}^o_{t_1}\subset\mathcal{F}_{t_1}$. 
	Since $\chi$ and the set of rational numbers are countable  and taking into account the fact that
	for any $r\geq 0$, $\mathcal{F}^o_r$ is countably generated, there exists a $\mathbbm{P}^{s,\eta}$-null set $N_2$ such that for any $\zeta\in N_2^c$, we have for any $t_1\leq t_2$
	in $[t,+\infty[\cap\mathbbm{Q}$, $M\in\chi$, $F\in\mathcal{F}^o_{t_1}$, that $\mathbbm{E}^{\mathbbm{Q}^{t,\zeta}}[(M_{t_2}-M_{t_1})\mathds{1}_F]=0$.
	
	Let $\zeta \in N_2^c$. We will now show that this still holds for any  $t_1\leq t_2$ in $[t,+\infty[$, $M\in\chi$, $F\in\mathcal{F}^o_{t_1}$. 
	We consider rational valued sequences $(t_1^n)_n$ (resp. $(t_2^n)_n$) which converge to $t_1$ (resp. to $t_2$) strictly from the right and such that $t_1^n\leq t_2^n$ for all $n$. For all $n$, $\mathbbm{E}^{\mathbbm{Q}^{t,\zeta}}[(M_{t_2^n}-M_{t_1^n})\mathds{1}_F]=0$; since $M$ is right-continuous and bounded on finite intervals, by dominated convergence, 
	we can pass to the limit in $n$ and we obtain
	$\mathbbm{E}^{\mathbbm{Q}^{t,\zeta}}[(M_{t_2}-M_{t_1})\mathds{1}_F]=0$. 
	Therefore if 
	$\zeta\notin N_1\bigcup N_2$
	 which is $\mathbbm{P}^{s,\eta}$-negligible, then $\mathbbm{Q}^{t,\zeta}(\omega^t=\zeta^t)=1$ and  all the elements of $\chi$ are
	$(\mathbbm{Q}^{t,\zeta},\mathbbm{F}^o)$-martingales.
	This means that $\mathbbm{Q}^{t,\zeta}=\mathbbm{P}^{t,\zeta}$ by well-posedness and concludes the proof of Proposition \ref{MPimpliesCond}.
	% We deduce from this that for any $\zeta\in N^c$, $\mathbbm{Q}^{t,\zeta}\in MP^{t,\zeta}(\chi)$ implying by well-posedness that $\mathbbm{Q}^{t,\zeta}=\mathbbm{P}^{t,\zeta}$ for $\mathbbm{P}^{s,\eta}$ all $\zeta$.  
	% 	So the set $(\mathbbm{P}^{t,\zeta})_{\zeta\in\Omega}$ also verifies that for 
	% every $F\in\mathcal{F}$, $\mathbbm{P}^{s,\eta}(F|\mathcal{F}_t)(\zeta)=\mathbbm{P}^{t,\zeta}(F)$ $\mathbbm{P}^{s,\eta}$ a.s. Since this holds for any $s,\eta, t$ the proof is complete.
\end{proof}

\subsection{Martingale problem associated to an operator and weak
 generators}\label{S3_3}

This section links the notion of martingale problem with respect
to a natural notion of (weak) generator.  
In this section Notation \ref{canonicalspace} will be again in force. Let $(\mathbbm{P}^{s,\eta})_{(s,\eta)\in\mathbbm{R}_+\times \Omega}$ be a path-dependent canonical class and the corresponding path-dependent system of projectors  $(P_s)_{s\in\mathbbm{R}_+}$, see Definition \ref{ProbaOp}.   Let $V:\mathbbm{R}_+\longmapsto\mathbbm{R}_+$ be a non-decreasing cadlag function.

In the sequel of this section, we are given a couple $(\mathcal{D}(A),A)$ verifying the following.
\begin{hypothesis}\label{HypDA}
\begin{enumerate}\
\item $\mathcal{D}(A)$ is a linear subspace of the space of    $\mathbbm{F}^o$-progressively measurable processes;
\item $A$ is a linear mapping from $\mathcal{D}(A)$ into the space of  $\mathbbm{F}^o$-progressively measurable processes;
\item for all $\Phi\in\mathcal{D}(A)$, $\omega\in\Omega$, $t\geq 0$, $\int_0^t|A\Phi_r(\omega)|dV_r<+\infty$;
\item for all $\Phi\in\mathcal{D}(A)$, $(s,\eta)\in\mathbbm{R}_+\times \Omega$ and 
	$t\in[s,+\infty[$, we have  $\mathbbm{E}^{s,\eta}\left[\int_{s}^{t}|A(\Phi)_r|dV_r\right]<+\infty$ and $\mathbbm{E}^{s,\eta}[|\Phi_t|]<+\infty$.
\end{enumerate}
\end{hypothesis}

Inspired from the classical literature (see 13.28 in \cite{jacod}) we introduce a notion of weak  generator.
\begin{definition}\label{WeakGen}
	We say that $(\mathcal{D}(A),A)$ is a \textbf{weak generator} of a path-dependent system of projectors $(P_s)_{s\in\mathbbm{R}_+}$ if for all $\Phi\in\mathcal{D}(A)$, $(s,\eta)\in\mathbbm{R}_+\times \Omega$ and 
	$t\in[s,+\infty[$, we have
	\begin{equation}
	P_s[\Phi_t](\eta)=\Phi_s(\eta)+\int_s^tP_s[A(\Phi)_r](\eta)dV_r.
	\end{equation}
\end{definition}

\begin{definition}\label{MPop}
	We will call \textbf{martingale problem associated to } $(\mathcal{D}(A),A)$ 
	the martingale problem (in the sense of  Definition \ref{D44}) associated to
	the set of processes $\chi$ constituted by the processes 
	$\Phi-\int_0^{\cdot}A(\Phi)_rdV_r$, $\Phi\in\mathcal{D}(A)$. 
	It will be said to be  \textbf{well-posed} if it is well-posed in the sense of Definition \ref{MP}.
\end{definition}

\begin{proposition}\label{MPopWellPosed}
 $(\mathcal{D}(A),A)$ is a weak generator of  $(P_s)_{s\in\mathbbm{R}_+}$ iff $(\mathbbm{P}^{s,\eta})_{(s,\eta)\in\mathbbm{R}_+\times \Omega}$ solves the martingale problem associated to $(\mathcal{D}(A),A)$.
	
	Moreover, if $(\mathbbm{P}^{s,\eta})_{(s,\eta)\in\mathbbm{R}_+\times \Omega}$ solves the well-posed martingale problem associated to  $(\mathcal{D}(A),A)$ then $(P_s)_{s\in\mathbbm{R}_+}$ is the unique path-dependent system of projectors for which $(\mathcal{D}(A),A)$ is a weak generator.
\end{proposition}
\begin{proof}
We start assuming that $(\mathcal{D}(A),A)$ is a weak generator of $(P_s)_{s\in\mathbbm{R}_+}$.
Let $\Phi\in\mathcal{D}(A)$, $s\leq t\leq u$. $\mathbbm{P}^{s,\eta}$ a.s. 
	we have
	\begin{equation} \label{E436}
	\begin{array}{rcl}
	&&\mathbbm{E}^{s,\eta}[\Phi_u-\Phi_t-\int_t^u A(\Phi)_rdV_r|\mathcal{F}^o_t](\omega)\\
	&=&\mathbbm{E}^{t,\omega}[\Phi_u-\Phi_t-\int_t^u A(\Phi)_rdV_r]\\
	&=&P_t[\Phi_u](\omega)-\Phi_t(\omega)-\int_t^uP_t[A(\Phi_r)](\omega)dV_r\\
	&=&0,
	\end{array}
	\end{equation}
	where the first equality holds by Remark \ref{Borel},
 the second one by Fubini's theorem and the third one because
 $(\mathcal{D}(A),A)$ is assumed to be a weak generator of $(P_s)_{s\in\mathbbm{R}_+}$. By definition of path-dependent canonical class,
 we have $\mathbbm{P}^{s,\eta}(\omega^s=\eta^s)=1$. By \eqref{E436}, 
 for all $\Phi\in\mathcal{D}(A)$, $\Phi-\int_s^{\cdot} A(\Phi)_rdV_r$ is a $(\mathbbm{P}^{s,\eta},\mathbbm{F}^o)$-martingale, and therefore $\mathbbm{P}^{s,\eta}$ solves the martingale problem associated to $(\mathcal{D}(A),A)$ starting in $(s,\eta)$.

Conversely, let us assume that $(\mathbbm{P}^{s,\eta})_{(s,\eta)\in\mathbbm{R}_+\times \Omega}$ solves the martingale problem associated to $(\mathcal{D}(A),A)$.
%Since $\mathcal{D}^{count}$ is countable and $\Phi-\int_0^{\cdot}A(\Phi)_rdV_r$ is
%a cadlag $\mathbbm{F}^o$-adapted process, bounded on bounded intervals
% for all $\Phi\in\mathcal{D}^{count}$, then
% by  Proposition \ref{MPimpliesCond}, 
%  $(\mathbbm{P}^{s,\eta})_{(s,\eta)\in\mathbbm{R}_+\times \Omega}$ is a path-dependent canonical class.
%  
%  
Let $\Phi\in\mathcal{D}(A)$ and $(s,\eta)\in\mathbbm{R}_+\times \Omega$ be fixed. By Definitions \ref{MPop} and \ref{MP}, 
$M[\Phi]:=
\Phi-\int_{0}^{\cdot}A(\Phi)_rdV_r$, 
%which we denote $M[\Phi]$, 
is a $(\mathbbm{P}^{s,\eta},\mathbbm{F}^o)$-martingale on $[s,+\infty[$.
 Moreover, since $\mathbbm{P}^{s,\eta}(\omega^s=\eta^s)=1$ and being $\Phi_s$ 
 $\mathcal{F}^o_s$-measurable, 
we obtain $\Phi_s=\Phi_s(\eta)$ $\mathbbm{P}^{s,\eta}$ a.s. Therefore, for any $t\geq s$,  $\Phi_{t}-\Phi_s(\eta)-\int_{s}^{t}A(\Phi)_rdV_r=M[\Phi]_t-M[\Phi]_s$ a.s.; so taking the $\mathbbm{P}^{s,\eta}$ expectation, by Fubini's Theorem and Definition \ref{ProbaOp} it yields
	\begin{equation}
		\begin{array}{rcl}
		&&P_s[\Phi_t](\eta)-\Phi_s(\eta)-\int_s^tP_s[A(\Phi)_r](\eta)dV_r\\
		&=& \mathbbm{E}^{s,\eta}\left[\Phi_{t}-\Phi_s(\eta)-\int_{s}^{t}A(\Phi)_rdV_r\right]\\
		&=& \mathbbm{E}^{s,\eta}\left[M[\Phi]_t-M[\Phi]_s\right]\\
		&=&0,
		\end{array}
	\end{equation}
	hence that $(\mathcal{D}(A),A)$ is a weak generator of $(P_s)_{s\in\mathbbm{R}_+}$.

Finally assume moreover that the martingale problem is well-posed and that $(\mathcal{D}(A),A)$ is a weak generator of another path-dependent system of projectors $(Q_s)_{s\in\mathbbm{R}_+}$ with associated path-dependent canonical class $(\mathbbm{Q}^{s,\eta})_{(s,\eta)\in\mathbbm{R}_+\times \Omega}$. Then by the first statement of the present proposition, $(\mathbbm{Q}^{s,\eta})_{(s,\eta)\in\mathbbm{R}_+\times \Omega}$ solves the martingale problem associated to  $(\mathcal{D}(A),A)$.
 Since that martingale problem is well-posed we have
 $(\mathbbm{Q}^{s,\eta})_{(s,\eta)\in\mathbbm{R}_+\times \Omega}=(\mathbbm{P}^{s,\eta})_{(s,\eta)\in\mathbbm{R}_+\times \Omega}$ and by Proposition \ref{EqProbaOp}, $(Q_s)_{s\in\mathbbm{R}_+}=(P_s)_{s\in\mathbbm{R}_+}$.
\end{proof}

\begin{remark}\label{OpGen}
When the conditions of  previous proposition are verified, one can therefore associate analytically  to $(\mathcal{D}(A),A)$ a unique path-dependent system of projectors $(P_s)_{s\in\mathbbm{R}_+}$ through Definition \ref{WeakGen}.
\end{remark}

Combining Proposition \ref{MPopWellPosed} and Lemma \ref{LemMAF} yields the following.
%%%
\begin{corollary}
Assume that $(\mathbbm{P}^{s,\eta})_{(s,\eta)\in\mathbbm{R}_+\times \Omega}$ is
 progressive and fulfills Hypothesis \ref{HypClass}. 
Suppose that $(\mathcal{D}(A), A)$ is a weak generator of 
$(P_s)_{s\in\mathbbm{R}_+}$. Let $\Phi\in\mathcal{D}(A)$,
 and fix $(s,\eta)$. Then 
 $\Phi-\int_0^{\cdot}A(\Phi)_rdV_r$ admits on $[s,+\infty[$ a
 $\mathbbm{P}^{s,\eta}$ version $M[\Phi]^{s,\eta}$ which is a 
$(\mathbbm{P}^{s,\eta},\mathbbm{F}^{s,\eta})$-cadlag martingale. 
In particular, the random field defined by $M[\Phi]_{t,u}(\omega):=\Phi_u(\omega)-\Phi_t(\omega)-\int_t^uA\Phi_r(\omega)dV_r$ defines a MAF with cadlag version 
$M[\Phi]^{s,\eta}$ under $\mathbbm{P}^{s,\eta}$.
\end{corollary}
We insist on the fact that in previous corollary, $\Phi$ is not necessarily cadlag.
That result will be crucial in the companion paper \cite{paperPathDep}.

\section{Weak solutions of path-dependent SDEs}\label{SDE}

We will now focus on a more specific type of martingale problem which will be associated to a path-dependent Stochastic Differential Equation with jumps. In this section we will refer to notions of \cite{jacod} Chapters II, III, VI and \cite{jacod79} Chapter XIV.5.

We fix $m\in\mathbbm{N}^*$, $E=\mathbbm{R}^m$, the associated canonical space, see Definition \ref{canonicalspace}, and a finite positive measure $F$ on 
%$(\mathbbm{R}^m,
$\mathcal{B}(\mathbbm{R}^m)$ not charging $0$.

\begin{definition}\label{WeakSol}
	$(\tilde \Omega,\tilde{\mathcal{F}},\tilde{\mathbbm{F}},\tilde{\mathbbm{P}},W,p)$  will be called a \textbf{space of driving processes} if $(\tilde \Omega,\tilde{\mathcal{F}},\tilde{\mathbbm{F}},\tilde{\mathbbm{P}})$ 
  is a  stochastic basis fulfilling the usual conditions, $W$ is an $m$-dimensional Brownian motion and $p$ is a Poisson measure of intensity 
	$q(dt,dx):=dt\otimes F(dx)$, and $W,p$ are optional for the underlying filtration.
\end{definition}
We now fix the following objects defined on the canonical space.
\begin{itemize}
	\item $\beta$, an $\mathbbm{R}^m$-valued $\mathbbm{F}^o$-predictable process;
	\item $\sigma$, a $\mathbbm{M}_m(\mathbbm{R})$-valued $\mathbbm{F}^o$-predictable process;
	\item $w$, an $\mathbbm{R}^m$-valued  $\mathcal{P}re^o\otimes\mathcal{B}(\mathbbm{R}^m)$-measurable function on $\mathbbm{R}_+\times\Omega\times\mathbbm{R}^m$,
\end{itemize}
where $\mathbbm{M}_m(\mathbbm{R})$ denotes the set of real-valued square matrices of size $m$.

\begin{definition} \label{D213}
	Let $(s,\eta)\in\mathbbm{R}_+\times\Omega$.
	We call a \textbf{weak solution of the SDE with coefficients $\beta$, $\sigma$, $w$ and starting in $(s,\eta)$} any probability measure $\mathbbm{P}^{s,\eta}$ on $(\Omega,\mathcal{F})$ such that there exists a space of driving processes 	$(\tilde \Omega,\tilde{\mathcal{F}},\tilde{\mathbbm{F}},\tilde{\mathbbm{P}}
	,W,p)$, on it an $m$-dimensional   $\tilde{\mathbbm{F}}$-adapted cadlag process $\tilde X$ such that $\mathbbm{P}^{s,\eta}=\tilde{\mathbbm{P}}\circ \tilde X^{-1}$ and such that the following holds.
	%Let $\Psi:\tilde \Omega\longrightarrow\Omega$, $\tilde \beta$, $\tilde \sigma$, $\tilde w$ be defined by $\Psi(\tilde \omega) =\tilde X(\tilde \omega)$,  $\tilde \beta_t(\tilde \omega):=\beta_t(\Psi(\tilde\omega))$, $\tilde \sigma_t(\tilde\omega):=\sigma_t(\Psi(\tilde\omega))$, $\tilde w(\tilde\omega,t,x):=w(\Psi(\tilde\omega),t,x)$, then 
	
	Let $\tilde \beta:=\beta_{\cdot}(\tilde X(\cdot))$, $\tilde \sigma:=\sigma_{\cdot}(\tilde X(\cdot))$ and $\tilde w:=w(\cdot,\tilde X(\cdot),\cdot)$. 
We have the following. 
	\begin{itemize}
		\item for all $t\in[0,s]$, $\tilde X_t=\eta(t)$ $\tilde{\mathbbm{P}}$ a.s.;
		%		\item $\int_s^{\cdot}\left(\|\beta_r\|+\|\sigma_r\|^2+\int_EF(dx)(\|\tilde w\|^2\mathds{1}_{\|\tilde w\|\leq 1}+\mathds{1}_{\|\tilde w\|> 1})\right)dr$ takes finite values $\mathbbm{P}$ a.s.;
		\item $\int_s^{\cdot}\left(\|\tilde{\beta}_r\|+\|\tilde{\sigma}_r\|^2+\int_{\mathbbm{R}^m}(\|\tilde w(r,\cdot,y)\|+\|\tilde w(r,\cdot,y)\|^2)F(dy)\right)dr$ takes finite values $\tilde{\mathbbm{P}}$ a.s.;
		%		\item $\tilde X^i_t=\eta(s)+\int_s^t\tilde \beta^i_rdr+\underset{j\leq m}{\sum}\int_s^t\tilde \sigma^{i,j}_rdW^j_r+\tilde w^i\mathds{1}_{\|\tilde w\|\leq 1}\star(p-q)_t+\tilde w^i\mathds{1}_{\|\tilde w\|> 1}\star p_t$ $\mathbbm{P}$ a.s. for all $t\geq s$, $i\leq m$,
		\item $\tilde X^i_t=\eta_i(s)+\int_s^t\tilde \beta^i_rdr+\underset{j\leq m}{\sum}\int_s^t\tilde \sigma^{i,j}_rdW^j_r+\tilde w^i\star(p-q)_t$ $\tilde{\mathbbm{P}}$ a.s. for all $t\geq s$, $i\leq m$,
	\end{itemize}
	where $\star$ is the integration against random measures, see \cite{jacod} Chapter II.2.d for instance.
\end{definition}

\begin{remark}
	Previous Definition \ref{D213} corresponds to Definition 14.73 in \cite{jacod79}. However, in the second item we have required that 
$$\int_s^{\cdot}\int_{\mathbbm{R}^m}(\|\tilde w(r,\cdot,y)\|+\|\tilde w(r,\cdot,y)\|^2)F(dy)dr$$
 takes finite values a.s. so that $\tilde w\star(p-q)$ is a well-defined purely discontinuous locally square integrable martingale with angle bracket
	the $\mathbbm{M}_m(\mathbbm{R})$-valued process
	$\int_s^{\cdot\vee s}\int_{\mathbbm{R}^m} \tilde w\tilde w^{\intercal}(r,\cdot,y)F(dy)dr,$ (see Definition 1.27, Proposition 1.28 and Theorem 1.33 in \cite{jacod79} chapter II) and we will not need to use any truncation function.
	
	With this definition, if $\mathbbm{P}^{s,\eta}$ is a weak solution of the SDE starting at some $(s,\eta)$, then under $\mathbbm{P}^{s,\eta}$, $(X_t)_{t\geq s}$ is a special semimartingale.
\end{remark}

\begin{definition}\label{Char}
	Let $s\in\mathbbm{R}_+$ and $(Y_t)_{t\geq s}$ be a cadlag special semimartingale defined on the canonical space with (unique) decomposition $Y=Y_s+B+M^c+M^d$ where $B$ is predictable with bounded variation, $M^c$ a continuous local martingale, $M^d$ a purely discontinuous local martingale, all three 
	vanishing at the initial time $t=s$.
	  We will call \textbf{characteristics} of $Y$ the triplet $(B,C,\nu)$ where $C=\langle M^c\rangle$ and $\nu$ is the predictable compensator of the measure of the jumps of $Y$.
\end{definition}

There are several known equivalent characterizations of weak solutions of path-dependent SDEs with jumps which we will now state in our setup.

\begin{notation}\label{NotA}
	For every $f\in\mathcal{C}^2_b(\mathbbm{R}^m)$ and $t\geq 0$, we denote by $A_tf$ the r.v.
	\begin{equation}
	\beta_t\cdot \nabla f(X_t)+\frac{1}{2}Tr(\sigma_t\sigma_t^{\intercal}\nabla^2f(X_t))+\int_{\mathbbm{R}^m}(f(X_t+w(t,\cdot,y))-f(X_t)-\nabla f(X_t)\cdot w(t,\cdot,y))F(dy).
	\end{equation}
\end{notation}

\begin{proposition}\label{SDEeq}
	Let $(s,\eta)\in\mathbbm{R}_+\times\Omega$ be fixed and let $\mathbbm{P}\in\mathcal{P}(\Omega)$. There is equivalence between the  following properties.
	\begin{enumerate}
		\item $\mathbbm{P}$ is a weak solution of the SDE with coefficients $\beta,\sigma,w$;
		\item $\mathbbm{P}(\omega^s=\eta^s)=1$ and $(X_t)_{t\geq s}$ is under $\mathbbm{P}$ a special semimartingale with characteristics 
		\begin{itemize}
		\item $B=\int_s^{\cdot}\beta_rdr$;
		\item $C=\int_s^{\cdot}(\sigma\sigma^{\intercal})_rdr$;
		\item $\nu:(\omega,G)\mapsto\int_s^{+\infty}\int_E\ \mathds{1}_G(r,w(\omega,r,y))\mathds{1}_{\{w(\omega,r,y)\neq 0\}}F(dy)dr$;
		\end{itemize}
		%		\item  $\mathbbm{P}$ solves $MP^{s,\eta}(\chi)$ where $\chi$ is constituted of processes $X^i_{\cdot}-\int_s^{\cdot}\beta^i_rdr$ for all $1\leq i\leq m$ and $X^i_{\cdot}X^j_{\cdot}-\int_s^{\cdot}(\sigma\sigma^T)^{i,j}_rdr$ for all $1\leq i,j\leq m$;\\
		\item  $\mathbbm{P}$ solves $MP^{s,\eta}(\chi)$ where $\chi$ is constituted of processes $f(X_{\cdot})-\int_0^{\cdot}A_rfdr$ for all $f\in\mathcal{C}^2_b(\mathbbm{R}^m)$.
		\item  $\mathbbm{P}$ solves $MP^{s,\eta}(\chi')$ where $\chi'$ is constituted of processes $f(X_{\cdot})-\int_0^{\cdot}A_rfdr$ for all functions $f:x\mapsto cos(\theta\cdot x)$ and $f:x\mapsto sin(\theta\cdot x)$ with $\theta\in\mathbbm{Q}^m$.
	\end{enumerate}
\end{proposition}
\begin{proof}
	Equivalence between items 1. and 2. is a consequence of Theorem 14.80 in \cite{jacod79}. The equivalence between items 2., 3. and 4. if $\theta$ was ranging in $\mathbbm{R}^m$ is shown in Theorem 2.42 of \cite{jacod} chapter II. 
Observe that 4. is stated for $\theta \in \mathbbm{R}^m$; however
	the proof of the implication $(4.\Longrightarrow 2.)$ in  Theorem 2.42 of \cite{jacod} chapter II
only uses the values of $\theta$ in $\mathbbm{Q}^m$.
\end{proof}

\begin{theorem}\label{SDEcond}
	Assume  that for any $(s,\eta)\in\mathbbm{R}_+\times\Omega$, the SDE with coefficients $\beta$, $\sigma$, $w$ and starting in $(s,\eta)$ admits a unique weak solution $\mathbbm{P}^{s,\eta}$. Then $(\mathbbm{P}^{s,\eta})_{(s,\eta)\in\mathbbm{R}_+\times \Omega}$ is a path-dependent canonical class verifying Hypothesis \ref{HypClass}.
\end{theorem}
\begin{proof}
	By Proposition \ref{SDEeq}, $\mathbbm{P}^{s,\eta}$ is for each $(s,\eta)$ the unique solution of $MP^{s,\eta}(\chi)$ where $\chi$ is constituted of the processes $f(X_{\cdot})-\int_s^{\cdot}A_rfdr$ for all functions $f:x\mapsto cos(\theta\cdot x)$ or $f:x\mapsto sin(\theta\cdot x)$ with $\theta\in\mathbbm{Q}^m$. Since $\chi$ is a countable set of cadlag $\mathbbm{F}^o$-adapted processes which are bounded on bounded intervals, we can conclude by Proposition \ref{MPimpliesCond}.
\end{proof}
We recall two classical examples of conditions on the coefficients for which it is known that there is existence
 and uniqueness of a weak solution for the path-dependent SDE, hence for which the above theorem applies, 
see Theorem 14.95 and Corollary 14.82 in \cite{jacod79}.
\begin{example}
 We suppose
 $\beta,\sigma,w$ to be bounded. Moreover we suppose that
  for all $n\in\mathbbm{N}^*$ there exist
%some functions $K^n_1,
$K^n_2\in L^1_{loc}(\mathbbm{R}_+)$ and
  a Borel function
$K^n_3: \mathbbm{R}^m \times \mathbbm{R}_+ \rightarrow \mathbbm{R} $ 
 such that $\int_{\mathbbm{R}^m} K^n_3(\cdot,y)F(dy)\in L^1_{loc}(\mathbbm{R}_+)$ verifying the following.

For all $x\in\mathbbm{R}^m$, $t\geq 0$ and $\omega,\omega'\in\Omega$ such that $\underset{r\leq t}{\text{sup }}\|\omega(r)\|\leq n$ and $\underset{r\leq t}{\text{sup }}\|\omega'(r)\|\leq n$, we have 
%	\item $\|\beta_t(\omega)-\beta_t(\omega')\|\leq K^n_1(t)\underset{r\leq t}{\text{sup }}
%\|\omega(r)-\omega'(r)\|$;
\begin{itemize}
	\item  $\|\sigma_t(\omega)-\sigma_t(\omega')\|\leq K^n_2(t)\underset{r\leq t}{\text{sup }}\|\omega(r)-\omega'(r)\|^2$;
	\item 
$\|w(t,\omega,x)-w(t,\omega',x)\|\leq K^n_3(t,x)\underset{r\leq t}{\text{sup }}\|\omega(r)-\omega'(r)\|^2$.
\end{itemize}
Finally we suppose that one of the two following hypotheses is fulfilled.
\begin{enumerate}
\item There exists $K^n_1 \in L^1_{loc}(\mathbbm{R}_+)$ such that 
for %all $x\in\mathbbm{R}^m$, 
all $t\geq 0$ and $\omega\in\Omega$, 
 $\|\beta_t(\omega)-\beta_t(\omega')\|\leq K^n_1(t)\underset{r\leq t}{\text{sup }}\|\omega(r)-\omega'(r)\|$;
%\begin{itemize}
%\item $w$ verifies the conditions of the example in item 1.;
\item  there exists $c>0$ such that for 
all $x\in\mathbbm{R}^m$, $t\geq 0$ and $\omega\in\Omega$, $x^{\intercal}\sigma_t(\omega)\sigma_t(\omega)^{\intercal}x\geq c\|x\|^2$;
%\item  $\beta$ is only required to be bounded.

%\end{itemize}

\end{enumerate}
\end{example}

If the assumptions of Theorem \ref{SDEcond} are fulfilled
and $\beta,\sigma$ (resp. $w$) are bounded and continuous in $\omega$ for fixed $t$ (resp. fixed $t,y$), then $(s,\eta)\longmapsto \mathbbm{P}^{s,\eta}$ is continuous for the topology of weak convergence, and in particular, the path-dependent canonical class is progressive hence all results of Section \ref{SectionMAF}
can be applied with respect to  $(\mathbbm{P}^{s,\eta})_{(s,\eta) \in 
 \mathbbm{R}_+ \times \Omega}.$

\begin{proposition}\label{tight} 
	Assume that  that $\beta,\sigma,w$ are bounded. Let $(s_n,\eta_n)_n$ be a sequence in $\Lambda$ which converges to some $(s,\eta)$. For every $n\in\mathbbm{N}$, let $\mathbbm{P}^n$ be a weak solution starting in $(s_n,\eta_n)$ of the SDE with coefficients $\beta,\sigma,w$. Then $(\mathbbm{P}^n)_{n\geq 0}$ is tight.
\end{proposition}

We recall some notations from \cite{jacod} Chapter VI which we will use in the 
proof of Proposition \ref{tight}.
\begin{notation}\label{Modulus}
	For any $\omega\in\Omega$ and interval $\mathcal{I}$ of $\mathbbm{R}_+$, we denote
	$W(\omega,\mathcal{I}):= \underset{s,t\in\mathcal{I}}{\text{\rm sup }}\|\omega(t)-\omega(s)\|$.
	For any $\omega\in\Omega$, $N\in\mathbbm{N}^\star$ and $\theta>0$, we
write
	\\
	$W_N(\omega,\theta):= \underset{0\leq t\leq t+\theta\leq N}{\text{\rm sup }}W(\omega,[t,t+\theta])=\underset{s,t\in[0,N]:\,|t-s|\leq \theta}{\text{\rm sup }}\|\omega(t)-\omega(s)\|$.
	\\
	For any $\omega\in\Omega$, $N\in\mathbbm{N}^\star$ and $\theta>0$, we denote
	\\
	$W'_N(\omega,\theta):= \text{\rm inf }\left\{\underset{i\leq r}{\text{\rm max } }W(\omega,[t_{i-1},t_i[):\quad 0=t_0<\cdots<t_r=N;\quad \forall 1\leq i\leq r: t_i-t_{i-1}\geq \theta \right\}$.
\end{notation}
%We recall that by Lemma 1.12 in \cite{jacod} Chapter VI, we have the following equivalent characterization of $W'_N(\omega,\theta)$ which is the one that we will need. 
%\begin{lemma}\label{Ltechproof}
%	For any $N,\omega,\theta$, we have 
%	\\
%	$W'_N(\omega,\theta)= \text{\rm inf }\left\{\underset{i\leq r}{\text{\rm max } }W(\omega,[t_{i-1},t_i[):\quad 0=t_0<\cdots<t_r=N;\quad \forall 1\leq i\leq r: \theta\leq t_i-t_{i-1}\leq 2\theta \right\}$.
%\end{lemma}	
We will also recall the classical general tightness criterion in $\mathcal{P}(\Omega)$ which one can find for example in Theorem 3.21 of \cite{jacod} Chapter VI.
\begin{theorem}\label{Tightness}
	Let $(\mathbbm{P}^n)_{n\geq 0}$ be a sequence of elements of $\mathcal{P}(\Omega)$, then it is tight iff it verifies the two following conditions.
	\begin{equation}
	\left\{
	\begin{array}{l}
	\forall N\in\mathbbm{N}^*\quad \forall \epsilon>0\quad \exists K>0\quad \forall n\in\mathbbm{N}:\quad \mathbbm{P}^n\left(\underset{t\leq N}{\text{\rm sup }}\|\omega(t)\|>K\right)\leq \epsilon\\
	\forall N\in\mathbbm{N}^*\quad \forall \epsilon>0 \quad \forall \alpha>0\quad \exists \theta\quad \forall n\in\mathbbm{N}:\quad \mathbbm{P}^n(W'_N(\omega,\theta)< \alpha)\geq 1-\epsilon.
	\end{array}
	\right.
	\end{equation}
\end{theorem}

Finally we will also need to introduce a definition.
\begin{definition}
	A sequence of probability measures on $(\Omega,\mathcal{F})$ is called $\mathcal{C}$-tight if it is tight and if each of its limiting points
	has all its support in $\mathcal{C}(\mathbbm{R}_+,\mathbbm{R}^m)$.
\end{definition}

\begin{prooff}.of Proposition \ref{tight}.\\
	We fix a converging sequence $(s_n,\eta_n)\underset{n}{\longrightarrow}(s,\eta)$ in $\Lambda$, and for every $n$, a weak solution $\mathbbm{P}^n$ of the SDE with coefficients $\beta,\sigma, w$ starting in $(s_n,\eta_n)$. In order to show that $(\mathbbm{P}^n)_{n\geq 0}$ is tight, we will use Theorem \ref{Tightness}.
	The main idea consists in combining the fact that 
	the canonical process $X$ under $\mathbbm{P}^n$ is deterministic on $[0,s_n]$, where it coincides 
	with $\eta_n$ with the fact that on $[s_n,+\infty[$
	it is a semimartingale with known characteristics. 
	So we will split the study of the modulus of continuity of path $\omega$  on these two intervals $[0,s_n]$ and $[s_n,+\infty[$.
	
	Since $\eta_n$ tends to $\eta$, the set $\{\eta_n:n\geq 0\}$ is relatively compact in $\Omega$ so by Theorem 1.14.b in \cite{jacod} Chapter VI we have
	%	clearly $(\delta_{\eta_n})$ is tight in $\mathcal{P}(\Omega)$ since all these measures ave their support in the previously mentioned relatively compact set. Applying Theorem 3.21 of \cite{jacod} Chapter VI and using its notations, this implies that
	\begin{equation}\label{EqTight}
	\left\{
	\begin{array}{l}
	\forall N\in\mathbbm{N}^*\quad\exists K_1>0\quad\forall n\in\mathbbm{N}:\quad\underset{t\in[0,N]}{\text{ sup }}\|\eta_n(t)\|\leq K_1\\
	\forall N\in\mathbbm{N}^*\quad\forall \alpha >0\quad\exists \theta_1>0\quad\forall n\in\mathbbm{N}:\quad W'_N(\eta_n,\theta_1)<\alpha.
	\end{array}
	\right.
	\end{equation}
	For fixed $n\in\mathbbm{N}$, we now introduce  the process
	\\
	$X^n:\omega\longmapsto\eta_n(s_n)\mathds{1}_{[0,s_n[}+\omega\mathds{1}_{[s_n,+\infty[}$, we denote by $\mathbbm{Q}^n:=\mathbbm{P}^n\circ(X^n)^{-1}\in\mathcal{P}(\Omega)$ its law under $\mathbbm{P}^n$ and we now show that $(\mathbbm{Q}^n)_{n\geq 0}$ is tight.
	
	By Proposition \ref{SDEeq}, under 
	$\mathbbm{P}^n$ ,
	$(X_t)_{t\in[s_n,+\infty[}$ is a special semimartingale  with initial value $\eta_n(s_n)$ and characteristics (see Definition \ref{Char}) $\int_{s_n}^{\cdot}\beta_rdr$,  $\int_{s_n}^{\cdot}(\sigma\sigma^{\intercal})_rdr$ and 
	$(\omega,A)\mapsto\int_{s_n}^{+\infty}\int_{\mathbbm{R}^m} \mathds{1}_A(r,w(r,\omega,y))\mathds{1}_{\{w(r,\omega,y)\neq 0\}}F(dy)dr$.
	Therefore, since $X^n$ is constant on $[0,s_n[$ and since on $[s_n,+\infty[$  its law under $\mathbbm{P}^n$ coincides with the one of $X$, we can say that  $\mathbbm{Q}^n$ is the law of a special semi-martingale (starting at time $t=0$) 
	with initial value $\eta_n(s_n)$, and 
	characteristics $\int_0^{\cdot}\mathds{1}_{[s_n,+\infty[}(r)\beta_rdr$,  $\int_{0}^{\cdot}\mathds{1}_{[s_n,+\infty[}(r)(\sigma\sigma^{\intercal})_rdr$ and \\
	$(\omega,G)\mapsto\int_0^{+\infty}\mathds{1}_{[s_n,+\infty[}(r)\int_{\mathbbm{R}^m} \mathds{1}_G(r,w(r,\omega,y))\mathds{1}_{\{w(r,\omega,y)\neq 0\}}F(dy)dr$.
	\\
	Theorem 4.18 in \cite{jacod} chapter VI implies that $(\mathbbm{Q}^n)_{n\geq 0}$ is tight if and only if the properties below hold true.
	\begin{enumerate}
		\item $(\mathbbm{Q}^n\circ X_0^{-1})_{n\geq 0}$ is tight;
		\item the following sequences are $\mathcal{C}$-tight (under $(\mathbbm{Q}^n)_{n\geq 0}$):
		\begin{enumerate}
			\item $(B^n:=\int_0^{\cdot}\mathds{1}_{[s_n,+\infty[}(r)\beta_rdr)_{n\geq 0}$;
			\item $\left(\tilde C^n:=\int_{0}^{\cdot}\mathds{1}_{[s_n,+\infty[}(r)\left((\sigma\sigma^{\intercal})_r+\int_{\mathbbm{R}^m} (ww^{\intercal})(r,\cdot,y)F(dy)\right)dr\right)_{n\geq 0}$;
			\item $\left(G^n_p:=\int_{0}^{\cdot}\mathds{1}_{[s_n,+\infty[}(r) \int_{\mathbbm{R}^m} \mathds{1}_{\{w(r,\omega,y)\neq 0\}}((p\|w(r,\cdot,y)\|-1)^+)\wedge 1F(dy)dr\right)_{n\geq 0}$
			\\
			for all $p\in\mathbbm{N}$;
		\end{enumerate}
		\item for all $N>0$, $\epsilon>0$,
		\begin{equation}
		\underset{a\rightarrow\infty}{\rm lim}\underset{n}{\rm sup}\,\mathbbm{Q}^n\left( \int_{s_n}^{N} \int_{\mathbbm{R}^m} \mathds{1}_{\{\|w(r,\cdot,y)\|>a\}}F(dy)dr >\epsilon\right)=0.
		\end{equation}
	\end{enumerate}
	Item 3. trivially holds since $w$ is bounded.
	At this point $\eta_n(s_n)$ is a
	bounded sequence according to the first line of \eqref{EqTight} and the fact that the sequence
	$(s_n)_{n\geq 0}$ is bounded, so $(\mathbbm{Q}^n\circ X_0^{-1})_{n\geq 0}=(\delta_{\eta_n(s_n)})_{n\geq 0}$ is obviously tight.
	We are left to show item 2.  By Proposition 3.36 in \cite{jacod} chapter VI, items 2. (a) and 2. (b) hold if 
	$(Var(B^n))_{n\geq 0}=(\int_0^{\cdot}\mathds{1}_{[s_n,+\infty[}(r)\|\beta_r\|dr)_{n\geq 0}$ and 
	\\
$(Tr(\tilde C^n))_{n\geq 0}=\left(\int_{0}^{\cdot}\mathds{1}_{[s_n,+\infty[}(r)\left(Tr(\sigma\sigma^{\intercal}_r)+\int_{\mathbbm{R}^m} Tr(ww^{\intercal}(r,\cdot,y))F(dy)\right)dr\right)_{n\geq 0}$ are $\mathcal{C}$-tight. Finally, $\beta,\sigma,w,F$ being bounded, there exists some strictly positive constant $K$ such that all the processes given below are increasing: 
	\begin{itemize}
	\item $t\mapsto Kt - Var(B^n)_t,\quad n\geq 0$;
	\item $t\mapsto Kt - Tr(\tilde C^n_t),\quad n\geq 0$;
		\item $t\mapsto Kt - (G^n_p)_t,\quad n,p\geq 0$.
	\end{itemize} 
In the terminology of \cite{jacod} chapter VI,
this means that the increasing processes 
$Var(B^n), \quad n\geq 0$,
$Tr(\tilde C^n),\quad n\geq 0$, $G^n_p\quad n,p\geq 0$
are strongly dominated by the increasing function
$t \mapsto Kt$.
The singleton $t\mapsto Kt$ being trivially $\mathcal{C}$-tight,
  Proposition 3.35 in \cite{jacod} chapter VI implies that 
the dominated sequences of processes $(Var(B^n))_{n\geq 0}$, $(Tr(\tilde C^n))_{n\geq 0}$ and $(G^n_p)_{n\geq 0}$ for all $p$ are  $\mathcal{C}$-tight. Finally $(\mathbbm{Q}^n)_{n\geq 0}$ is tight.

	Now by Theorem \ref{Tightness}
	this implies that \\
	\begin{equation}\label{EqTight3}
	\left\{
	\begin{array}{l}
	\forall N\in\mathbbm{N}^*\quad \forall \epsilon>0\quad \exists K_2>0\quad \forall n\in\mathbbm{N}:\quad \mathbbm{Q}^n\left(\underset{t\leq N}{\text{sup }}\|\omega(t)\|>K_2\right)\leq \epsilon\\
	\forall N\in\mathbbm{N}^*\quad \forall \epsilon>0 \quad \forall \alpha>0\quad \exists \theta_2\quad \forall n\in\mathbbm{N}:\quad \mathbbm{Q}^n(W_N'(\omega,\theta_2)< \alpha)\geq 1-\epsilon.
	\end{array}
	\right.
	\end{equation}
	Combining the first line of \eqref{EqTight} and the first line of \eqref{EqTight3} and by construction of $\mathbbm{Q}^n$, taking $K=K_1+K_2$ for instance, we have
	\begin{equation}\label{EqTight4}
	\forall N\in\mathbbm{N}^*\quad \forall \epsilon>0\quad \exists K>0\quad \forall n\in\mathbbm{N}:\quad \mathbbm{P}^n\left(\underset{t\leq N}{\text{sup
	}}\|\omega(t)\|>K\right)\leq \epsilon.
	\end{equation}
	Our aim is now to show that
	\begin{equation}\label{EqTight5}
	\forall N\in\mathbbm{N}^*\quad \forall \epsilon>0 \quad \forall \alpha>0\quad \exists \theta\quad \forall n\in\mathbbm{N}:\quad \mathbbm{P}^n(W'_N(\omega,\theta)< \alpha)\geq 1-\epsilon;
	\end{equation}
this combined with \eqref{EqTight4} will imply by Theorem \ref{Tightness} 
	%Theorem 3.21 of \cite{jacod} Chapter VI 
	that $(\mathbbm{P}^n)_{n\geq 0}$ is tight.
	
	In what follows, if $\eta,\omega\in\Omega$ and $s\in\mathbbm{R}_+$, $\eta\otimes_s\omega$ will denote the path $\eta\mathds{1}_{[0,s[}+\omega\mathds{1}_{[s,+\infty[}$, which still belongs to $\Omega$.
	
	By construction of $\mathbbm{Q}^n$, for every $n$,  $\mathbbm{P}^n$ is the law of $\eta_n\otimes_{s_n}\omega$ under $\mathbbm{Q}^n$. Therefore, \eqref{EqTight5} is equivalent to 
	\begin{equation}\label{EqTight6}
	\forall N\in\mathbbm{N}^*\quad \forall \epsilon>0 \quad \forall \alpha>0\quad \exists \theta\quad \forall n\in\mathbbm{N}:\quad \mathbbm{Q}^n(W'_N(\eta_n\otimes_{s_n}\omega,\theta)< \alpha)\geq 1-\epsilon,
	\end{equation}
	and this is what we will now show to conclude the proof of Proposition
\ref{tight}. So we prove \eqref{EqTight6}.
	
	We fix some $N\in\mathbbm{N}^*$, $\alpha>0$ and $\epsilon>0$. Combining the second lines of \eqref{EqTight} and of \eqref{EqTight3}, there exists $\theta>0$ such that for all $n\geq 0$, 
	\begin{equation}\label{EqTight7}
	\left\{
	\begin{array}{l}
	W'_N(\eta_n,\theta)<\frac{\alpha}{4}\\
	\mathbbm{Q}^n(W_N'(\omega,\theta)<\frac{\alpha}{4})\geq 1-\epsilon.
	\end{array}
	\right.
	\end{equation}
	We show below  that, for every  $n$ 
	\begin{equation} \label{EInclusion}
	\{\omega \vert W_N'(\omega,\theta)<\frac{\alpha}{4}\}
	\subset 
	\{\omega \vert W'_N(\eta_n\otimes_{s_n}\omega,\theta)< \alpha\}.
	\end{equation}
This together with \eqref{EqTight7}  will imply that for all $n$,
	$$\mathbbm{Q}^n(W'_N(\eta_n\otimes_{s_n}\omega,\theta)< \alpha)\geq  	\mathbbm{Q}^n(W_N'(\omega,\theta)<\frac{\alpha}{4})\geq 1-\epsilon, $$
  hence that \eqref{EqTight6} is verified.
	
	We fix $n$. To establish \eqref{EInclusion}
let $\omega$ such that
	$W_N'(\omega,\theta)<\frac{\alpha}{4}$; we need to show that
	\begin{equation}\label{EqTight8}
	W'_N(\eta_n\otimes_{s_n}\omega,\theta)< \alpha.
	\end{equation}
	By the first line of \eqref{EqTight7} and the definition of $W'_N$ (see Notation \ref{Modulus}), there exist two subdivisons of $[0,N]$ $0=t^1_0<\cdots < t^1_{r_1}=N$, $0=t^2_0<\cdots < t^2_{r_2}=N$ with increments $t^j_i-t^j_{i-1}\geq \theta$ for all $1\leq i\leq r_j$ and $j=1,2$, such that 
	\begin{equation}\label{EqTight9}
	\left\{\begin{array}{rcl}
	W(\eta_n,[t^1_{i-1},t^1_i[)&\leq& \frac{\alpha}{4}\text{ for all }1\leq i\leq r_1\\
	W(\omega,[t^2_{i-1},t^2_i[)&\leq& \frac{\alpha}{4}\text{ for all }1\leq i\leq r_2.
	\end{array}\right.
	\end{equation}
	We set $i^*_j:=\text{max }\{i:t^j_i\leq s_n\}$ for $j=1,2$ and introduce the third subdivision 
	\begin{equation}
	(t^3_0,\cdots, t^3_{r_3}):=(t^1_0,\cdots,t^1_{i^*_1-1},t^2_{i^*_2+1},\cdots,t^2_{r_2}),
	\end{equation}
	which we represent in the following graphic.
	\\
\includegraphics[width=12cm]{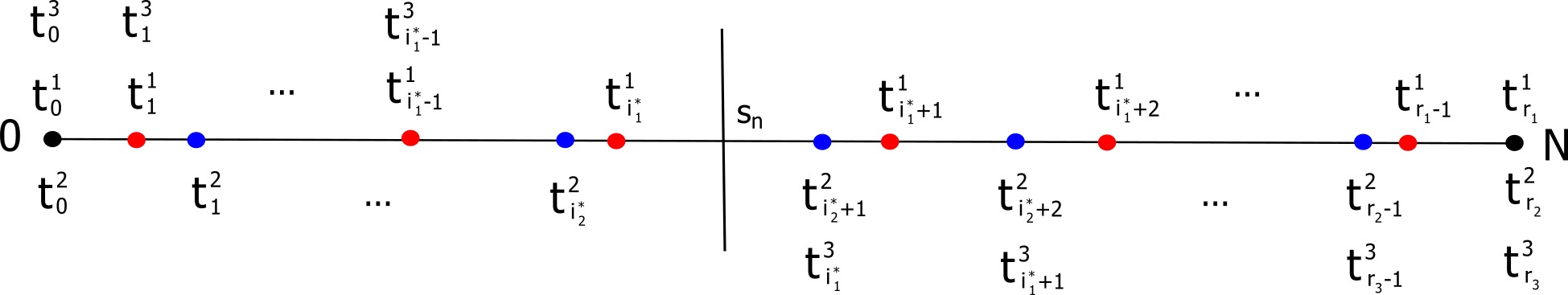}
%	\includegraphics[width=12cm]{SchemaDemoTight.eps}.
% \includepdf[width=0.4\textwidth, angle = 0]{SchemaDemotight.eps}
	\\
	As for the other two, the subdivision of $[0,N]$ above verifies
 $t^3_i-t^3_{i-1}\geq \theta$ for all $i$. Indeed, $t^3_i-t^3_{i-1}$ is either equal to $t^1_i-t^1_{i-1}\geq \theta$, or to $t^2_j-t^2_{j-1}\geq \theta$ for some $j$, or to $t^2_{i^*_2+1}-t^1_{i^*_1-1}\geq t^1_{i^*_1}-t^1_{i^*_1-1} \geq \theta$ where the first inequality follows by the fact that $t^1_{i^*_1-1}\leq t^1_{i^*_1}\leq s_n<t^2_{i^*_2+1}$.
	
	Now by definition of $W'_N(\eta_n\otimes_{s_n}\omega,\theta)$, in order to show \eqref{EqTight8} and conclude this proof, it is enough to show that 
	\begin{equation}\label{EqTight10}
	W(\eta_n\otimes_{s_n}\omega,[t^3_{i-1},t^3_i[)<\alpha,
	\end{equation}
	for all $1\leq i\leq r_3$.
	
	If $i\leq i^*_1-1$, then $[t^3_{i-1},t^3_i[=[t^1_{i-1},t^1_i[\subset[0,s_n[$ where $\eta_n\otimes_{s_n}\omega$ coincides with
	$\eta_n$ so $W(\eta_n\otimes_{s_n}\omega,[t^3_{i-1},t^3_i[)=W(\eta_n,[t^1_{i-1},t^1_i[)\leq \frac{\alpha}{4}<\alpha$ by the first line of \eqref{EqTight9}.
	Similarly, if $i\geq i^*_1+1$, then $[t^3_{i-1},t^3_i[=[t^2_{i-i^*_1+i^*_2},t^2_{i-i^*_1+i^*_2+1}[\subset[s_n,+\infty[$ where $\eta_n\otimes_{s_n}\omega$ coincides with $\omega$ so $W(\eta_n\otimes_{s_n}\omega,[t^3_{i-1},t^3_i[)=
	W(\omega,[t^2_{i-i^*_1+i^*_2},t^2_{i-i^*_1+i^*_2+1}[)\leq \frac{\alpha}{4}<\alpha$ by the second line of \eqref{EqTight9}.
	Finally, we consider the specific case $i=i^*_1$ meaning that $[t^3_{i-1},t^3_i[=[t^1_{i^*_1-1},t^2_{i^*_2+1}[$ contains $s_n$. We have
	\begin{equation}
	\begin{array}{rcl}
	W(\eta_n\otimes_{s_n}\omega,[t^1_{i^*_1-1},t^2_{i^*_2+1}[)&\leq& W(\eta_n\otimes_{s_n}\omega,[t^1_{i^*_1-1},t^1_{i^*_1}[) \\
	&+&W(\eta_n\otimes_{s_n}\omega,[t^1_{i^*_1},s_n[)+W(\eta_n\otimes_{s_n}\omega,[s_n,t^2_{i^*_2+1}[)\\
	&\leq& W(\eta_n,[t^1_{i^*_1-1},t^1_{i^*_1}[)+W(\eta_n,[t^1_{i^*_1},s_n[)+W(\omega,[s_n,t^2_{i^*_2+1}[)\\
	&\leq& W(\eta_n,[t^1_{i^*_1-1},t^1_{i^*_1}[)+W(\eta_n,[t^1_{i^*_1},t^1_{i^*_1+1}[)\\
	&+& W(\omega,[t^2_{i^*_2},t^2_{i^*_2+1}[)\\
	&\leq&\frac{\alpha}{4}+\frac{\alpha}{4}+\frac{\alpha}{4}\\
	&<& \alpha,
	\end{array}
	\end{equation}
	by \eqref{EqTight9}. So \eqref{EqTight10} is verified for all $i$ and the proof is complete.
	
\end{prooff}

%FRANCESCO. FIN VERIFICATIONS FINALES

\begin{proposition}\label{UniqueMPcontImpliesProg}
	Assume that $\beta,\sigma$ (resp. $w$) are bounded and that for Lebesgue almost all $t$ (resp. $dt\otimes dF$ almost all $(t,y)$), $\beta(t,\cdot),\sigma(t,\cdot)$ (resp. $w(t,\cdot,y)$) are continuous. Assume that for any $(s,\eta)\in\mathbbm{R}_+\times\Omega$ there exists a unique weak solution $\mathbbm{P}^{s,\eta}$ of the SDE of coefficients $\beta,\sigma,w$ starting in $(s,\eta)$.
	Then $\begin{array}{ccl}
	(s,\eta)&\longmapsto &\mathbbm{P}^{s,\eta}\\
	\Lambda&\longrightarrow&\mathcal{P}(\Omega)
	\end{array}$ is continuous.
	Moreover the path-dependent canonical class $(\mathbbm{P}^{s,\eta})_{(s,\eta)\in\mathbbm{R}_+\times \Omega}$ is progressive.
\end{proposition}

\begin{remark}
Taking Theorem \ref{SDEcond} into account,
the family of probabilities 
$(\mathbbm{P}^{s,\eta})_{(s,\eta)\in\mathbbm{R}_+\times \Omega}$
 of Proposition \ref{UniqueMPcontImpliesProg}
 constitutes a progressive path-dependent canonical class verifying   Hypothesis \ref{HypClass}. It therefore verifies Hypothesis \ref{HypAF} and all results of Section \ref{SectionMAF} apply. 
\end{remark}
\begin{proof} of Proposition \ref{UniqueMPcontImpliesProg}. \\
	We consider a convergent sequence $(s_n,\eta_n)\underset{n}{\longrightarrow}(s,\eta)$ in $\Lambda$. Since $\beta,\sigma$ are bounded, 
	by Proposition \ref{tight}  $(\mathbbm{P}^{s_n,\eta_n})_{n\in\mathbbm{N}}$ is tight,
 hence relatively compact by Prokhorov's theorem. We consider a subsequence  $\mathbbm{P}^{s_{n_k},\eta_{n_k}}\underset{k}{\longrightarrow}\mathbbm{Q}$ and we show below that $\mathbbm{Q}$ is a weak solution of the SDE 
	with coefficients $\beta, \sigma, w$, starting at $(s,\eta)$.
	Since that problem has a unique solution, we will have $\mathbbm{Q}=\mathbbm{P}^{s,\eta}$. This will imply that $\mathbbm{P}^{s_{n},\eta_{n}}\underset{n}{\longrightarrow}\mathbbm{P}^{s,\eta}$, hence the announced continuity.
	
We  will indeed verify item 3. of Proposition \ref{SDEeq}.
	For the convenience of the reader, we will omit the extraction of the subsequence in the notations.  
	
	We start by showing 
	\begin{equation}\label{EQs}
	\mathbbm{Q}(\omega^s=\eta^s)=1.
	\end{equation} 
	The set
	\begin{equation} \label{E417}
	D:=\left\{t\in\mathbbm{R}_+:\mathbbm{Q}(X_t\neq X_{t^-})>0\right\}\cup \left\{t\in[0,s]:\eta(t)\neq \eta(t^-)\right\},
	\end{equation}
	is countable because first $\eta$ is a  cadlag function
	and second because of Proposition  3.12 in \cite{jacod} Chapter VI
which states that, for every probability $\mathbbm{Q}$
on $(\Omega,\mathcal F)$, the set
$D_0:=\left\{t\in\mathbbm{R}_+:\mathbbm{Q}(X_t\neq X_{t^-})>0\right\}$
is countable. If $t \notin D_0$ then 
\begin{equation} \label{NotD_0}	
\mathbbm{P}^{s_{n},\eta_{n}}\circ X_t^{-1}\underset{n}{\Longrightarrow}\mathbbm{Q}\circ X_t^{-1},
\end{equation}
 by  Proposition 3.14 ibidem.
 Since $\eta_n $ converges to $\eta$
in the Skorohod topology, if $t \notin D$
($t$  is a continuity point of $\eta$), then
  it follows that 
	$\eta_{n}(t)\underset{n}{\longrightarrow}\eta(t)$, see Proposition 2.3 of
 \cite{jacod} Chapter VI.
% 	%and the fact that by it Proposition 2.3 $\eta_n $ converges to $\eta$ for every continuity point of $\eta$ implies the following.
% %	that for any $t\in [0,s]\cap D^c$, 
% 	$\mathbbm{P}^{s_{n},\eta_{n}}\circ X_t^{-1}\underset{n}{\Longrightarrow}\mathbbm{Q}\circ X_t^{-1}$. 
% Moreover by  Proposition 2.3 ibidem 
%   for any $t\in [0,s]\cap D^c$ it follows that 
% 	$\eta_{n}(t)\underset{n}{\longrightarrow}\eta(t)$.
	\\
	
Let $\epsilon>0$, $t\in[0,s-\epsilon]\cap D^c$ be fixed. 
	Since $s_n$ tends to $s$, we can suppose without 
loss of generality that $s_n\geq s-\epsilon$ for all $n$, so that 
$\mathbbm{P}^{s_{n},\eta_{n}}\circ X_t^{-1}=\delta_{\eta_{n}(t)}$.
By \eqref{NotD_0}
 this sequence converges to $\mathbbm{Q}\circ X_t^{-1}$ which is therefore necessarily equal to $\delta_{\eta(t)}$ since $\eta_n(t)$ tends to $\eta(t)$ being $t\notin D$. 
	This means that
	\begin{equation} \label{Qetat}
	\mathbbm{Q}(\omega(t)=\eta(t))=1,
	%  \ \forall
	% t\in[0,s-\epsilon]\backslash D.
	\end{equation}
	for all  $t\in[0,s-\epsilon]\cap D^c$.
	Since $\epsilon>0$ is arbitrary, \eqref{Qetat} holds 
	for all $t\in[0,s[\cap D^c$; and since $\omega$ is right-continuous and 
	$D$ is countable,  \eqref{Qetat} 
        holds for all $t\in[0,s[$. We will now show that 
	\eqref{Qetat} also holds for $t=s$.
	% i.e. \eqref{EQs}.
	We first note that 
	\begin{equation}\label{Elim1}
	\eta_{n}(s_{n})\underset{n}{\longrightarrow}\eta(s).
	\end{equation}
	Indeed,  without restriction of generality we can consider that
	$s_n\leq s+1$, so since $(s_n,\eta_n(s_n)) \in \Lambda$, $\eta_n$
	is constantly
	equal to $\eta_n(s_n)$ on $[s_n,+\infty[$ which contains $[s+1,+\infty[$.
	On the other hand $\eta$ is constantly equal
	to $\eta(s)$ on $[s,+\infty[$
	which also contains $[s+1,+\infty[$, and $\eta_n$ tends to $\eta$ almost everywhere on that interval, 
	because it converges in the
	Skorokhod sense. So necessarily
\eqref{Elim1} holds.
% $\eta_{n}(s_{n})\underset{n}{\longrightarrow}\eta(s)$. 
	\\
	We fix now some $f\in\mathcal{C}_c^{\infty}(\mathbbm{R}^m)$.
	For all $n$, since $\mathbbm{P}^{s_{n},\eta_{n}}$ is a weak solution of the SDE 
	starting at $(s_n,\eta_n)$  and by Proposition \ref{SDEeq}, 
	it follows that
	$f(\omega(\cdot))-f(\eta_n(s_n))-\int_{s_n}^{\cdot}A_rf(\omega)dr$ (see Notation \ref{NotA}) is  a  martingale on $[s_n,+\infty[$
 under $\mathbbm{P}^{s_{n},\eta_{n}}$ vanishing in $s_n$.
	We consider a sequence $(t_p)_{p\in\mathbbm{N}}$ in $D^c$ converging to $t$  
	strictly from the right.
	For all $n,p$ we have
	\begin{equation}\label{Elim0}
	\begin{array}{rcl}
	\mathbbm{E}^{s_{n},\eta_{n}}[f(\omega(t_p))]&=&f(\eta_n(s_n))+\mathbbm{E}^{s_{n},\eta_{n}}\left[\int_{s_n}^{t_p}A_rf(\omega)dr\right]\\
	&=&f(\eta_n(s_n))+\mathbbm{E}^{s_{n},\eta_{n}}\left[\int_{s}^{t_p}A_rf(\omega)dr\right]+\int_{s_n}^{s}\mathbbm{E}^{s_{n},\eta_{n}}[A_rf(\omega)]dr,
	\end{array}
	\end{equation}
	where the second equality holds by Fubini's theorem since $A_rf(\omega)$ is uniformly bounded for $r$ varying
	on bounded intervals. 
	We now pass to the limit in $n$. Since $t_p\notin D$, 
taking into account 
%the considerations after \eqref{E417}
\eqref{NotD_0}, we have	$\mathbbm{P}^{s_{n},\eta_{n}}\circ X_{t_p}^{-1}\underset{n}{\Longrightarrow}\mathbbm{Q}\circ X_{t_p}^{-1}$;
	moreover $f$ is bounded and continuous, so
	\begin{equation}\label{Elim2}
	\mathbbm{E}^{s_{n},\eta_{n}}[f(\omega(t_p))]\underset{n}{\longrightarrow}\mathbbm{E}^{\mathbbm{Q}}[f(\omega(t_p))].
	\end{equation}
	Since $\beta,\sigma,w$ are bounded and $\beta(r,\cdot),\sigma(r,\cdot)$ (resp. $w(r,\cdot,y)$) are continuous for Lebesgue almost all $r$ (resp. $dt\otimes dF$ almost all $(r,y)$) and since $f\in\mathcal{C}_c^{\infty}$, then $\Phi:\omega\longmapsto \int_{s}^{t_p}A_rf(\omega)dr$
 is a bounded continuous functional for the Skorokhod topology, so
	\begin{equation}\label{Elim3}
	\mathbbm{E}^{s_{n},\eta_{n}}\left[\int_{s}^{t_p}A_rf(\omega)dr\right]\underset{n}{\longrightarrow}\mathbbm{E}^{\mathbbm{Q}}\left[\int_{s}^{t_p}A_rf(\omega)dr\right].
	\end{equation}
	Finally since $s_n$ tends to $s$ and $A_rf$ is uniformly bounded for $r$ varying
	on bounded intervals, we have  
	\begin{equation}\label{Elim4}
	\int_{s_n}^{s}\mathbbm{E}^{s_{n},\eta_{n}}[A_rf(\omega)]dr\underset{n}{\longrightarrow}0.
	\end{equation}
	
	Combining relations \eqref{Elim0}, \eqref{Elim1}, \eqref{Elim2}, \eqref{Elim3}, \eqref{Elim4}, for all $p$, we get
	\begin{equation}\label{Elim5}
	\mathbbm{E}^{\mathbbm{Q}}[f(\omega(t_p))]=f(\eta(s))+\mathbbm{E}^{\mathbbm{Q}}\left[\int_{s}^{t_p}A_rf(\omega)dr\right].
	\end{equation}
	We now pass to the limit in $p$. Since $t_p$ tends to $s$ 
	from the right and $\omega$ is right-continuous, the 
	left-hand side of  \eqref{Elim5} tends to $\mathbbm{E}^{\mathbbm{Q}}[f(\omega(s))]$. By dominated convergence, the second 
	term in the right-hand side of \eqref{Elim5} tends to $0$. This yields
	$\mathbbm{E}^{\mathbbm{Q}}[f(\omega(s))]=f(\eta(s))$ and this holds for all $f\in\mathcal{C}_c^{\infty}(\mathbbm{R}^m)$, which implies that $\mathbbm{Q}\circ X_s^{-1}=\delta_{\eta(s)}$. So we have shown that 
	\eqref{Qetat} for $t = s$ and finally \eqref{EQs} since $\omega$ and $\eta$ are cadlag.
	
	We will proceed showing that $\mathbbm{Q}$ solves weakly the SDE
	with respect to $\beta, \sigma, w$ 
	starting in $(s,\eta)$. 
	By Proposition \ref{SDEeq} this holds iff for any  $f\in\mathcal{C}^2_b(\mathbbm{R}^m)$, $f(X_{\cdot})-\int_s^{\cdot}A_rfdr$ is a $(\mathbbm{Q},(\mathcal{F}_t)_{t\in[s,+\infty[})$-martingale.
	We fix such an $f$, some $t\leq u$ in $]s,+\infty[\cap D^c$, $N\in\mathbbm{N}^*$, $t_1\leq\cdots\leq t_N\in[s,t]\cap D^c$ and $\phi_1,\cdots,\phi_N\in\mathcal{C}_b(\mathbbm{R}^m,\mathbbm{R})$.
Taking into account  Proposition \ref{SDEeq},
 since $s<t$, 
for $n$ large enough,
we can suppose that  $f(X_{\cdot})-\int_t^{\cdot}A_rfdr$ is under every  $\mathbbm{P}^{s_{n},\eta_{n}}$ a martingale on the interval $[t,+\infty[$.
	Therefore, for all $n$,
	we have 
	\begin{equation}\label{Elim6}
	\mathbbm{E}^{s_{n},\eta_{n}}\left[\left(f(\omega(u))-f(\omega(t))-\int_t^uA_rf(\omega)dr\right)\underset{1\leq i\leq N}{\Pi}\phi_i(\omega(t_i))\right]=0.
	\end{equation}
We wish to pass to the limit in $n$. By Theorem 12.5 in \cite{billingsley86}, 
for any $r\in\mathbbm{R}_+$, the mapping $X_r$ is continuous on the set 
$C_r:=\{\omega\in \Omega:\omega(r)=\omega(r^-)\}$. By construction of $D$ and since 
$t,u,t_1,\cdots,t_N\notin D$, then $C_t,C_u,C_{t_1},\cdots,C_{t_N}$ are of full $\mathbbm{Q}$-measure hence that 
 $\Phi:=(X,X_u,X_t,X_{t_1},\cdots,X_{t_N})$ is continuous on a set of full $\mathbbm{Q}$-measure. 
By the mapping theorem (see Theorem 2.7 in \cite{billingsley86} for instance), since
 $\mathbbm{P}^{s_{n},\eta_{n}}\underset{n}{\Longrightarrow}\mathbbm{Q}$ and $\Phi$ is continuous on a set of full $\mathbbm{Q}$-measure, then  $\mathbbm{P}^{s_{n},\eta_{n}}\circ\Phi^{-1}\underset{n}{\Longrightarrow}\mathbbm{Q}\circ\Phi^{-1}$,
	meaning  $\mathbbm{P}^{s_{n},\eta_{n}}\circ (X,X_u,X_t,X_{t_1},\cdots,X_{t_N})^{-1}\underset{n}{\Longrightarrow}\mathbbm{Q}\circ (X,X_u,X_t,X_{t_1},\cdots,X_{t_N})^{-1}$. Since 
	$\omega\mapsto \int_t^uA_rf(\omega)dr,f,\phi_1,\cdots,\phi_N$ are bounded continuous functions, the previous convergence in law  allows to
	pass to the limit in $n$ in \eqref{Elim6}   so that
	for any $t\leq u\in]s,+\infty[\cap D^c$ and $t_1,\cdots, t_N\in[s,t]\cap D^c$
	\begin{equation}\label{eqFellerCond}
	\mathbbm{E}^{\mathbbm{Q}}\left[\left(f(\omega(u))-f(\omega(t))-\int_t^uA_rf(\omega)dr\right)\underset{1\leq i\leq N}{\Pi}\phi_i(\omega(t_i))\right]=0.
	\end{equation}
	Equality \eqref{eqFellerCond} still holds if $t=s$ and if some of the values $t,u,t_1,\cdots,t_N$ belong to $D$.
	Indeed to show this statement we approximate from the right 
	such values by  sequences of times  not belonging to $D$ and strictly greater than $s$ and we then use the right-continuity of $\omega$ and 
	the  dominated convergence theorem. 
	
	By use of the functional monotone class theorem (see Theorem 21 in \cite{dellmeyer75} Chapter I), we have
	\begin{equation}\label{eqFellerCond1}
	\mathbbm{E}^{\mathbbm{Q}}\left[\left(f(\omega(u))-f(\omega(t))-\int_t^uA_rf(\omega)dr\right) \mathds{1}_G
	\right]=0,
	\end{equation}
 for any $s\leq t\leq u$ and $G\in\sigma(X_r|r\in[s,t])$. Since $\mathbbm{Q}(\omega^s=\eta^s)=1$ then $\mathcal{F}_s^o$ is $\mathbbm{Q}$-trivial,
	so equality \eqref{eqFellerCond1} holds for all $G=G_s\cap G^s_t$ where $G_s\in\mathcal{F}^o_s$ and $G^s_t\in\sigma(X_r|r\in[s,t])$. Events of such type form a $\pi$-system generating $\mathcal{F}^o_t$ so by Dynkin's Lemma, \eqref{eqFellerCond1} holds for all $G\in\mathcal{F}^o_t$.
	For all $s\leq t\leq u$, then we have
	\begin{equation}
	\mathbbm{E}^{\mathbbm{Q}}\left[\left(f(\omega(u))-f(\omega(t))-\int_t^uA_rf(\omega)_rdr\right)\middle|\mathcal{F}^o_t\right]=0.
	\end{equation}
	So  $f(X)-\int_s^{\cdot}A_rf_rdr$ is a $(\mathbbm{Q},(\mathcal{F}^o_t)_{t\in[s,+\infty[})$-martingale hence a $(\mathbbm{Q},(\mathcal{F}_t)_{t\in[s,+\infty[})$-martingale 
	by Theorem 3 in \cite{dellmeyerB} Chapter VI, 
	that process being right-continuous. 
	\\
	This implies that $\mathbbm{Q}$ is a weak solution of the SDE with coefficients $\beta,\sigma, w$ starting in $(s,\eta)$. As anticipated, since
	the SDE is well-posed  
	for every $(s,\eta)$, we have $\mathbbm{Q}=\mathbbm{P}^{s,\eta}$ and the proof of the first statement is complete. 
	\\
	The second statement follows from the fact that a continuous function is 
	Borel and that $\mathcal{B}(\Lambda)=\Lambda\cap\mathcal{P}ro^o$, see Proposition \ref{BorelLambda}. 
\end{proof}

%FIN VERIFICATIONS FRANCESCO

\begin{appendices}
\section{Proofs of Section \ref{SectionMAF}}
%\ref{SectionMAF}

\begin{prooff}. of Proposition \ref{VarQuadAF}.
\\
In the whole proof $t<u$ will be fixed. 
 We consider a sequence of subdivisions of $[t,u]$: $t=t^k_1<t^k_2<\cdots<t^k_k=u$ such that $\underset{i<k}{\text{min }}(t^k_{i+1}-t^k_i)
\underset{k\rightarrow \infty}{\longrightarrow} 0$.
Let $(s,\eta)\in[0,t]\times \Omega$  with corresponding
probability $\mathbbm{P}^{s,\eta}$.
For any $k$, we have $\underset{i< k}{\sum}\left( M_{t^k_i,t^k_{i+1}}\right)^2=\underset{i< k}{\sum}(M^{s,\eta}_{t^k_{i+1}}-M^{s,\eta}_{t^k_i})^2$ $\mathbbm{P}^{s,\eta}$ a.s.,
  so by definition of  quadratic variation we know that $\underset{i< k}{\sum}\left( M_{t^k_i,t^k_{i+1}}\right)^2 \underset{k\rightarrow \infty}{\overset{\mathbbm{P}^{s,\eta}}{\longrightarrow}} [M^{s,\eta}]_u - [M^{s,\eta}]_t$.
%\begin{equation} \label{B2}
%\underset{i< k}{\sum}\left( M_{t^k_i,t^k_{i+1}}\right)^2 \underset{k\rightarrow \infty}{\overset{\mathbbm{P}^{s,\eta}}{\longrightarrow}} [M^{s,\eta}]_u - [M^{s,\eta}]_t.
%\end{equation}
In the sequel we will construct  an 
%% unique
$\mathcal{F}^o_{u}$-measurable random variable $[M]_{t,u}$ such that for any $(s,\eta)\in[0,t]\times \Omega$, 
$\sum_{i< k}\left( M_{t^k_i,t^k_{i+1}}\right)^2 \underset{k\rightarrow \infty}{\overset{\mathbbm{P}^{s,\eta}}{\longrightarrow}} [M]_{t,u}$. In that case
 $[M]_{t,u}$ will then be $\mathbbm{P}^{s,\eta}$ a.s. equal to $[M^{s,\eta}]_u - [M^{s,\eta}]_t$.

Let $\eta \in \Omega$. $[M^{t,\eta}]$ is $\mathbbm{F}^{t,\eta}$-adapted, so 
%Since $M$ is a path-dependent AF, for any $k$, $\underset{i< k}{\sum}\left( M_{t^k_i,t^k_{i+1}}\right)^2$ is $\mathcal{F}^o_{u}$-measurable and therefore $\mathcal{F}^{t,\eta}_{u}$-measurable. Since $\mathcal{F}^{t,\eta}_{u}$ is complete, the limit in probability of this sequence,
 $[M^{t,\eta}]_u-[M^{t,\eta}]_t$ is  $\mathcal{F}^{t,\eta}_{u}$-measurable and by Corollary \ref{F0version}, 
there is  an $\mathcal{F}^o_{u}$-measurable variable which depends on $(t,u,\eta)$, 
that we denote $\omega\mapsto a_{t,u}(\eta,\omega)$ such that $a_{t,u}(\eta,\omega) =  [M^{t,\eta}]_u-[M^{t,\eta}]_t, {\mathbb P}^{t,\eta} \ {\rm a.s.}$
%\begin{equation} \label{B2bis}
% a_{t,u}(\eta,\omega) =  [M^{t,\eta}]_u-[M^{t,\eta}]_t, {\mathbb P}^{t,\eta} \ {\rm a.s.}
%\end{equation}
We will show below that there is a jointly $\mathcal{F}^o_t\otimes\mathcal{F}^o_u$ -measurable version 
of $(\eta,\omega) \mapsto a_{t,u}(\eta,\omega)$.

For every integer $n\geq 0$, we set $a^n_{t,u}(\eta,\omega):=n\wedge a_{t,u}(\eta,\omega)$ 
which is in particular limit in probability of $n\wedge\underset{i\leq k}{\sum}\left( M_{t^k_i,t^k_{i+1}}\right)^2$ under $\mathbbm{P}^{t,\eta}$.
For any integers $k,n$ and any $\eta\in \Omega$, we define the finite positive measures $\mathbbm{Q}^{k,n,\eta}$, $\mathbbm{Q}^{n,\eta}$ and $\mathbbm{Q}^{\eta}$ on $(\Omega,\mathcal{F}^o_{u})$ by
\begin{enumerate}
\item $\mathbbm{Q}^{k,n,\eta}(F) := \mathbbm{E}^{t,\eta}\left[\mathds{1}_F\left(n\wedge\underset{i< k}{\sum}\left( M_{t^k_i,t^k_{i+1}}\right)^2\right)\right]$; 
\item $\mathbbm{Q}^{n,\eta}(F): = \mathbbm{E}^{t,\eta}[\mathds{1}_F\left(a^n_{t,u}(\eta,\omega)\right)]$;
\item $\mathbbm{Q}^{\eta}(F) := \mathbbm{E}^{t,\eta}[\mathds{1}_F\left(a_{t,u}(\eta,\omega)\right)]$.
\end{enumerate}
When $k$ and $n$ are fixed integers and  $F$ is a fixed event, by Remark \ref{Borel}, 
\\
$\eta\longmapsto\mathbbm{E}^{t,\eta}\left[F\left(n\wedge\underset{i< k}{\sum}\left( M_{t^k_i,t^k_{i+1}}\right)^2\right)\right],$ is $\mathcal{F}^o_t$ -measurable. 

Then $n\wedge\underset{i< k}{\sum}\left( M_{t^k_i,t^k_{i+1}}\right)^2\overset{\mathbbm{P}^{t,\eta}}{\underset{k\rightarrow \infty}{\longrightarrow}} a^n_{t,u}(\eta,\omega)$, and this sequence is uniformly bounded by the constant $n$, so the convergence takes place in $L^1$, therefore 
$\eta\longmapsto\mathbbm{Q}^{n,\eta}(F)$ is also $\mathcal{F}^o_t$-measurable as the pointwise limit in $k$ of the functions 
$\eta\longmapsto\mathbbm{Q}^{k,n,\eta}(F)$.
 Similarly, $a^n_{t,u}(\eta,\omega)\underset{n\rightarrow \infty}{\overset{\mathbbm{P}^{t,\eta}-a.s.}{\longrightarrow}}a_t(\eta,\omega)$ and is non-decreasing, so by 
 monotone convergence theorem,  the function
$\eta\longmapsto\mathbbm{Q}^{\eta}(F)$ is $\mathcal{F}^o_t$-measurable
being  a pointwise limit in $n$ of the functions
 $\eta\longmapsto\mathbbm{Q}^{n,\eta}(F)$.

%We recall that $\mathcal{F}$ is separable.
We make then use of Theorem 58 Chapter V in \cite{dellmeyerB}:
 the  property above, the separability of $\mathcal{F}$  and
the fact that for any $\eta$,  
 $\mathbbm{Q}^{\eta}\ll \mathbbm{P}^{t,\eta}$ by item 3. above, imply 
the existence of a jointly measurable
 (for $\mathcal{F}^o_t\otimes\mathcal{F}^o_u$) version
  of   $(\eta,\omega)\mapsto a_{t,u}(\eta,\omega)$.
 That version will still be denoted
by the same symbol.
 We recall that
for any $\eta$,  $a_{t,u}(\eta,\cdot)$ is the Radon-Nykodim density
of  $\mathbbm{Q}^{\eta}$ 
with respect to $\mathbbm{P}^{t,\eta}$.

We can now set $[M]_{t,u}(\omega):=a_{t,u}(\omega,\omega)$,
 which is a well-defined $\mathcal{F}^o_{u}$-measurable random variable.
Since $a_{t,u}$ is $\mathcal{F}^o_t$-measurable in the first variable and 
for any $\eta$
$\mathbbm{P}^{t,\eta}(\omega^t=\eta^t)=1$  we have
 the equalities 
\begin{equation} \label{B3bis}
[M]_{t,u}(\omega) =a_{t,u}(\omega,\omega)=a_{t,u}(\eta,\omega)=[M^{t,\eta}]_u(\omega) - [M^{t,\eta}]_t(\omega)  \ \mathbbm{P}^{t,\eta} {\rm a.s.}
\end{equation}

We can then show that
\begin{equation} \label{Etsx}
[M]_{t,u} = [M^{s,\eta}]_u - [M^{s,\eta}]_t\,\text{ }\,\mathbbm{P}^{s,\eta} 
\text{ a.s.},
\end{equation}
holds   for every $(s,\eta)\in[0,t]\times \Omega$, and not just in the
 case $s=t$ 
that  we have just established in \eqref{B3bis}. 
This can be done  reasoning as in the proof of Proposition 4.4 in 
\cite{paperAF}, replacing the use of the Markov property with item 3. of Definition \ref{DefCondSyst}.
%\\
%We proceed proving the validity of  \eqref{Etsx} also 
%for a fixed  $s<t$ and $\eta \in \Omega$.
%%(s,x)\in[0,T]\times E$ such that $s<t$. 
%It will be enough to show 
% that, under any $\mathbbm{P}^{s,\eta}$, $[M]_{t,u}$ is the limit in probability of $\underset{i< k}{\sum}\left( M_{t^k_i,t^k_{i+1}}\right)^2$. Indeed, let $\epsilon>0$. By conditioning and using Definition \ref{DefCondSyst} we have
%\begin{equation*}
%\begin{array}{rcl}
%   && \mathbbm{P}^{s,\eta}\left(\left|\underset{i< k}{\sum}\left( M_{t^k_i,t^k_{i+1}}\right)^2-[M]_{t,u}\right|>\epsilon\right)\\ &=& \mathbbm{E}^{s,\eta}\left[\mathbbm{P}^{s,\eta}\left(\left|\underset{i< k}{\sum}\left( M_{t^k_i,t^k_{i+1}}\right)^2-[M]_{t,u}\right|>\epsilon\middle|\mathcal{F}_t\right)\right]\\
%    &=& \mathbbm{E}^{s,\eta}\left[\mathbbm{P}^{t,\omega}\left(\left|\underset{i< k}{\sum}\left( M_{t^k_i,t^k_{i+1}}\right)^2-[M]_{t,u}\right|>\epsilon\right)\right].
%\end{array}
%\end{equation*}
%For any fixed $\omega$, 
%by \eqref{B2} and \eqref{B3bis},
%$\mathbbm{P}^{t,\omega}\left(\left|\underset{i< k}{\sum}\left( M_{t^k_i,t^k_{i+1}}\right)^2-[M]_{t,u}\right|>\epsilon\right)$ tends to zero when $k$ goes to infinity, 
%so a.s. under the  probability $\mathbbm{P}^{s,\eta}$, it yields
% $\mathbbm{P}^{t,\omega}\left(\left|\underset{i< k}{\sum}\left( M_{t^k_i,t^k_{i+1}}\right)^2-[M]_{t,u}\right|>\epsilon\right)$ tends  to zero when 
%$k$ goes to infinity. Since this sequence is dominated by the constant $1$,
% that convergence still holds under the expectation thanks to the dominated convergence theorem.
%\\

So we have built an $\mathcal{F}^o_{u}$-measurable variable $[M]_{t,u}$ such that under any $\mathbbm{P}^{s,\eta}$ with $s\leq t$, 
$[M^{s,\eta}]_u - [M^{s,\eta}]_t=[M]_{t,u}$ a.s. and this concludes the proof.
\end{prooff}

\begin{prooff}. of Proposition \ref{AngleBracketAF}.
\\
We start defining $A_{t,t} = 0$ for every $t \ge 0$.
We then recall a property of the predictable dual projection
 which we will have to extend slightly. 
Let us fix $(s,\eta)$ and the corresponding stochastic basis 
 $(\Omega,\mathcal{F}^{s,\eta},\mathbbm{F}^{s,\eta},\mathbbm{P}^{s,\eta})$. 
For any $F\in \mathcal{F}^{s,\eta}$, let $N^{s,\eta,F}$ be the cadlag version of the martingale 
$r\longmapsto\mathbbm{E}^{s,\eta}[\mathds{1}_F|\mathcal{F}_r]$. Then for any $0\leq t\leq u$, the predictable projection of the process $r\mapsto \mathds{1}_F\mathds{1}_{[t,u[}(r)$ is $r\mapsto N^{s,\eta,F}_{r^-}\mathds{1}_{[t,u[}(r)$, see the proof of Theorem 43 Chapter VI in \cite{dellmeyerB}. Therefore by definition 
of the dual predictable projection (see Definition 73 Chapter VI in \cite{dellmeyerB}), for any $0\leq t\leq u$ and $F\in\mathcal{F}^{s,\eta}$ we have
$\mathbbm{E}^{s,\eta}\left[\mathds{1}_F(A^{s,\eta}_{u^-} - A^{s,\eta}_t)\right]=\mathbbm{E}^{s,\eta}\left[\int_t^{u^-} N^{s,\eta,F}_{r^-}dB^{s,\eta}_r\right]$.
Then, at fixed $t,u,F$, since for every $\epsilon>0$ we have 
$\mathbbm{E}^{s,\eta}\left[\mathds{1}_F(A^{s,\eta}_{(u+\epsilon)^-} - A^{s,\eta}_t)\right]=\mathbbm{E}^{s,\eta}\left[\int_t^{(u+\epsilon)^-} N^{s,\eta,F}_{r^-}dB^{s,\eta}_r\right]$, 
letting $\epsilon$ tend to zero we obtain by dominated convergence theorem that
\begin{equation}\label{dualprojection}
\mathbbm{E}^{s,\eta}\left[\mathds{1}_F(A^{s,\eta}_{u} - A^{s,\eta}_t)\right]=\mathbbm{E}^{s,\eta}\left[\int_t^{u} N^{s,\eta,F}_{r^-}dB^{s,\eta}_r\right],
\end{equation}  
taking into account the right-continuity of $A^{s,\eta},B^{s,\eta}$ and the fact that they are both non-decreasing processes with ${\mathcal L}^1$ -terminal value.

For any $F\in\mathcal{F}$, we introduce the process $N^F:(t,\omega)\longmapsto \mathbbm{P}^{t,\omega}(F)$. $N^F$ takes values in $[0,1]$ for every $(t,\omega)$,  and by Definition \ref{DefCondSyst},
 it is an $\mathbbm{F}^o$-progressively measurable process such that for any $(s,\eta)\in\mathbbm{R}_+\times \Omega$, $N^{s,\eta,F}$ is 
 a $\mathbbm{P}^{s,\eta}$ cadlag version of $N^F$ on $[s,+\infty[$.

For the rest of the proof, $0\leq t<u $ are fixed.
Following the same proof than that of Lemma 4.9 in \cite{paperAF} but with parameter $(s,x)$ replaced with $(s,\eta)$, we obtain the following.
\begin{lemma}\label{commonint}
Let  $F\in\mathcal{F}$.  There exists an $\mathcal{F}_u$-measurable random variable
% independent of $(s,x)$ 
which we will call $\int_t^u N^F_{r^-}dB_r$ such that for any 
$(s,\eta)\in[0,t]\times \Omega$, \\ $\int_t^u N^F_{r^-}dB_r=\int_t^u N^{s,\eta,F}_{r^-}dB^{s,\eta}_r$ $\mathbbm{P}^{s,\eta}$ a.s. 
%If moreover $F\in\mathcal{F}_{t,u}$ then $\int_t^u N^F_{r^-}dB_r$ is $\mathcal{F}_{t,u}$-measurable. 
\end{lemma}
\begin{remark} \label{Rsx}
By definition,  the r.v.
 $\int_t^u N^F_{r^-}dB_r$ will not depend on  $(s,\eta)$.
\end{remark}

We continue now the proof of
Proposition \ref{AngleBracketAF} by showing  that for given $ 0 \leq t < u$ there is  
an $\mathcal{F}^o_u$-measurable r.v. $A_{t,u}$ 
%the existence of an $\mathcal{F}_{t,u}$-measurable r.v. $A_{t,u}$ 
such that for every $(s,\eta)\in[0,t]\times \Omega$,  $(A^{s,\eta}_u-A^{s,\eta}_t)=A_{t,u}$ $\quad\mathbbm{P}^{s,\eta}$ a.s.

Similarly to what we did with the quadratic variation in 
Proposition  \ref{VarQuadAF}, we start  noticing that for any $\eta\in \Omega$, being
$(A^{t,\eta}_u-A^{t,\eta}_t)$  $\mathcal{F}^{t,\eta}_{u}$-measurable, there exists by Corollary \ref{F0version} an $\mathcal{F}^o_{u}$-measurable r.v. $\omega\mapsto a_{t,u}(\eta,\omega)$ such that 
\begin{equation}\label{E425}
	a_{t,u}(\eta,\omega)=A^{t,\eta}_u(\omega)-A^{t,\eta}_t(\omega)\quad \mathbbm{P}^{t,\eta}\text{ a.s.}
\end{equation}
As in the proof of Proposition  \ref{VarQuadAF}, we show below the existence of a jointly-measurable version of  $(\eta,\omega)\mapsto a_{t,u}(\eta,\omega)$.

For  every $\eta\in \Omega$  we define 
on $\mathcal{F}^o_{u}$ the positive measure
\begin{equation} \label{E424}
\mathbbm{Q}^{\eta}:F\longmapsto
\mathbbm{E}^{t,\eta}\left[\mathds{1}_Fa_{t,u}(\eta,\omega)\right] =
\mathbbm{E}^{t,\eta}\left[\mathds{1}_F(A^{t,\eta}_u - A^{t,\eta}_t)\right].
\end{equation}
By Lemma \ref{commonint} and \eqref{dualprojection}, for every
 $F\in \mathcal{F}^o_{u}$ we have
\begin{equation}\label{E426}
	\mathbbm{Q}^{\eta}(F)=\mathbbm{E}^{t,\eta}\left[\int_t^u N^F_{r^-}dB_r\right],
\end{equation}
 where we recall that $\int_t^u N^F_{r^-}dB_r$  does not depend on $\eta$. 
 So by Remark \ref{Borel},  $\eta\longmapsto\mathbbm{Q}^{\eta}(F)$ is $\mathcal{F}^o_t$-measurable for any $F$. Moreover, by \eqref{E424}
 for any $\eta$,
 $\mathbbm{Q}^{\eta}\ll \mathbbm{P}^{t,\eta}$. Again by Theorem 58 Chapter V in \cite{dellmeyerB}, there exists a  version $(\eta,\omega)\mapsto a_{t,u}(\eta,\omega)$ 
-measurable for $\mathcal{F}^o_t\otimes \mathcal{F}^o_u$ of the 
related Radon-Nikodym densities.

 We can now set $A_{t,u}(\omega) := a_{t,u}(\omega,\omega)$ 
which is then an  $\mathcal{F}^o_{u}$-measurable r.v.  
\\
It yields for any $\eta\in\Omega$
\begin{equation}\label{E427}
	A_{t,u}(\omega) = a_{t,u}(\omega,\omega) = a_{t,u}(\eta,\omega)= A^{t,\eta}_u(\omega) - A^{t,\eta}_t(\omega)\quad \mathbbm{P}^{t,\eta}\text{ a.s.}
\end{equation}
Indeed the second equality holds
given that $a_{t,u}$ is $\mathcal{F}^o_t$-measurable with respect to the first variable, taking into account that $\mathbbm{P}^{t,\eta}(\omega^t=\eta^t) = 1$;
 the third equality follows by \eqref{E425}.

We now set $s<t$ and $\eta\in \Omega$. We want to show that  we still have 
\\
$A_{t,u} =  A^{s,\eta}_u - A^{s,\eta}_t$ $\mathbbm{P}^{s,\eta}$ a.s. So we
 consider  $F\in\mathcal{F}^o_{u}$; 
%thanks to \eqref{dualprojection} and Lemma \ref{commonint},
 we compute
\begin{equation}\label{E428}
 \begin{array}{rclcl}
    &&\mathbbm{E}^{s,\eta}\left[\mathds{1}_F(A^{s,\eta}_u - A^{s,\eta}_t)\right] &=&\mathbbm{E}^{s,\eta}\left[\int_t^u N^F_{r^-}dB_r\right]\\
    &=&\mathbbm{E}^{s,\eta}\left[\mathbbm{E}^{s,\eta}\left[\int_t^u N^{F}_{r^-}dB_r|\mathcal{F}_t\right]\right]
    &=&\mathbbm{E}^{s,\eta}\left[\mathbbm{E}^{t,\omega}\left[\int_t^u N^F_{r^-}dB_r\right]\right]\\
    &=&\mathbbm{E}^{s,\eta}\left[\mathbbm{E}^{t,\omega}\left[\mathds{1}_FA_{t,u}\right]\right]
    &=&\mathbbm{E}^{s,\eta}\left[\mathbbm{E}^{s,\eta}\left[\mathds{1}_FA_{t,u}|\mathcal{F}_t\right]\right]\\
    &=&\mathbbm{E}^{s,\eta}\left[\mathds{1}_FA_{t,u}\right].&&
\end{array}
\end{equation}
Indeed, the first equality comes from \eqref{dualprojection} and Lemma \ref{commonint}; concerning the fourth equality we recall that, by \eqref{E424}, \eqref{E426} and 
\eqref{E427}, we have
$\mathbbm{E}^{t,\omega}\left[\int_t^u N^F_{r^-}dB_r\right]=\mathbbm{E}^{t,\omega}\left[\mathds{1}_FA_{t,u}\right]$ for all $\omega$. The third and fifth equalities come from Remark \ref{Borel}. 

Since adding $\mathbbm{P}^{s,\eta}$-null sets does not change the validity of \eqref{E428}, by Proposition \ref{CompFiltIni}  for
 any $F\in\mathcal{F}^{s,\eta}_{u}$ we have $\mathbbm{E}^{s,\eta}\left[\mathds{1}_F(A^{s,\eta}_u - A^{s,\eta}_t)\right]=\mathbbm{E}^{s,\eta}\left[\mathds{1}_FA_{t,u}\right]$.

Finally, since both $A^{s,\eta}_u - A^{s,\eta}_t$ and $A_{t,u}$ are $\mathcal{F}^{s,\eta}_{u}$-measurable, 
we can conclude that $A^{s,\eta}_u - A^{s,\eta}_t=A_{t,u}$  $\mathbbm{P}^{s,\eta}$ a.s.
\\
We emphasize that this holds for any $t\leq u$
 and $(s,\eta)\in[0,t]\times \Omega$, $(A_{t,u})_{(t,u)\in\Delta}$ is the desired path-dependent AF, which ends the proof of 
Proposition \ref{AngleBracketAF}.
\end{prooff}

%We recall .
%\begin{lemma} \label{JointMeasurability}
%Let $(E,\mathcal{E})$ be a measurable space and let $f:E\times\mathbbm{R}_+\longrightarrow\mathbbm{R}$ be a mapping such that for all $t\in\mathbbm{R}_+$, $x\mapsto f(x,t)$ is measurable with respect to $\mathcal{E}$ and for all $x\in E$, $t\mapsto f(x,t)$ is right-continuous, then $f$ is measurable with respect to $\mathcal{E}\otimes \mathcal{B}(\mathbbm{R}_+)$.
%\end{lemma}

\begin{prooff}. of Proposition \ref{RadonDerivAF}.
\\
We set 
\begin{equation} \label{EDC} 
C_{t,u}=A_{t,u} +(V_u-V_t)+(u-t),
\end{equation}
 which is a path-dependent AF with cadlag versions $C^{s,\eta}_t=A^{s,\eta}_t+V_t+t$
%\begin{equation} \label{EDCsx} 
%C^{s,\eta}_t=A^{s,\eta}_t+V_t+t,
%\end{equation}
 and we start by showing the statement
for $A$ and $C$ instead of $A$ and $V$.
\\
The reason of the  introduction of the
 intermediary function $C$ is  that for any $u>t$ we have
 $\frac{A^{s,\eta}_u-A^{s,\eta}_t}{C^{s,\eta}_u-C^{s,\eta}_t}\in[0,1]$; that property  will 
be used extensively in connections with the application
of  dominated convergence theorem. 
%FIN VERIFICATIONS FRANCESCO
\\
Since $A^{s,\eta}$ is non-decreasing for any $(s,\eta)\in\mathbbm{R}_+\times \Omega$, $A$ can be taken positive (in the sense that $A_{t,u}(\omega)\geq 0$ for any $(t,u)\in\Delta$ and $\omega\in\Omega$) by considering $A^+$ (defined by $(A^+)_{t,u}(\omega):=A_{t,u}(\omega)^+$) instead of $A$.
On $\mathbbm{R}_+$ we set   
\begin{eqnarray} \label{EB29}
K_t &=& \underset{n\rightarrow \infty}{\text{liminf }}\frac{A_{t,t+\frac{1}{n}}}{A_{t,t+\frac{1}{n}}+\frac{1}{n}+(V_{t+\frac{1}{n}}-V_t)} \nonumber \\
%&=\underset{n\rightarrow \infty}{\text{lim }}\underset{p\geq n}{\text{inf }}\frac{A_{t,t+\frac{1}{p}}}{A_{t,t+\frac{1}{p}}+\frac{1}{p}+(V_{t+\frac{1}{p}}-V_t)}\\
&& \\
&=&\underset{n\rightarrow \infty}{\text{lim }}\underset{m\rightarrow \infty}{\text{lim  }}\underset{n\leq p \leq m}{\text{min }}\frac{A_{t,t+\frac{1}{p}}}{A_{t,t+\frac{1}{p}}+\frac{1}{p}+(V_{t+\frac{1}{p}}-V_t)}. \nonumber
\end{eqnarray}
This liminf always exists and belongs to $[0,1]$ since the sequence belongs to $[0,1]$. For any $(s,\eta)\in\mathbbm{R}_+\times \Omega$, since for all $t\geq s$ and $n\geq 0$, 
\\
$A_{t,t+\frac{1}{n}} = A^{s,\eta}_{t+\frac{1}{n}}-A^{s,\eta}_t$ $\mathbbm{P}^{s,\eta}$ a.s., then $K^{s,\eta}$ defined by
$K^{s,\eta}_t:=\underset{n\rightarrow \infty}{\text{liminf }}\frac{A^{s,\eta}_{t+\frac{1}{n}}-A^{s,\eta}_t}{C^{s,\eta}_{t+\frac{1}{n}}-C^{s,\eta}_t}$ is a $\mathbbm{P}^{s,\eta}$-version of $K$, for  $ t \in  [s,+\infty[$.
\\
By  Lebesgue Differentiation theorem (see Theorem 12 Chapter XV in \cite{dellmeyerD} for a version of the theorem with a general atomless measure),
 for any $(s,\eta)$, for  $\mathbbm{P}^{s,\eta}$-almost all  $\omega$,
since
% $t \mapsto
$ dC^{s,\eta}(\omega)$ is absolutely continuous with respect
to $dA^{s,\eta}(\omega)$, 
 $K^{s,\eta}(\omega)$ is a density of $dA^{s,\eta}(\omega)$ with respect to $dC^{s,\eta}(\omega)$. 
\\
For any $t\geq 0$, $K_t$ is measurable with respect to $\underset{n\geq 0}{\bigcap} \mathcal{F}^o_{t+\frac{1}{n}}=\mathcal{F}_{t}$, 
by definition of the canonical filtration. 
For any $(t,\omega)\in\mathbbm{R}_+\times \Omega$,
we now set
\begin{equation} \label{EL320} 
k_t(\omega):=\mathbbm{E}^{t,\omega}[K_t].
\end{equation} 
Remark \ref{Borel} implies that $k$ is an $\mathbbm{F}^o$-adapted process. 
The path-dependent canonical class  verifies Hypothesis \ref{HypClass}, and 
%since
 $K_t$ is $\mathcal{F}_{t}$-measurable then for any $(s,\eta)\in[t,+\infty[\times \Omega$, $K_t(\omega)=\mathbbm{E}^{s,\eta}[K_t|\mathcal{F}_t](\omega)=\mathbbm{E}^{t,\omega}[K_t]=k_t(\omega)$ $\mathbbm{P}^{s,\eta}$-a.s.:
 hence 
 $k$ is on $[s,+\infty[$ a $\mathbbm{P}^{s,\eta}$-version of $K$, and therefore of $K^{s,\eta}$. 

The next main object of this proof is to show that $k$ is an $\mathbbm{F}^o$-progressively measurable 
 process. 
For any integers $(n,m)$, we define 
\begin{equation*}
k^{n,m}:(t,\eta)\mapsto\mathbbm{E}^{t,\eta}\left[\underset{n\leq p \leq m}{\text{min }}\frac{A_{t,t+\frac{1}{p}}}{A_{t,t+\frac{1}{p}}+\frac{1}{p}+(V_{t+\frac{1}{p}}-V_t)}\right],
\end{equation*}
and for all $n$, 
\begin{equation} \label{EB32Pre}
k^n:(t,\eta)\mapsto\mathbbm{E}^{t,\eta}\left[\underset{p\geq n}{\text{inf }}\frac{A_{t,t+\frac{1}{p}}}{A_{t,t+\frac{1}{p}}+\frac{1}{p}+(V_{t+\frac{1}{p}}-V_t)}\right].
\end{equation}
We start  showing that 
\begin{equation} \label{EB32}
\tilde{k}^{n,m}:\begin{array}{rcl}
    ((s,\eta),t)&\longmapsto& \mathbbm{E}^{s,\eta}\left[\underset{n\leq p\leq m}{\text{min }}\frac{A_{t,t+\frac{1}{p}}}{A_{t,t+\frac{1}{p}}+\frac{1}{p}+(V_{t+\frac{1}{p}}-V_t)}\right]\mathds{1}_{s\leq t},\\ \relax 
    (\mathbbm{R}_+\times \Omega)\times\mathbbm{R}_+  &\longrightarrow& [0,1],
\end{array}
\end{equation}
is measurable with respect to $\mathcal{P}ro^o\otimes\mathcal{B}(\mathbbm{R}_+)$. In order to do so, we will show that it is measurable in the first variable $(s,\eta)$, and right-continuous in the second variable $t$, and conclude with Lemma 4.12 in \cite{paperAF}.
\\ 
We fix $t\in\mathbbm{R}_+$. Since the path-dependent canonical class is progressive,  
by Remark \ref{Borel},  the map
\begin{eqnarray} \label{EPM}
 %\begin{array}{rcl}
(s,\eta)&\longmapsto&\mathbbm{E}^{s,\eta}\left[\underset{n\leq p\leq m }{\text{min }}\frac{A_{t,t+\frac{1}{p}}}{A_{t,t+\frac{1}{p}}+\frac{1}{p}+(V_{t+\frac{1}{p}}-V_t)}\right]\\ \mathbbm{R}_+\times \Omega&\longrightarrow& [0,1], \nonumber
\end{eqnarray}
 is measurable with respect to $\mathcal{P}ro^o$. The map $(s,\eta)\longmapsto \mathds{1}_{[t,+\infty[}(s)$ is also trivially  measurable with respect to $\mathcal{P}ro^o$; therefore the product of the latter map and \eqref{EPM}, 
that we denote by $\tilde{k}(\cdot,\cdot,t)$ is also measurable 
with respect to $\mathcal{P}ro^o$.
%so  composing with $(s,\eta)\mapsto (s\wedge t,\eta)$, 
%\\
%$\begin{array}{rcl}(s,\eta)&\longmapsto&\tilde{k}^{n,m}(s,\eta,t)\\ \,[0,T_0]\times \Omega&\longrightarrow& [0,1],
%\end{array}$
%is again $\mathcal{B}([0,T_0])\otimes \mathcal{F}^o_{T_0}$-measurable. 
%This holds for any $T_0 > 0$ so $(s,\eta)\longmapsto\tilde{k}^{n,m}(s,\eta,t)$ is $\mathcal{P}ro^o$-measurable.
 Moreover, if we fix $(s,\eta)\in \mathbbm{R}_+\times \Omega$, reasoning exactly as in the proof of Proposition 4.13 in \cite{paperAF}
  we  see that $t\mapsto \tilde{k}^{n,m}(s,\eta,t)$ is right-continuous, which by Lemma 
 4.12 in \cite{paperAF} implies the joint measurability of $\tilde{k}^{n,m}$.

Since $k^{n,m}(t,\eta)=\tilde{k}^{n,m}(t,t,\eta)$, and since $(t,\eta)\mapsto(t,\eta,t)$ is obviously\\ $(\mathcal{P}ro^o,\mathcal{P}ro^o\otimes\mathcal{B}(\mathbbm{R}_+))$-measurable, then by composition we can deduce that for any $n,m$, $k^{n,m}$ is an $\mathbbm{F}^o$-progressively measurable process. By the dominated convergence theorem, $k^{n,m}$ tends pointwise to $k^n$ when $m$ goes to infinity, so  $k^n$ also is an $\mathbbm{F}^o$-progressively measurable process 
for every $n$. 
Finally, since $K_t = \underset{n\rightarrow \infty}{\text{lim }}\underset{p\geq n}{\text{inf }}\frac{A_{t,t+\frac{1}{p}}}{A_{t,t+\frac{1}{p}}+\frac{1}{p}+(V_{t+\frac{1}{p}}-V_t)}$,
taking the expectation and again by the dominated 
convergence theorem, $k^n$
(defined in \eqref{EB32Pre}) 
 tends pointwise to $k$ (defined in \eqref{EL320}),
 when $n$ goes to infinity, so $k$ is an $\mathbbm{F}^o$-progressively measurable process.
%We now show that, for any $(s,\eta)\in \mathbbm{R}_+\times \Omega$,  $k$ is a $\mathbbm{P}^{s,x}$-version of $K$ on $[s,+\infty[$. 
%\\
%By Definition \ref{DefCondSyst}, we know that for any $(t,\eta)\in \mathbbm{R}_+\times \Omega$, we have 
%$K_t(\omega)=k_t(\eta)=k_t(\omega)$ $\mathbbm{P}^{t,\eta}$-a.s., and we prove below
% that for any $t \in \mathbbm{R}_+$, $(s,\eta)\in [0,t]\times \Omega$, we have 
%$K_t(\omega)=k_t(X_t)$ $\mathbbm{P}^{s,x}$-a.s.
%\\
%Let $t \in [0,T]$ be fixed.
%Since $A$ is an AF, for any $n$, $\frac{A_{t,t+\frac{1}{p}}}{A_{t,t+\frac{1}{p}}+\frac{1}{n}+(V_{t+\frac{1}{n}}-V_t)}$ is $\mathcal{F}_{t,t+\frac{1}{n}}$-measurable.
%\\
%So the event $\left\{\underset{n\rightarrow \infty}{\text{liminf }}\frac{A_{t,t+\frac{1}{p}}}{A_{t,t+\frac{1}{p}}+\frac{1}{n}+(V_{t+\frac{1}{n}}-V_t)}=k(t,X_t)\right\}$ belongs to $\mathcal{F}_{t,T}$ and
% by Markov property \eqref{Markov3}, for any $(s,x)\in[0,t]\times E$, 
%we get 
%\begin{eqnarray*}
%\mathbbm{P}^{s,x}(K_t= k(t,X_t))&=&\mathbbm{E}^{s,x}[\mathbbm{P}^{s,x}\left(K_t= k(t,X_t)\middle|\mathcal{F}_t\right)]  \nonumber \\
% &=&\mathbbm{E}^{s,x}[\mathbbm{P}^{t,X_t}\left(K_t= k(t,X_t)\right)]\\
% &=& 1. \nonumber
%\end{eqnarray*}
%For any $(s,\eta)$, the process $k$ is therefore on $[s,+\infty[$ a 
%$\mathbbm{P}^{s,\eta}$-modification of $K^{s,\eta}$.
%\\
%\\
% But it is not yet clear that it provides another density of $dA^{s,x}$
% with respect to $dC^{s,x}$, which was defined at \eqref{EDCsx}.
%\\
Considering that $(t,u,\omega)\mapsto V_u-V_t$ also trivially 
defines a non-negative 
non-decreasing path-dependent AF absolutely continuous with respect to $C$,
defined in \eqref{EDC},
 we proceed similarly as at the beginning of the proof,
 replacing the path-dependent AF $A$ with $V$. 

Let the process $K'$ be defined by
$K'_t = \underset{n\rightarrow \infty}{\text{liminf }}\frac{V_{t+\frac{1}{n}}-V_t}{A_{t,t+\frac{1}{n}}+\frac{1}{n}+(V_{t+\frac{1}{n}}-V_t)}$,
and for any $(s,\eta)$, let $K'^{s,\eta}$ be defined on $[s,+\infty[$ by
$K'^{s,\eta}_t = \underset{n\rightarrow \infty}{\text{liminf }}\frac{V_{t+\frac{1}{n}}-V_t}{A^{s,\eta}_{t+\frac{1}{n}}-A^{s,\eta}_t+\frac{1}{n}+(V_{t+\frac{1}{n}}-V_t)}$.
Then, for any $(s,\eta)$, $K'^{s,\eta}$ on $[s,+\infty[$ is
 a $\mathbbm{P}^{s,\eta}$-version of $K'$, and 
it constitutes  a density of $dV$ with respect to $dC^{s,\eta}(\omega)$
 on $[s,+\infty[$, for almost all $\omega$. One shows then the  existence of an $\mathbbm{F}^o$-progressively measurable process $k'$ such that for any $(s,\eta)$, $k'$ is a $\mathbbm{P}^{s,\eta}$-version of $K'$ and 
of $K'^{s,\eta}$ on $[s,+\infty[$.

By the considerations after (\ref{EB29}), 
for any $(s,\eta)$,   under $\mathbbm{P}^{s,\eta},$ we can write
$\left\{\begin{array}{rcl}
A^{s,\eta}&=&\int_s^{\cdot\vee s}K^{s,\eta}_rdC^{s,\eta}_r \\
V_{\cdot\vee s} - V_s &=& \int_s^{\cdot\vee s} K'^{s,\eta}_rdC^{s,\eta}_r.
\end{array}\right.$
Now since $dA^{s,\eta}\ll dV$,  we have for $\mathbbm{P}^{s,\eta}$ almost all $\omega$ that the set
 $\{r \in [s,+\infty[:  |K'^{s,\eta}_r(\omega)=0\}$ is negligible with respect to 
 $dV$ so also for $dA^{s,\eta}(\omega)$ and therefore
 we can write
\begin{equation*}
\begin{array}{rcl}
A^{s,\eta}&=&\int_s^{\cdot\vee s}K^{s,\eta}_rdC^{s,\eta}_r \\
&=&\int_s^{\cdot\vee s}\frac{K^{s,\eta}_r}{ K'^{s,\eta}_r}
\mathds{1}_{\{K'^{s,\eta}_r\neq 0\}}K'^{s,\eta}_rdC^{s,\eta}_r 
+ \int_s^{\cdot\vee s}   \mathds{1}_{\{K'^{s,\eta}_r = 0\}}  dA^{s,\eta}_r 
  \\ 
&=&\int_s^{\cdot\vee s}\frac{K^{s,\eta}_r}{K'^{s,\eta}_r}\mathds{1}_{\{K'^{s,\eta}_r\neq 0\}}dV_r,
\end{array}
\end{equation*}
where we use the convention  that
for any two functions $\phi,\psi$ then $\frac{\phi}{\psi}\mathds{1}_{\psi\neq 0}$ is defined by 
$\frac{\phi}{\psi}\mathds{1}_{\{\psi\neq 0\}}(x)=\left\{\begin{array}{l}
\frac{\phi(x)}{\psi(x)}\text{ if } \psi(x)\neq 0\\
0\text{ if } \psi(x)=0.
\end{array}\right.$

We now set 
$h:=\frac{k}{k'}\mathds{1}_{\{k'_r\neq 0\}}$ which is an $\mathbbm{F}^o$-progressively measurable process, and clearly for any $(s,\eta)$, $h$ is a $\mathbbm{P}^{s,\eta}$-version of $H^{s,\eta}:=\frac{K^{s,\eta}}{K'^{s,\eta}}\mathds{1}_{\{K'^{s,\eta}\neq 0\}}$ on $[s,+\infty[$. So by Lemma 5.12 in \cite{paper1preprint},  
$H^{s,\eta}=h$ $dV\otimes d\mathbbm{P}^{s,\eta}$ a.e. on $[s,+\infty[$ 
and finally we have shown that under any $\mathbbm{P}^{s,\eta}$, 
$A^{s,\eta}=\int_s^{\cdot\vee s}h_rdV_r$.
\end{prooff}

\end{appendices}
{\bf ACKNOWLEDGEMENTS.}
The research of the first named author was provided 
by a PhD fellowship (AMX) of the Ecole Polytechnique.
The contribution of the second named author
 was partially supported by the grant 346300 for IMPAN from the Simons Foundation and the matching 2015-2019 Polish MNiSW fund.

\bibliographystyle{plain}
\bibliography{../biblioPhDBarrasso_bib/biblioPhDBarrasso}
%\bibliography{biblioPhDBarrasso}

\end{document}